\documentclass{article}
\usepackage[utf8]{inputenc}
\usepackage{amssymb,amsmath, amsthm, mathabx,mathtools}
\usepackage{amsmath,amsfonts,amssymb,amsthm,epsfig,epstopdf,titling,url,array}
\usepackage{color, enumerate}
\usepackage{comment, authblk}
\usepackage{hyperref}
\hypersetup{
	colorlinks,
	citecolor=blue,
	linkcolor=blue,
	urlcolor=blue}
\usepackage{pgfplots}
\usepackage{bm}
\usepackage{stmaryrd}
\usepackage{graphicx}
\usepackage{caption}
\usepackage{subcaption}
\usepackage{titlesec}
\usepackage{epstopdf}

\usepackage{sectsty}
\usepackage{lipsum}
\usepackage{algorithm}
\usepackage[noend]{algpseudocode}
\usepackage{algpseudocode}
\usepackage{pifont}
\usepackage{tikz}
\usetikzlibrary{positioning}
\usepackage{longtable}
\usepackage{enumitem}
\usetikzlibrary{shapes.geometric, arrows.meta, positioning, calc, snakes}

\titleformat*{\subsection}{\large\bfseries}
\numberwithin{equation}{section}

\pgfplotsset{compat=newest}
\pgfplotsset{plot coordinates/math parser=false}
\newlength\figureheight
\newlength\figurewidth

\newcommand{\Rprod}[2]{\langle #1 , #2 \rangle}

\def\cF{{\mathcal F}}

\newcommand{\zb}{\mathbf{z}}

\usepackage{enumerate}
\usepackage{mathrsfs}
\usepackage{wrapfig}

\usepackage{color}

\usepackage{a4wide}
\usepackage[a4paper, total={6in, 8in}]{geometry}

\numberwithin{equation}{section}

\newcommand{\xb}{\mathbf{x}}
\newcommand{\yb}{\mathbf{y}}

\newcommand{\Dr}{\Delta_\varrho}

\newcommand{\Kb}{\mathbf{K}}

\newcommand{\Db}{\mathbf{D}}
\newcommand{\Lb}{\mathbf{L}}
\newcommand{\Ib}{\mathbf{I}}

\newcommand{\Ka}{\widetilde{\mathbf{K}}}
\newcommand{\Da}{\widetilde{\mathbf{D}}}
\newcommand{\Lc}{\widetilde{\mathbf{L}}}

\newcommand{\beq}{\begin{equation}}
	\newcommand{\bEq}{\end{equation}}

\newcommand{\be}{\begin{equation}}
	\newcommand{\ee}{\end{equation}}

\renewcommand{\div}{\mathop{\mathrm{div}}}

\usepackage{amsmath} 
\usepackage{amssymb}
\usepackage{amsthm}

\newcommand{\dd}{\mathrm d}

\renewcommand{\epsilon}{\varepsilon}
\renewcommand{\leq}{\leqslant}
\renewcommand{\geq}{\geqslant}
\newcommand{\floor}[1] {\lfloor {#1} \rfloor}

\renewcommand{\le}{\leq}
\renewcommand{\ge}{\geq}

\renewcommand{\P}{\mathbb{P}}
\newcommand{\E}{\mathbb{E}}
\newcommand{\R}{\mathbb{R}}
\newcommand{\C}{\mathbb{C}}
\newcommand{\N}{\mathbb{N}}

\newcommand{\norm}[1]{\lVert #1 \rVert}
\newcommand{\normp}[1]{\lVert #1 \rVert_p}
\newcommand{\normb}[1]{\bigl\lVert #1 \bigr\rVert}

\newcommand{\normbb}[1]{\biggl\lVert #1 \biggr\rVert}

\newcommand{\normf}[1]{\lVert #1 \rVert_{\mathcal{F}}}
\newcommand{\normbf}[1]{\biggl\lVert #1 \biggr\rVert_{\mathcal{F}}}

\DeclareMathOperator{\OO}{O}
\DeclareMathOperator{\oo}{o}

\DeclareMathOperator{\OOp}{\OO_{\mathbb{P}}}
\DeclareMathOperator{\oop}{\oo_{\mathbb{P}}}

\DeclareMathOperator{\sL}{\mathcal{L}}

\newcommand{\sLn}{\widetilde{\mathbf{\mathcal{L}}}_n}

\DeclareMathOperator{\sig}{\mathsf{s}}

\theoremstyle{plain} 
\newtheorem{theorem}{Theorem}[section]
\newtheorem*{theorem*}{Theorem}
\newtheorem{lemma}[theorem]{Lemma}

\newtheorem*{lemma*}{Lemma}

\newtheorem*{corollary*}{Corollary}
\newtheorem{proposition}[theorem]{Proposition}
\newtheorem*{proposition*}{Proposition}

\newtheorem{definition}[theorem]{Definition}
\newtheorem*{definition*}{Definition}
\theoremstyle{remark}
\newtheorem{example}[theorem]{Example}
\newtheorem*{example*}{Example}
\newtheorem{remark}[theorem]{Remark}

\newtheorem*{remark*}{Remark}
\newtheorem*{remarks*}{Remarks}

\makeatletter
\def\env@dmatrix{\hskip -\arraycolsep
	\let\@ifnextchar\new@ifnextchar
	\extrarowheight=2ex
	\array{*\c@MaxMatrixCols{>{\displaystyle}c}}}

\date{}

\allowdisplaybreaks

\title{On the edge eigenvalues of  sparse random geometric graphs}
\author[1]{Xiucai Ding \thanks{E-mail: xcading@ucdavis.edu. The author is partially supported by NSF DMS-2306439 and DMS-2515104.}}
\author[1]{Yichen Hu  \thanks{E-mail: ethhu@ucdavis.edu. The author is partially supported by NSF DMS-2515104.}}

\affil[1]{Department of Statistics, University of California, Davis}

\begin{document}
	\maketitle
	\begin{abstract}
In this paper, we study the edge eigenvalues of random geometric graphs (RGGs) generated by multivariate Gaussian samples in the sparse regime under a broad class of distance metrics. Previous work on edge eigenvalues under related setups has relied on methods based on integral operators or the Courant–Fischer min–max principle with interpolation. However, these approaches typically require either a dense regime or sampling distributions that are compactly supported and non-vanishing, and therefore cannot be generalized to our setting. We introduce a two-step smoothing–matching argument. First, we construct a smoothed empirical operator from the given RGG. We then match its edge eigenvalues to those of a continuum limit operator via a counting argument. We show that, after proper normalization, the first few nontrivial edge eigenvalues of the RGG converge with high probability to those of a differential operator, which can be computed explicitly through a simple second-order linear partial differential equation. To the best of our knowledge, these are the first results on the edge eigenvalues of RGGs generated from samples with unbounded support and vanishing density functions.
	\end{abstract}
	
	\section{Introduction}
	Random geometric graphs (RGGs) \cite{duchemin2023random,penrose2003random} have been widely employed in modern data science to uncover latent geometric structures underlying complex datasets, wherein each node is associated with a feature vector and edges are established based on geometric information--especially the pairwise distances between these features--rather than being formed randomly according to prescribed probabilities. For example, RGGs provide a tractable framework for modeling latent spatial structures in network analysis, where edges are formed based on underlying geometric spatial relationships. This framework has found broad applications in areas such as {wireless networks \cite{haenggi2009stochastic}}, social networks \cite{hoff2002latent}, biological networks \cite{betzel2019distance,estrada2016epidemic,higham2008fitting,preciado2009spectral}, community detection \cite{davis2018consistency,peche2020robustness}, {and consensus dynamics \cite{estrada2016consensus}}. As another example, in manifold learning, many graph-based algorithms, such as Diffusion Maps (DM) \cite{coifman2006diffusion}, Laplacian Eigenmaps (LE) \cite{belkin2003laplacian}, and Isomap \cite{tenenbaum2000global}, are widely used to extract the underlying geometric structures. These algorithms typically begin by constructing a random geometric graph in which edges are formed based on pairwise distances between samples, providing a discrete approximation of the underlying manifold.
	
	Broadly speaking, RGGs can be constructed as follows. Suppose we observe a point cloud $\mathcal{X} \subset \mathbb{R}^d$ consisting of $n$ independently and identically distributed $d$-dimensional samples $\xb_1, \xb_2, \dots, \xb_n $, for some fixed $d\in\N^+$. Given a radius $r_n \equiv r_n(d)>0$, a metric $\normp{\cdot}$ for $1\le p\le\infty$ on $\mathbb{R}^d$ and a kernel function $g \equiv g(n,d):[0,\infty)\to\mathbb{R}^+$, to construct the random geometric graph, a common approach is to connect the samples $\xb_i$ and $\xb_j$ whenever $\normp{\xb_i-\xb_j} \leq r_n$. If an edge exists between $\xb_i$ and $\xb_j$, a weight $g(\normp{\xb_i-\xb_j}/r_n)$ is assigned. Consequently, the $n \times n$ affinity matrix $\Kb = (\Kb(i,j))$ is defined as
	\begin{equation}\label{eq_originalK} 
		\Kb(i,j) = \mathbf{1}(\normp{\xb_i-\xb_j}\leq r_n) g(\normp{\xb_i-\xb_j}/r_n), 
	\end{equation}
	where the precise assumptions on $g$ and $r_n$ will be given in Section \ref{sec:mainresult}. Many methodologies and algorithms focus on the Laplacian matrix $\Db-\Kb$ and, in particular, the random-walk Laplacian matrix
	\begin{equation} 
		\Lb_{rw} = \mathbf{I}-\Db^{-1}\Kb, \label{eq:1.1}
	\end{equation} 
	where $\Db = \operatorname{diag}(\Db(i,i))$ with $\Db(i,i) = \sum_{j=1}^n \Kb(i,j)$.
	
	For practical application of RGGs, {many} methodologies and algorithms {rely crucially on the choice of $r_n$.} We illustrate this with manifold learning algorithms. DM employs a complete graph, meaning all $\xb_i$ and $\xb_j$ are connected, which is equivalent to selecting a sufficiently large $r_n$. In contrast, LE and Isomap require {a smaller} $r_n$ for their implementation. Despite the wide application, however, the choice of $r_n$ (equivalently, the construction of the graph) is typically guided by heuristics or conventional empiricism, with no systematic method for selection, except for some discussions under mixture models in \cite{maier2009optimal}. This is primarily due to a lack of theoretical understanding of these methods, particularly regarding the spectral behavior of the matrix $\Lb_{rw}$, especially when $\{\xb_j\}$ has unbounded support. While the spectrum of some other random graphs, e.g., the \emph{Erdős–Rényi graph} (ERG) \cite{Bollobas2001}, has been extensively studied, {far less is known for the spectrum--especially, the edge eigenvalues (the smallest eigenvalues)--of $\Kb$ and $\Lb_{rw}$ of the RGG, particularly when the sampling distribution has unbounded support.}
	
	Motivated by these challenges, this paper studies the edge eigenvalues of $\Lb_{rw}$ for general classes of kernel functions $g$ and metrics $\normp{\cdot}$ for $1\le p\le\infty$ in the sparse regime, where $r_n$ is chosen so that the average degree is $\mathrm{o}(n)$. Unlike most of the existing literature like \cite{adhikari2022spectrum,garcia2020error,trillos2021geometric,li2025central}, we do not require compact support or densities bounded from above and below. {Concretely, we focus on samples with multivariate Gaussian distribution, one of the most common distributions with unbounded support.} In the following sections, we will summarize related results in Section \ref{intro_related} and provide an overview of our findings and contributions in Section \ref{intro_new}.
	
	\subsection{Some related results on RGGs and open questions}\label{intro_related}
	RGGs have been studied from the perspective of graph theory, with research focusing on various aspects such as component counts, typical vertex degrees, clique detection, component connectivity, and their differences from classical Erdős–Rényi graphs. For a comprehensive overview, we refer readers to the monograph \cite{penrose2003random} and the recent review paper \cite{duchemin2023random}. 
	
	Regarding the edge eigenvalues of RGGs through the matrices \eqref{eq_originalK} and \eqref{eq:1.1} with fixed $d$, most work has been conducted under compact support and bounded density assumptions. With uniform samples on $[-1,1]^d$ and fixed $r_n$, the behavior of the edge eigenvalues of $\Lb_{rw}$ was studied in  \cite{adhikari2022spectrum} for $\norm{\cdot}_\infty$ metric. When $r_n=\oo(1)$, Cheeger’s inequality (e.g., \cite{chung1997spectral})
	implies that the second edge eigenvalue tends to 0, which is very different from the constant $r_n$ case. To obtain nontrivial limits, \cite{garcia2020error,trillos2021geometric}
	analyzed a rescaled $\Lb_{rw}$ with factor of order $r_n^{-2}$ and proved convergence
	(with rates) to the edge eigenvalues of a differential operator, for compactly supported
	densities and $\norm{\cdot}_2$. Subsequently, \cite{li2025central} derived asymptotic
	distributions for these edge eigenvalues. This line of work nearly completes the discussion of the compact
	support case in the sparse regime.
	On the other hand, the empirical spectral density (ESD) of $\Lb_{rw}$ has also been studied in the sparse regime. 
	For uniform samples
	on $[0,1]^d$ with $g\equiv1$, the ESD of $\Lb_{rw}$ converges to that of  a deterministic graph
	under $\norm{\cdot}_2$ \cite{rai2007spectrum}; \cite{hamidouche2020spectral} extended
	this to any $\normp{\cdot}$, $1\le p\le\infty$, and further showed its convergence to a Dirac mass at $1$.
	Another different but closely related line of work studies affinity matrices \eqref{eq_originalK} constructed as $\Kb(i,j)=g(\normp{\xb_i-\xb_j}/r_n)$, without the hard-radius indicator. 
	With constant $r_n$ and compactly supported sampling densities bounded
	above and below, \cite{von2008consistency} established spectral convergence of the edge eigenvalues/eigenvectors of $\Kb$ and $\Lb_{rw}$ to those of associated integral
	operators. When $r_n=\oo(1)$, the samples are uniform on $[0,1]^d$, and
	$g(t)=\exp(-t^2/4)$, \cite{belkin2006convergence} proved convergence of the (scaled)
	edge eigenvalues of $\Lb_{rw}$ to those of a differential operator, and
	\cite{cheng2022eigen} later obtained their convergence rates. However, all the above methods rely on compact domains and near-uniform distribution properties, together with specific metrics and kernels in the construction of (\ref{eq_originalK}), and thus are not applicable to the Gaussian case (with unbounded support) under our general framework.

	Although it is beyond the scope of our paper, we note that RGGs have also been studied to some extent in the regime where $d$ diverges. For example, for RGGs on the sphere
	with uniform sampling on $\mathbb{S}^{d-1}$ and the $\norm{\cdot}_2$ metric,
	\cite{brennan2020phase,journals/rsa/BubeckDER16,devroye2011high,thesisone,liu2022testing,lu2022equivalence}
	investigated distinguishability from ERGs, and
	\cite{cao2025spectra,dubova2023universality,lu2022equivalence} analyzed the ESDs of their adjacency matrices.
	Another prominent research direction involves analyzing $\Kb$ and $\Lb_{rw}$ when the graph is complete, i.e., $\Kb(i,j)=g(\normp{\xb_i-\xb_j}/r_n)$, see, for example, \cite{MR2462254,Braun2006AccurateEB,cheng2013, DW1,ding2022impact,do2013spectrum,Fan2018,MR3338323,el2010spectrum,koltchinskii2000random}. Nevertheless, the behavior of the edge eigenvalues in those settings also remains largely unknown.
	
	While these results offer valuable insights into certain matrices related to the RGGs, many important questions remain open, particularly with regard to practical applications. For instance, the results regarding the affinity matrix with general kernel functions $g$ and metrics $\normp{\cdot}$ for $1\le p\le\infty$ in (\ref{eq_originalK}), as well as for ${\xb_i}$ with general distributions, especially for those with unbounded support, for example, Gaussian distributions, remain open. Additionally, understanding the convergence of the edge eigenvalues of the aforementioned matrices is crucial. Addressing these questions not only resolves several mathematically intriguing problems but also enables the development of an efficient algorithm for selecting $r_n$ in algorithms like Laplacian Eigenmaps. Furthermore, it allows us to propose novel statistical methods to detect the presence of geometry and signals in graphs. 
	
	\subsection{Our contributions and novelties}\label{intro_new}
	In this paper, we study the edge eigenvalues of random geometric graph generated by $d$-dimensional $\operatorname{i.i.d.}$ multivariate Gaussian samples in the sparse regime. Unlike the bounded support and regular density assumptions commonly imposed in the literature—which typically ensure that the resulting graphs have only one connected component—the unbounded support and vanishing density of Gaussian samples will lead to multiple connected components. Consequently, the matrix (\ref{eq:1.1}) can exhibit several trivially small eigenvalues. In this context, by edge eigenvalues we refer to the non-trivial ones. In Theorem \ref{theo:normald=1}, we show that, after proper normalization, the first few non-trivial edge eigenvalues of (\ref{eq:1.1}) converge with high probability to deterministic limits, which are characterized by the eigenvalues of the continuum limit operator defined in \eqref{eq:2.1.1}. Furthermore, in Proposition \ref{prop:d>1}, we show that these limits can be computed explicitly by solving a second-order linear partial differential equation (PDE) (cf. (\ref{eq:4.2})).

	Different techniques have been employed to study problems with related setups in the literature. However, none of these methods can be directly applied or readily generalized to address our problem. In the dense regime, \cite{adhikari2022spectrum} use integral operators to study the limiting behavior of the edge eigenvalues. However, in our sparse regime, since $r_n$ tends to zero, we cannot construct such valid integral operators using their approaches. On the other hand, in the sparse regime with compact support and bounded densities, Courant–Fischer min–max principle is utilized to link the eigenvalues of RGG to those of some differential operators (e.g., \cite{calder2022improved,cheng2022eigen,garcia2020error,trillos2021geometric}). This argument relies on interpolations that crucially require  bounded eigenfunctions or bounded support, neither of which holds in our Gaussian setup. Consequently, it is difficult to construct an appropriate interpolation to carry out the min-max framework in our setup, if not impossible.
	
	Unlike the dense regime studied through the lens of integral operators in \cite{adhikari2022spectrum}, we investigate the eigenvalues of the RGG via a differential operator that emerges as the continuum limit of the graph. Rather than employing the min–max principle and interpolation as in \cite{calder2022improved,cheng2022eigen,garcia2020error,trillos2021geometric}, to link the eigenvalues of the random geometric graph to the limit operator, we propose a two-step smoothing-matching approach. Specifically, we construct a smoothed empirical operator based on the random geometric graph and then match the first few non-trivial edge eigenvalues of this empirical operator to those of the limit operator. A crucial aspect of our approach is to show the closeness between the empirical operator and the limit operator on some carefully chosen functions. This is achieved by detailed analysis of an orthonormal decomposition the limit operator based on the physicist's Hermite polynomials \cite{szeg1939orthogonal} and a truncation argument.
	
	While this Hermite-based decomposition is specific to Gaussian distribution, our approach is adaptable to other distributions: operators associated with other sampling distributions that admit
	orthonormal bases can, in principle, be treated by replacing Hermite polynomials with the corresponding basis to carry out our arguments. On the other hand, we point out that in bounded-support domains, the truncation argument becomes unnecessary and our approach applies with minor modification. As future work, we plan to explore the statistical applications discussed at the end of Section \ref{sec:mainresult} and extend our approach to high-dimensional regimes.
	
	The paper is organized as follows. Section \ref{sec:2} presents our main results and outlines the proof strategies. In Section \ref{sec_3}, we analyze and compute the eigenvalues and eigenfunctions of the limiting differential operators. Section \ref{sec:mainproof} contains the proofs of the main results. Additional details are provided in the appendices.
	
	\vspace{6pt}
	\noindent{\bf Conventions.} We write $C$, $C_0$, etc.\ for generic large positive constants whose values may change from line to line, and $\epsilon$, $\delta$, etc.\ for generic small positive constants. We use $\floor{a}$ for the greatest integer less than or equal to $a$. When a constant depends on a parameter $a$, we indicate this by $C(a)$ or $C_a$. For sequences $\{a_n\}$ and $\{b_n\}$ depending on $n$, $a_n=\OO(b_n)$ means $|a_n|\le C|b_n|$ for some constant $C>0$, while $a_n=\oo(b_n)$ (equivalently $b_n\gg a_n$) means $|a_n|\le c_n|b_n|$ for some positive sequence $c_n\downarrow0$ as $n\to\infty$. We also use $a_n\lesssim b_n$ to mean $a_n=\OO(b_n)$, $a_n\gtrsim b_n$ to mean $b_n=\OO(a_n)$, and $a_n\asymp b_n$ to mean both $a_n=\OO(b_n)$ and $b_n=\OO(a_n)$. The probabilistic analogues are $a_n=\OOp(b_n)$, meaning $a_n=\OO(b_n)$ with probability at least $1-\oo(1)$, and $a_n=\oop(b_n)$, meaning $a_n=\oo(b_n)$ with probability at least $1-\oo(1)$. We use bold lowercase such as $\xb,\yb$ for vectors and bold uppercase such as $\bf{D},\bf{K}$ for matrices. Calligraphic letters such as $\mathcal{L},\cF$ denote operators or spaces, and plain letters such as $f,g$ denote generic functions. For norms, $\normp{\xb}=\big(\sum_{j=1}^d |x_j|^p\big)^{1/p}$ denotes the $\ell_p$-norm of $\xb=(x_1,\cdots,x_d)^{\top}\in\R^d$, for $1\le p\le\infty$. For a function $f$ on a space $\cF$ with a weight $w(\xb)$, we write $\norm{f}_{\cF}=\big(\int f^2(\xb)\,w(\xb)\dd\xb\big)^{1/2}$. We denote by $\norm{\mathcal{L}}$ and $\norm{\bf{A}}$ the operator norms of an operator $\mathcal{L}$ and a matrix $\bf{A}$, respectively.

	\section{Main results and proof strategy}\label{sec:2}
	\subsection{Main results}\label{sec:mainresult}
	Let $\mathcal{X}=\{\xb_1,\cdots,\xb_n\}\subset\R^d$ consist of i.i.d. samples from the Gaussian distribution $\mathcal{N}_d(\bm{0},\Sigma)$, where
	\begin{align}
		\Sigma=\operatorname{diag}(\sigma_1^2,\cdots,\sigma_d^2),\quad 0<\sigma_i<\infty,\quad 1\le i\le d.\label{eq:2.1.0}
	\end{align}
	Recall \eqref{eq_originalK}. For notational simplicity, we denote $h(x)=\bm{1}_{[0,1]}(x)g(x)$. Moreover, for $l\ge0$, we define
	\begin{align}
		m_l\coloneq\int_{\R^d} |u_i|^l h(\normp{\bm{u}})\dd\bm{u},\quad \bm{u}=(u_1,\cdots,u_d)^\top,\label{eq:mj}
	\end{align}
	for any $i=1,\cdots,d$. Note that by symmetry, $m_l$ is independent of the choice of $i$. Together with \eqref{eq:1.1}, we define the scaled random-walk Laplacian matrix
	\begin{align}
		\Lb\coloneq\frac{2m_0}{m_2\,r_n^2}\Lb_{rw}.\label{eq:2.1.2}
	\end{align}
	Throughout the paper, we list the eigenvalues of $\Lb$ in non‑decreasing order,
	$$\lambda_1(\Lb)\le\lambda_2(\Lb)\le\cdots\le\lambda_n(\Lb).$$
	We point it out that according to classic spectral graph theory (see, e.g., \cite[Chapter 1]{chung1997spectral}), $\lambda_1(\Lb)=0$.
	
	For $\xb=(x_1,\cdots,x_d)^\top\in\R^d$ and $\{\sigma_i^2\}_{i=1}^d$ in \eqref{eq:2.1.0}, let
	\begin{align}
		\varrho(\xb)=\exp\left(-\sum_{i=1}^d\frac{1}{2\sigma_i^2}x_i^2\right).\label{eq:2.1.4}
	\end{align}
	Based on $\varrho(\xb)$, we define a Hilbert space
	\begin{align}
		\cF=\left\{f:\R^d\rightarrow\R\biggm|\int_{\R^d}f^2(\xb)\varrho^2(\xb)\dd\xb<\infty\right\},\label{eq:2.1.3}
	\end{align}
	equipped with the inner product
	\begin{align}
		\Rprod{f_1}{f_2}=\int_{\R^d}f_1(\xb)f_2(\xb)\varrho^2(\xb)\dd \xb,\quad f_1,f_2\in\cF,\notag
	\end{align}
	and the induced norm
	\begin{align}
		\normf{f}\coloneq\sqrt{\Rprod{f}{f}},\quad f\in\cF.\label{eq:fnorm}
	\end{align}
	With the above notations in place, we define the weighted Laplace–Beltrami operator on $(\cF,\Rprod{\cdot}{\cdot})$ by
	\begin{align}
		\Dr f=-\frac{1}{\varrho^2}\,\div\bigl(\varrho^2\nabla f\bigr).\label{eq:2.1.1}
	\end{align}
	As we will show in Theorem \ref{theo:eigen}, the spectrum of $\Dr$ is discrete and bounded from below on $(\cF,\Rprod{\cdot}{\cdot})$. We list its eigenvalues as $\mu_1(\Dr)\le\mu_2(\Dr)\le\cdots$.
	
	\begin{remark}
		The density function of $\mathcal{N}_d(\bm{0},\Sigma)$ with $\Sigma$ in \eqref{eq:2.1.0} is
		$$\varrho^*(\xb)=(2\pi)^{-d/2}\prod_{i=1}^d\sigma_i^{-1}\exp\left(-\sum_{i=1}^d\frac{1}{2\sigma_i^2}x_i^2\right)=\left[(2\pi)^{-d/2}\prod_{i=1}^d\sigma_i^{-1}\right]\varrho(\xb),$$
		which differs from $\varrho(\xb)$ in \eqref{eq:2.1.4} up to some constant.
		One can also construct another Hilbert space $(\cF^*,\Rprod{\cdot}{\cdot}_*)$ based on $\varrho^*(\xb)$,
		where 
		$$\cF^*\coloneq\left\{f^*:\R^d\rightarrow\R\biggm|\int_{\R^d}(f^*(\xb))^2(\varrho^*(\xb))^2\dd\xb<\infty\right\},$$
		and 
		$$\Rprod{f_1^*}{f_2^*}_*\coloneq\int_{\R^d}f_1^*(\xb)f_2^*(\xb)(\varrho^*(\xb))^2\dd \xb,\quad f_1^*,f_2^*\in\cF^*.$$
		Note that $(\cF^*,\Rprod{\cdot}{\cdot}_*)$ is essentially the same as $(\cF,\Rprod{\cdot}{\cdot})$ up to some constant in the inner product. Since 
		\begin{align*}
			\Delta_{\varrho^*} f\coloneq-\frac{1}{(\varrho^*)^2}\,\div\bigl((\varrho^*)^2\nabla f\bigr)=-\frac{1}{\varrho^2}\,\div\bigl(\varrho^2\nabla f\bigr)=\Dr f,
		\end{align*}
		we have that $\varrho(\xb)$ and $\varrho^*(\xb)$ induce the same weighted Laplace-Beltrami operator. Combining the above discussion, without loss of generality, we can utilize $\varrho(\xb)$ in \eqref{eq:2.1.4} to construct the Hilbert space and the weighted Laplace-Beltrami operator related to $\mathcal{N}_d(\bm{0},\Sigma)$.
	\end{remark}
	We are now ready to state the main results. For some small constant $\delta>0$, set
	\begin{align}
		K_0\equiv K_0(\delta,\Lb)\coloneq\max_j\{j:\lambda_j(\Lb)\le\delta\}.\label{eq:2.1}
	\end{align}
	\begin{theorem}
		\label{theo:normald=1}
		Fix some small constant $0<\epsilon<\frac{1}{2(d+4)}$ and assume 
		\begin{align}
			n^{-\tfrac{1}{d+4}+\epsilon}\ll r_n\ll n^{-\epsilon}.\label{eq:2.2}
		\end{align}
		For some fixed integer $d\in\N^+$, suppose $\xb_j\stackrel{\mathrm{i.i.d.}}{\sim}\mathcal{N}_d(\bm{0},\Sigma)$ for $j=1,\cdots,n$, with $\Sigma$ as in \eqref{eq:2.1.0}. Moreover, for $\delta$ in \eqref{eq:2.1}, we assume that
		\begin{align}
			\delta<\min\{\sigma_1^{-2},\cdots,\sigma_d^{-2}\}.\label{eq:delta}
		\end{align}
		In addition, for $h(y)=\bm{1}_{[0,1]}(y)g(y)$, $y\ge0$, we assume that $g\in C^2([0,\infty))$ and $g(\cdot)$ is bounded from above and below on $[0,1]$ by some universal positive constants. Let $\Lb$ be defined in \eqref{eq:2.1.2} via \eqref{eq:1.1} and \eqref{eq_originalK} with $1\le p\le\infty$. Let $M$ be any fixed positive integer and $K_0$ be defined in \eqref{eq:2.1}. Then, for sufficiently large $n$, with probability $1-\oo(1)$, $K_0\le n^{1-\tfrac{2}{(d+2)^2+\eta}}$, for some sufficiently small constant $\eta>0$. Moreover,
		\begin{align}
			\lambda_{K_0+k}(\Lb)=\mu_{k+1}(\Dr)+\oo(1),\quad k=1,\cdots,M.\label{eq:main}
		\end{align}
	\end{theorem}
	
	\begin{remark}
		Theorem \ref{theo:normald=1} shows that the first few non-trivial eigenvalues of the scaled random‑walk Laplacian matrix $\Lb$ converge, with probability $1-\oo(1)$, to those of the weighted Laplace–Beltrami operator $\Dr$ on $(\cF,\Rprod{\cdot}{\cdot})$. We demonstrate the empirical accuracy of the result in Figure \ref{fig:sub11}.  Moreover, to make \eqref{eq:main} valid, we only need that $K_0=\oo(n)$. In this sense, $K_0\le n^{1-\tfrac{2}{(d+2)^2+\eta}}$ is sufficient for our purpose. Nevertheless, we believe that such a bound can be improved, especially if \eqref{eq:2.2} can be relaxed. Since this is out of the scope of our current work, we will pursue this direction in the future.

		We will show in Section \ref{sec:prop} that $m_0$ and $m_2$ defined in \eqref{eq:mj} are bounded from above and below by some positive constants. Consequently, Theorem \ref{theo:normald=1} implies that the first few eigenvalues of the unscaled random-walk Laplacian matrix $\Lb_{rw}$ in \eqref{eq:1.1} are $\oo(1)$, which differs significantly from the results for dense RGGs \cite{adhikari2022spectrum}. We now provide some discussion on our assumptions.
		
		First, $g(\cdot)$ can be relaxed into more general functions. In fact, as $\normp{\xb_i-\xb_j}/r_n\ge 0$ and whenever $\normp{\xb_i-\xb_j}/r_n>1$, $h(\normp{\xb_i-\xb_j}/r_n)=0$, only the behavior of $g(\cdot)$ on $[0,1]$ is relevant in our analysis. Therefore, we can allow $g(\cdot)$ to be defined on $\R$. We can also weaken the differentiability assumption of $g(\cdot)$ outside $[0,1]$.
		
		Second, the condition \eqref{eq:2.2} balances a tradeoff for the radius $r_n$. On the one hand, if $r_n$ is too large, each $\xb_j$ connects too many neighbors, which obscures the local geometry around $\xb_j$. On the other hand, if $r_n$ is too small, it is difficult to connect sufficient neighbors to capture enough information of the data. This tradeoff in the choice of $r_n$ can be empirically visualized in Figure \ref{fig:sub12}. Moreover, due to the unbounded support of Gaussian distribution, we require a larger lower bound of $r_n$ than those of the existing literature like \cite{calder2022improved,cheng2022eigen,garcia2020error} assuming bounded support. For example, when $\xb_j$'s are i.i.d. sampled from a distribution with bounded support, \cite{cheng2022eigen} only requires $r_n\gg\left(\log n/n\right)^{1/d}$. Technically, the larger lower bound on $r_n$ in our setup required by \eqref{eq:2.2} arises from the truncation argument we employ in Section \ref{sec:prop}. We believe that, even if it is unlikely to recover the above lower bound for $r_n$ when $\xb_j$'s have unbounded support, \eqref{eq:2.2} can be weakened to some extent with substantial technical effort. We will explore this aspect in future work.
		
		Third, in this paper, we focus on Gaussian samples with diagonal population covariance matrices. On the one hand, for the ease of statement, we assume that $\Sigma$ is diagonal so that the induced weighted Laplace-Beltrami operator $\Dr$ can be studied via $d$ operators acting on Hilbert spaces of one-dimensional functions; see Appendix \ref{appendix:pde} for more detailed discussion. Dealing with more general covariance matrices would require a more involved analysis of the weighted Laplace-Beltrami operator. This is left for future work. On the other hand, we focus on the Gaussian distribution due to the connection between its density function and the Hermite polynomials, which are actually the eigenfunctions of $\Dr$. We believe our results could be generalized to sub-Gaussian samples with additional technical work. We will pursue this direction in the future.
	\end{remark}
	
	In the following proposition, we further study the spectrum of $\Dr$ so that the limits of the non-trivial edge eigenvalues of $\Lb$ become explicit. For all $k_1,\cdots,k_d\in\N^+$, we denote a non-negative sequence as follows
	$$\left\{\sum_{i=1}^d 2(k_i-1)/\sigma_i^2\right\},$$
	where some of the values may repeat and multiplicities may appear. For simplicity, we list them in non-decreasing order and denote them as
	\begin{align}
		\mathsf{a}_1\le\mathsf{a}_2\le\cdots.\label{eq:2.1.5}
	\end{align}
	
	\begin{proposition}
		\label{prop:d>1}
		Under the assumptions of Theorem \ref{theo:normald=1}, we have that $\mu_j(\Dr)=\mathsf{a}_j$, for all $j\in\N^+$. Consequently, for sufficiently large $n$ and any fixed positive integer $M$, with probability $1-\oo(1)$,
		$$\lambda_{K_0+k}(\Lb)=\mathsf{a}_{k+1}+\oo(1),\quad k=1,\cdots,M.$$
	\end{proposition}
	
	For illustration of Proposition \ref{prop:d>1}, we introduce several examples on how to calculate \eqref{eq:2.1.5}.
	\begin{example}
		\label{example:eg1}
		When $d=1$ and $\sigma_1^2\equiv\sigma^2$, all the eigenvalues of $\Dr$ are simple and $\mathsf{a}_j=2(j-1)/\sigma^2$ for $j\in\N^+$. Therefore, we have that for sufficiently large $n$ and any fixed integer $M$, with probability $1-\oo(1)$,
		$$\lambda_{K_0+k}(\Lb)=2(k-1)/\sigma^2+\oo(1),\quad k=1,\cdots,M.$$
	\end{example}
	\begin{example}
		\label{example:eg2}
		When $d=2$ and $\sigma_1^2=\sigma_2^2=\sigma^2$, the spectrum of $\Dr$ consists of values $2(s-1)/\sigma^2$ with multiplicities $s$ for $s\in\N^+$, i.e.,
		$$\mathsf{a}_1=0,\quad \mathsf{a}_2=\mathsf{a}_3=2/\sigma^2,\quad \mathsf{a}_4=\mathsf{a}_5=\mathsf{a}_6=4/\sigma^2,\quad\cdots.$$
		Therefore, we have that for sufficiently large $n$ and any fixed integer $M$, with probability $1-\oo(1)$,
		$$\lambda_{K_0+k}(\Lb)=\frac{2}{\sigma^2}\min\left\{l:\sum_{s=1}^ls\ge k+1,l\in\N\right\}-\frac{2}{\sigma^2}+\oo(1),\quad k=1,\cdots,M.$$
	\end{example}
	\begin{example}
		\label{example:eg3}
		When $d=2$, $\sigma_1^2=\sigma^2$ and $\sigma_2^2=2\sigma^2$, the spectrum of $\Dr$ consists of values $s/\sigma^2$ for $s\in\N$ with multiplicities $\floor{s/2}+1$, where $\floor{s/2}$ is the largest integer smaller than or equal to $s/2$. That is to say,
		$$\mathsf{a}_1=0,\quad \mathsf{a}_2=1/\sigma^2,\quad \mathsf{a}_3=\mathsf{a}_4=2/\sigma^2,\quad \mathsf{a}_5=\mathsf{a}_6=3/\sigma^2,\quad\cdots.$$
		Therefore, we have that for sufficiently large $n$ and any fixed integer $M$, with probability $1-\oo(1)$,
		$$\lambda_{K_0+k}(\Lb)=\frac{1}{\sigma^2}\min\left\{l:\sum_{s=0}^l(\floor{s/2}+1)\ge k+1,l\in\N\right\}+\oo(1),\quad k=1,\cdots,M.$$
	\end{example}

	Before concluding this section, we present simulations (Figures~\ref{figure1} and \ref{figure2}) to illustrate the accuracy and potential applications of Theorem \ref{theo:normald=1} and Proposition \ref{prop:d>1}. For convenience, we focus on $d=1$, as in Example \ref{example:eg1}, and construct the RGGs with $g \equiv 1$ in (\ref{eq_originalK}). In Figure \ref{figure1}, we assess the accuracy of \eqref{eq:main} and the effect of $r_n$ on our results. More specifically, in Figure \ref{fig:sub11}, we illustrate the empirical convergence of the first few non-trivial edge eigenvalues of $\Lb$ to their theoretical limits from Example \ref{example:eg1}, which verifies our results in \eqref{eq:main}.
	In Figure \ref{fig:sub12}, we report the empirical mean of the relative error: $\sum_{k=1}^6 |\lambda_{K_0+k}(\Lb)-\mu_{k+1}(\Dr)|/(6 \mu_{k+1}(\Dr)),$ 
	as $r_n$ varies. The curve exhibits a U-shaped with a flat bottom: either too small or too large $r_n$ induce large deviations from the limits, whereas mid-range $r_n$ yields small and stable errors, which is consistent with condition \eqref{eq:2.2}.

	\begin{figure}[!ht] 
		\centering 
		
		\begin{subfigure}[b]{0.48\textwidth}
			\centering
			\includegraphics[width=8cm, height=6cm]{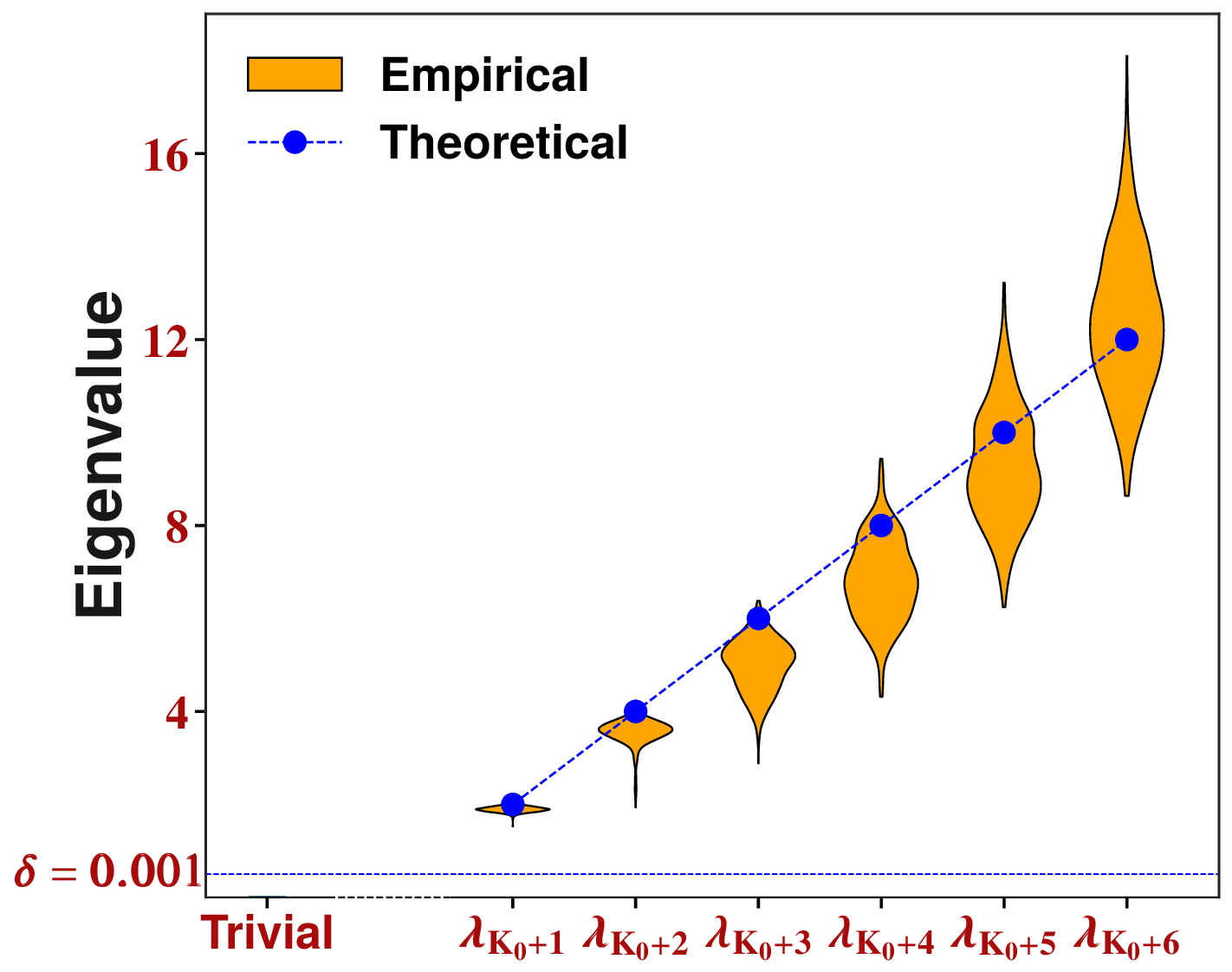}
			\caption{Empirical accuracy}
			\label{fig:sub11}
		\end{subfigure}
		\hfill 
		\begin{subfigure}[b]{0.49\textwidth}
			\centering
			\includegraphics[width=8cm, height=6cm]{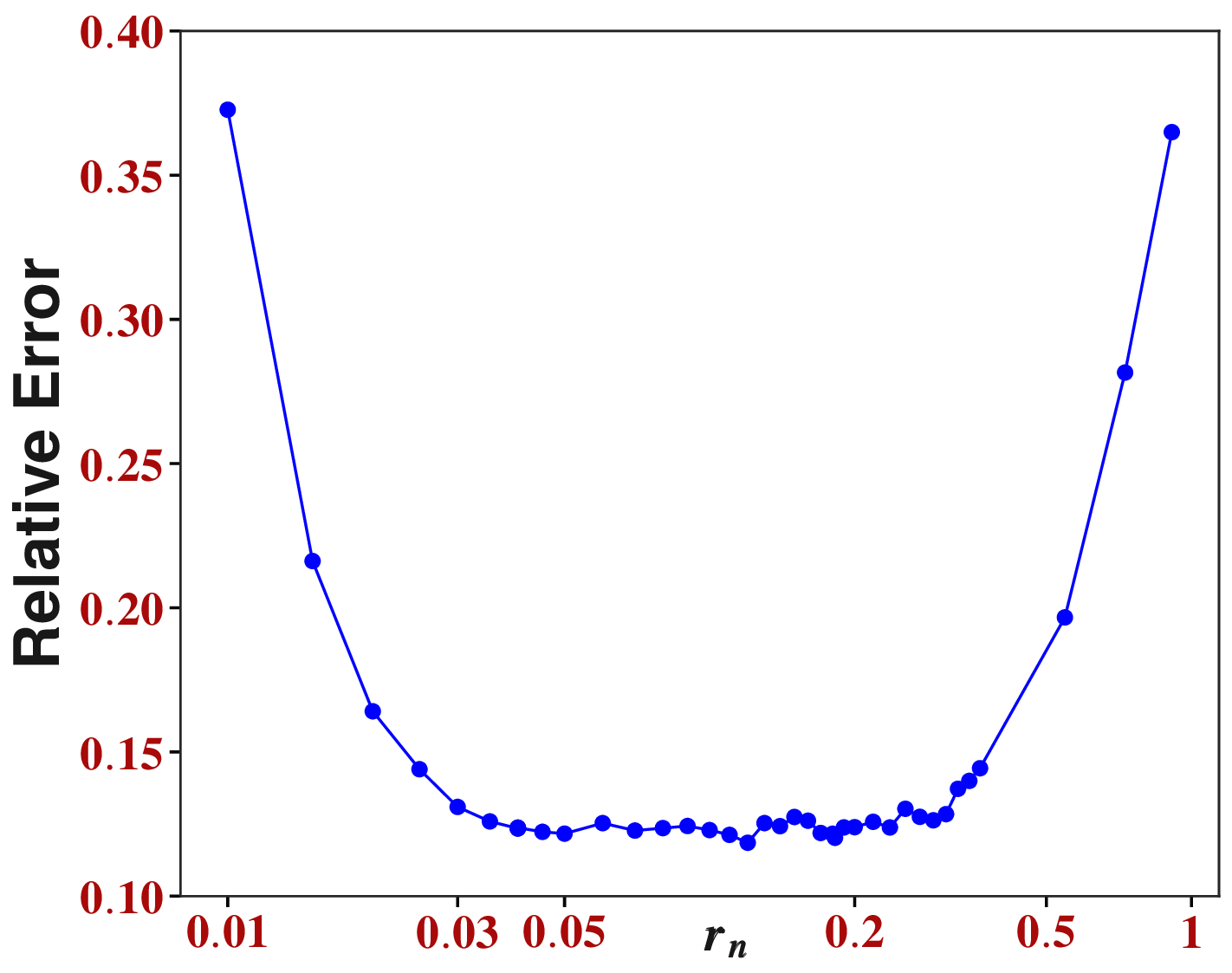}
			\caption{Impact of $r_n$}
			\label{fig:sub12}
		\end{subfigure}
		\caption{Empirical accuracy and the impact of $r_n$. Here we use the setup of Example \ref{example:eg1} with $\sigma=1$, $n=5,000$ and $ g\equiv 1$ in (\ref{eq_originalK}). In  Figure \ref{fig:sub11}, we choose $r_n=0.05$ and illustrate the empirical convergence of the first six non-trivial edge eigenvalues using violin plots; the theoretical limits $\mu_{2}(\Dr),\cdots,\mu_7(\Dr)$ from Example \ref{example:eg1} are marked by the blue points. In  Figure \ref{fig:sub12}, we compare the relative errors $\sum_{k=1}^6 |\lambda_{K_0+k}(\Lb)-\mu_{k+1}(\Dr)|/(6\mu_{k+1}(\Dr))$ under different chioces of $r_n$. All results are based on 1,000 repetitions.}
		\label{figure1}
	\end{figure}
	
	Then we discuss some potential statistical applications of our results. In manifold learning, the clean signals $\{\mathbf{s}_j\}$ are often assumed to be sampled from some compact manifold with regular sampling density. Moreover, the algorithms usually assume the underlying dataset is clean, an assumption which is rarely satisfied in practice. To check or test this, a popular model in statistics and signal processing is the signal-plus-noise model, where the signals are corrupted by Gaussian noise: that is, $\yb_j=\mathbf{s}_j+\xb_j$.
	Figure \ref{figure2} shows that, by combining our results in Theorem \ref{theo:normald=1}, we can potentially study the detection of complicated signals from the signal-plus-noise model due to the very different behaviors of the edge eigenvalues under the signal setup and the Gaussian noise setup. In particular, in Figure \ref{fig:sub21}, we compare the first few non-trivial edge eigenvalues of RGGs constructed from two compactly supported signals with those of standard Gaussian noise. The signal cases exhibit very different behavior compared to the noise.
	Moreover, based on this observation, we can construct some statistics to detect whether the signals are clean. For example, we can use
	\begin{equation}
		\mathbb{T}\coloneq\sum_{k=1}^{M}|\lambda_{K_0+k}(\Lb^{\yb})-\nu_{k+1}|/M,\label{eq:t}
	\end{equation}
	where $\Lb^{\yb}$ is the counterpart of $\Lb$ in \eqref{eq:2.1.2} built from ${\yb_j}$, and ${\nu_{k+1}}$ are the eigenvalues of an operator associated with the signal distribution, which can be computed directly (see, e.g.,~\cite{garcia2020error}).
	
	When the noise level of Gaussian noise $\sigma$ is small so that $\yb_j \approx \mathbf{s}_j$, we expect $\mathbb{T}$ to be small. When $\sigma$ is large and $\yb_j\approx\xb_j$, we expect that the limit of $\lambda_{K_0+k}(\Lb^{\yb})$ can be described by Example \ref{example:eg1}. This phenomenon is illustrated in Figure \ref{fig:sub22}, suggesting a potential use of $\mathbb{T}$ for signal detection. Since the rigorous justification of this application is beyond the scope of the present paper, we will pursue this direction in future work.	
	\begin{figure}[!ht] 
		\centering 
		
		\begin{subfigure}[b]{0.48\textwidth}
			\centering
			\includegraphics[width=8cm, height=6cm]{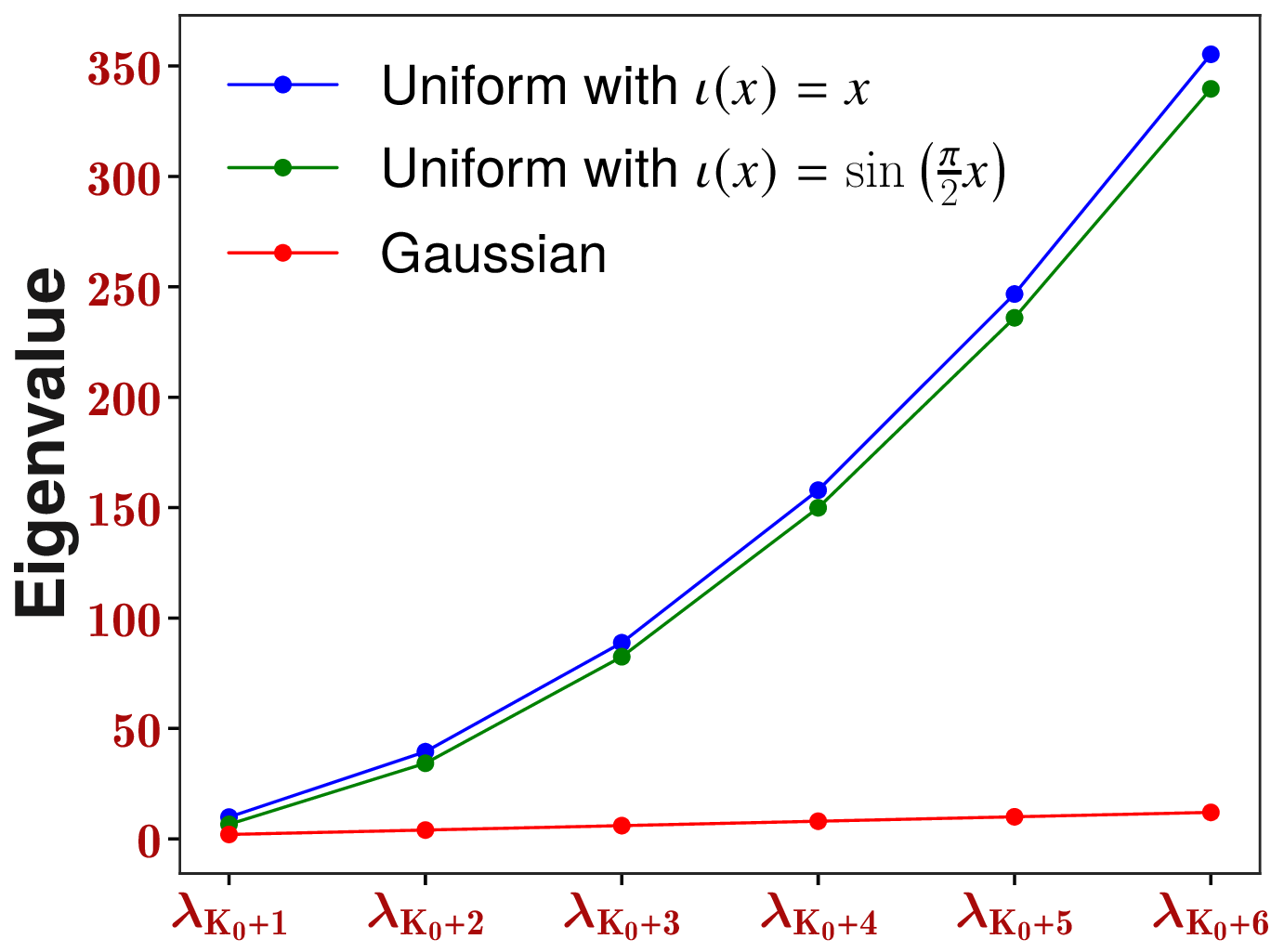}
			\caption{Behavior of eigenvalues}
			\label{fig:sub21}
		\end{subfigure}
		\hfill 
		\begin{subfigure}[b]{0.455\textwidth}
			\centering
			\includegraphics[width=8cm, height=6cm]{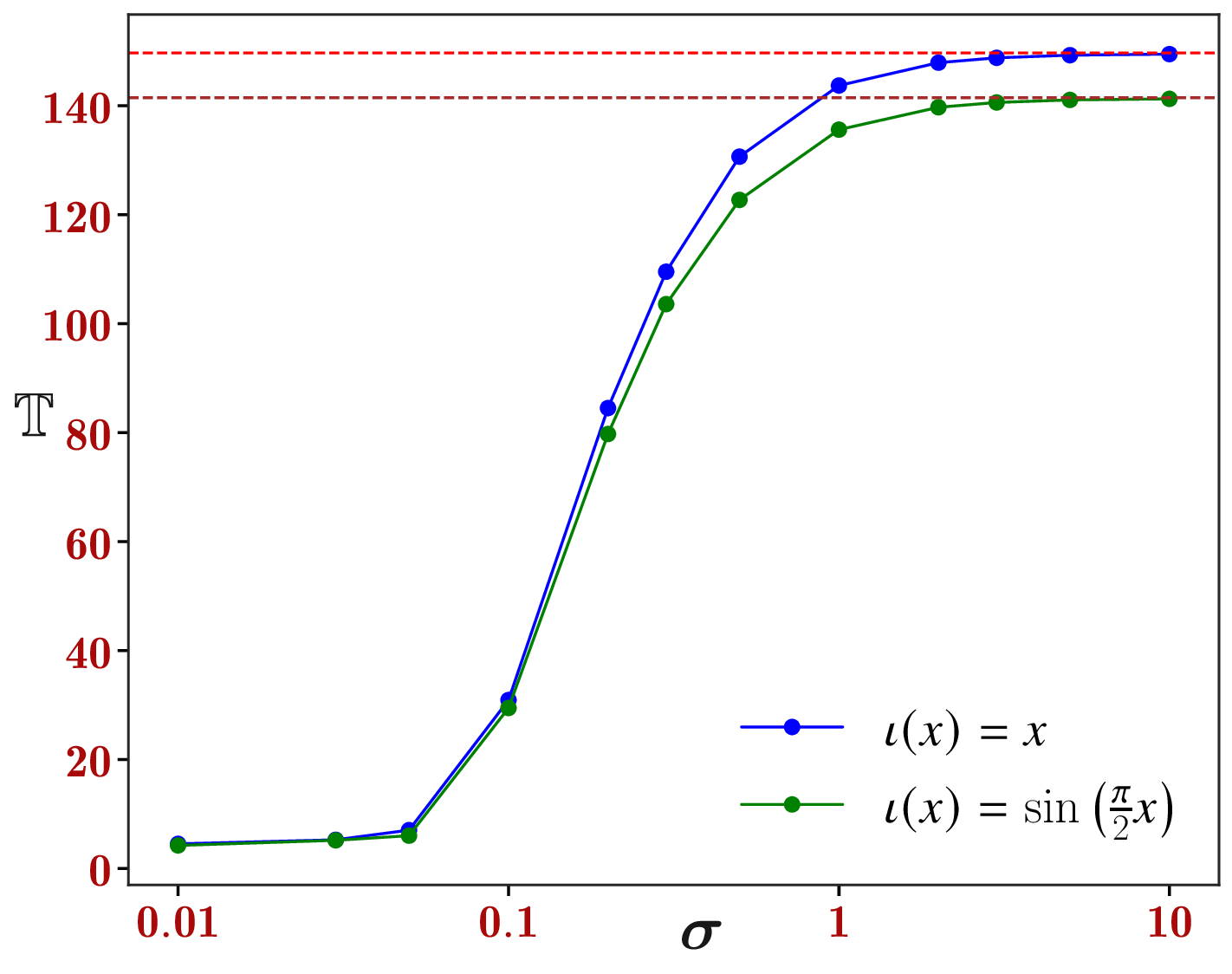}
			\caption{Behavior of $\mathbb{T}$}
			\label{fig:sub22}
		\end{subfigure}
		\caption{{Behavior of edge eigenvalues across different setups. Here we use the setup of Example \ref{example:eg1} with $n=5{,}000$, $g\equiv 1$ and $r_n=0.05$ in (\ref{eq_originalK}). Figure \ref{fig:sub21} shows their theoretical limits for RGGs built from (i) clean signal $\mathbf{s}_j=\iota(\zb_j)$ with $\iota(x)=x$ and $\zb_j\sim\mathrm{Unif}[0,1]$; (ii) clean signal $\mathbf{s}_j=\iota(\zb_j)$ with $\iota(x)=\sin(\pi x/2)$ and $\zb_j\sim\mathrm{Unif}[0,1]$; and (iii) Gaussian noise $\xb_j\sim\mathcal{N}(0,1)$. Figure \ref{fig:sub22} illustrates how the eigenvalues of $\Lb^{\yb}$, constructed from the signal-plus-noise model $\yb_j=\mathbf{s}_j+\xb_j$, $\xb_j \sim \mathcal{N}(0, \sigma^2),$ deviate from the limiting values in the clean signal cases (i)–(ii) as $\sigma$ varies, as measured by $\mathbb{T}$ with $M=6$ in \eqref{eq:t}. We point out that when $\sigma$ becomes large, the values of $\mathbb{T}$ can be predicted using the limiting behavior described in Example \ref{example:eg1}, as illustrated by the dashed lines.  All results are based on 1,000 repetitions.}}
		\label{figure2}
	\end{figure}
	
	\subsection{Proof strategy}\label{sec:stg}
	Before discussing our proof strategy, we first point out that when the density function $\varrho$ of $\{\xb_j\}_{j=1}^\infty$ has compact support and is bounded from above and below, the edge eigenvalues of $\Lb$ have been studied in the literature, for example, \cite{calder2022improved,cheng2022eigen,garcia2020error}. However, none of those critical assumptions hold in our setup. Consequently, the techniques developed in those papers cannot be carried over to study our problem. For example, when $\varrho$ has unbounded support, the  infinity transportation distances used in \cite{calder2022improved,garcia2020error} become infinity. As a result, the errors introduced by transportation-based interpolation between the eigenvectors of $\Lb$ and the eigenfunctions of $\Dr$ in \cite{calder2022improved,garcia2020error} will also become infinity, which indicates that the idea of \cite{calder2022improved,garcia2020error} cannot be applied. For another instance, in \cite{cheng2022eigen}, the interpolation-based errors rely on the $\ell_\infty$ norm of the eigenfunctions of $\Dr$. However, in our setup that $\varrho$ has unbounded support, as can be seen in the proof of Theorem \ref{theo:eigen}, the $\ell_\infty$ norm of the eigenfunctions of $\Dr$ will be unbounded, and the interpolation-based error bound will diverge so that the idea in \cite{cheng2022eigen} cannot be applied. Therefore, new ideas are needed to study $\Lb$ in our setup.
	
	In our first step, unlike \cite{cheng2022eigen,garcia2020error}, we will construct a suitable empirical operator to connect $\Lb$ and $\Dr$. Similar ideas have been used in \cite{belkin2006convergence,calder2022improved,cheng2022convergence,coifman2006diffusion,dunson2021spectral,hein2007graph,singer2006graph,singer2017spectral}. However, the ideas from the aforementioned works cannot be applied to our setup as $\operatorname{supp}(\varrho)=\R^d$. Especially, the empirical operator directly based on $\Lb$, denoted as
	\begin{align}
		\mathcal{L}_n f(\xb)\coloneq\frac{2m_0}{m_2r_n^2}\frac{\tfrac{1}{n}\sum_{j=1}^n\bm{1}(\normp{\xb-\xb_j}\le r_n)g(\normp{\xb-\xb_j}/r_n)(f(\xb)-f(\xb_j))}{\tfrac{1}{n}\sum_{j=1}^n\bm{1}(\normp{\xb-\xb_j}\le r_n)g(\normp{\xb-\xb_j}/r_n)},\label{eq:2.2.4}
	\end{align}
	is not well-defined for all $\xb\in\R^d$ such that $\normp{\xb-\xb_j}>r_n$ for all $j=1,\cdots,n$. In order to address this issue, we instead construct an empirical operator based on a smoothed version of $\Lb$. Specifically, we smooth the hard cut-off $\bm{1}_{[0,1]}(\cdot)$ with a sigmoid function and construct a smoothed affinity matrix 
	\begin{align}
		\Ka(i,j)=\sig\left(\alpha_n(r_n^2 -\normp{\xb_i-\xb_j}^2)\right)g^*(\normp{\xb_i-\xb_j}/r_n),\quad \sig(t)=(1+\exp(-t))^{-1},\label{eq:2.2.1}
	\end{align}
	where $g^*\in C^2([0,\infty))$ satisfies $g\equiv g^*$ on $[0,1]$ and $0<g^*(y)\le \sup_{y\ge0} g^*(y)<\infty$ for $y\in[0,\infty)$.
	The corresponding scaled random-walk Laplacian matrix is given by
	\begin{align}
		\Lc\coloneq\frac{2m_0}{m_2r_n^2}(\Ib-\Da^{-1}\Ka),\label{eq:2.2.2}
	\end{align}
	where $\Da = \operatorname{diag}(\Da(i,i))$ with $\Da(i,i) = \sum_{j=1}^n \Ka(i,j)$.
	As we will show in Lemma \ref{lemma:sig2ind}, for suitably chosen large $\alpha_n$, $\sig\left(\alpha_n(r_n^2 -\normp{\xb_i-\xb_j}^2)\right)$ sharply approximates $\bm{1}_{[0,1]}(\normp{\xb_i-\xb_j}/r_n)$.
	Together with the boundedness of $g^*(\cdot)$, the entries of $\Kb$ and $\Ka$ are close with high probability regardless of whether $\normp{\xb_i-\xb_j}\le r_n$ or $\normp{\xb_i-\xb_j}>r_n$, which implies that $\Lb$ and $\Lc$ are also close in operator norm.
	
	Then, we construct an empirical operator $\sLn:\cF\rightarrow\cF$ based on $\Lc$ defined by
	\begin{align}
		\sLn f(\xb)\coloneq\frac{2m_0}{m_2r_n^2}\frac{\tfrac{1}{n}\sum_{j=1}^n\sig(\alpha_n(r_n^2-\normp{\xb-\xb_j}^2))g^*(\normp{\xb-\xb_j}/r_n)(f(\xb)-f(\xb_j))}{\tfrac{1}{n}\sum_{j=1}^n\sig(\alpha_n(r_n^2-\normp{\xb-\xb_j}^2))g^*(\normp{\xb-\xb_j}/r_n)}.\label{eq:2.2.3}
	\end{align}
	There are several advantages of using \eqref{eq:2.2.3}. Firstly, compared to \eqref{eq:2.2.4}, \eqref{eq:2.2.3} is well-defined for all $\xb\in\R^d$. Secondly, it can be shown that except for at most one diverging eigenvalue, the remaining eigenvalues of $\sLn$ and $\Lc$ coincide (see Lemma \ref{lemma:S2Sn} or \cite[Proposition 9]{von2008consistency}). Moreover, we can show that the edge eigenvalues of $\sLn$ are bounded from above. Consequently, this builds a one-to-one correspondence between the edge eigenvalues of $\sLn$ and $\Lc$. As a result, we build the connection between the edge eigenvalues of $\sLn$ and those of $\Lb$.
	
	The rest of the work remains to connect the eigenvalues of $\sLn$ to those of $\Dr$. Existing literature provides two possible strategies, \cite{calder2022improved,cheng2022eigen,garcia2020error} apply the min-max theorem by interpolations based on optimal transport or heat kernel while \cite{dunson2021spectral,shen2022scalability,singer2017spectral} aim at controlling $\norm{\sLn-\Dr}$. The ideas from the aforementioned works require either compact support or bounded $\ell_\infty$ norm of the eigenfunctions of $\Dr$ and they often assume the eigenvalues of $\Dr$ are simple. However, none of the assumptions are satisfied in our setup. In our paper, we propose an eigenvalue-counting-based method. For any fixed integer $M_0>0$, let $\mu_1^*<\mu_2^*<\cdots<\mu_{M_0}^*$ denote the first $M_0$ distinct non‑trivial eigenvalues of $\Dr$, which are separated by a constant gap (see Theorem \ref{theo:eigen}). For arbitrary small $\epsilon>0$, we can split $(\delta,\mu_{M_0}^*+\epsilon)$ with $\delta$ in \eqref{eq:2.1} into the following three possible types of intervals: $(\mu_j^*-\epsilon,\mu_j^*+\epsilon)$ for $j=1,\cdots,M_0$, $[\mu_{j}^*+\epsilon,\mu_{j+1}^*-\epsilon]$ for $j=1,\cdots,M_0-1$ and $(\delta,\mu_1^*-\epsilon]$. We can show that for small $\delta$ in \eqref{eq:2.1}, $K_0=\oo(n)$ with high probability, and at least one of the eigenvalues of $\sLn$ is located in each $(\mu_j^*-\epsilon,\mu_j^*+\epsilon)$ with high probability. Moreover, none of the eigenvalues of $\sLn$ falls into $\cup_{j=1}^{M_0-1}[\mu_{j}^*+\epsilon,\mu_{j+1}^*-\epsilon]$ or $(\delta,\mu_1^*-\epsilon]$ with high probability. In a second step, we will use a counting formula (Theorem \ref{theo:counting}) to show that $\sLn$ and $\Dr$ have the same number of eigenvalues (with multiplicities) inside each $(\mu_j^*-\epsilon,\mu_j^*+\epsilon)$ with high probability. The key ingredient to establish the results is to show that for some properly chosen testing functions $\{\varphi_{j,l}\}$, with high probability,
	\begin{align}
		\normf{(\sLn-\Dr)\varphi_{j,l}}=\oo(1).\label{eq:ingredient}
	\end{align}
	To build such a result, we utilize a bias-variance-type argument by introducing a deterministic operator
	\begin{align}
		\mathcal{T}_nf(\xb)\coloneq\frac{2}{r_n^{d+2}m_2\varrho(\xb)}\int_{\R^d}\bm{1}(\normp{\xb-\yb}\le r_n)g(\normp{\xb-\yb}/r_n)(f(\xb)-f(\yb))\varrho(\yb)\dd \yb,\label{eq:2.2.5}
	\end{align}which can be regarded as a close approximation of $\E\sLn$. Then by careful and detailed analysis of the eigenfunctions of $\Dr$, we can show that $\normf{(\mathcal{T}_n-\sLn)\varphi_{j,l}}=\oo(1)$ with high probability and $\normf{(\mathcal{T}_n-\Dr)\varphi_{j,l}}=\oo(1)$. This completes our proof.
	
	
	\section{Spectrum of $\Dr$}\label{sec_3}
	In this section, we study the spectrum of $\Dr$ in \eqref{eq:2.1.1} with $\varrho$ in \eqref{eq:2.1.4}. The results provide the testing functions used in the proof of Theorem \ref{theo:normald=1} and the limits of the first few non-trivial eigenvalues of $\Lb$ in Proposition \ref{prop:d>1}. The main results are stated in Theorem \ref{theo:eigen} below. Recall that the physicist's Hermite polynomials \cite{szeg1939orthogonal} are defined as, for $j\in\N$,
	$$H_j(x)=(-1)^j e^{x^2}\frac{\dd^j}{\dd x^j}e^{-x^2},\quad x\in\R.$$
	
	\begin{theorem}
		\label{theo:eigen}
		For the weighted Laplace-Beltrami operator $\Dr$ on $(\cF,\Rprod{\cdot}{\cdot})$ in \eqref{eq:2.1.3}, its spectrum is discrete and consists of eigenvalues $\sum_{i=1}^d 2(k_i-1)/\sigma_i^2$ with corresponding orthonormal eigenfunctions $$\prod_{i=1}^{d}\psi_{i,k_i-1}(x_i),\quad k_1,\cdots,k_d\in\N^+,$$ where 
		$$\psi_{i,k_i-1}(x_i)=\left(\frac{1}{\sqrt{\pi}(k_i-1)!2^{k_i-1}\sigma_i}\right)^{1/2}H_{k_i-1}(x_i/\sigma_i),$$ 
		and $H_{k_i-1}$ is the $(k_i-1)$-th physicist's Hermite polynomial.
	\end{theorem}
	\begin{remark}
		\label{rmk:eigen}
		For any eigenvalue $\lambda$ of $\Dr$, its multiplicity equals the dimension of 
		$$\operatorname{span}\left\{\prod_{i=1}^d\psi_{i,k_i-1}(x_i):\sum_{i=1}^d 2(k_i-1)/\sigma_i^2=\lambda,\quad k_1,\cdots,k_d\in\N^+\right\}.$$
		We can see that zero eigenvalue is always simple. Moreover, when $d=1$, all eigenvalues of $\Dr$ are simple, as in Example \ref{example:eg1}. When $d\ge2$, some of the eigenvalues can be identical. Furthermore, for any finite index sequence $\{(k_1,\cdots,k_d)\}$, the associated eigenvalues are bounded from above. In addition, by our assumption on $\delta$ in \eqref{eq:delta}, $\mu_2(\Dr)>\delta$.
	\end{remark}
	
	The proof of Theorem \ref{theo:eigen} relies on the observation that the eigenvalues and eigenfunctions of $\Dr$ can be constructed using the solutions of $d$ ordinary differential equations, which are all the Hermite differential equations \cite{szeg1939orthogonal} up to change of variables. Moreover, the analytical properties of the eigenfunctions can be studied using the confluent hypergeometric functions \cite[Chapter 13]{olver2010nist}. 
	
	\begin{proof}
		To solve the eigenvalue problem, we seek solutions to
		\begin{align}
			\Dr f=\lambda f,\quad f\in\cF.\label{eq:4.1}
		\end{align}
		Recalling $\varrho(\xb)$ in \eqref{eq:2.1.4}, we can rewrite \eqref{eq:4.1} as
		\begin{align}
			\sum_{i=1}^d\frac{\partial^2f}{\partial x_i^2}-\sum_{i=1}^d \frac{2x_i}{\sigma_i^2}\frac{\partial f}{\partial x_i}+\lambda f=0,\quad f\in\cF,\quad f(\xb)\not\equiv0.\label{eq:4.2}
		\end{align}
		According to \eqref{eq:A2.1.3}, to solve \eqref{eq:4.2}, it suffices to solve
		\begin{align}
			\mathsf{f}_i''(x_i)-\frac{2x_i}{\sigma_i^2}\mathsf{f}_i'(x_i)+\lambda_i\mathsf{f}_i(x_i)=0,\quad \mathsf{f}_i\in\cF_i,\quad \mathsf{f}_i(x_i)\not\equiv0,\quad i=1,\cdots,d,\label{eq:4.3}
		\end{align}
		where $\sum_{i=1}^d\lambda_i=\lambda$ and $\cF_i$ is the Hilbert space defined as 
		\begin{align}
			\cF_i\coloneq\left\{\mathsf{f}_i:\R\rightarrow\R\biggm|\int_\R \mathsf{f}_i^2(x_i)\exp\left(-x_i^2/\sigma_i^2\right)\dd x_i<\infty\right\},\label{eq:4.4}
		\end{align}
		with corresponding inner product 
		\begin{align}
			\Rprod{\mathsf{f}_i}{\mathsf{g}_i}_i\coloneq\int_\R\mathsf{f}_i(x_i)\mathsf{g}_i(x_i)\exp\left(-x_i^2/\sigma_i^2\right)\dd x_i,\quad \mathsf{f}_i,\mathsf{g}_i\in\cF_i,\label{eq:4.9}
		\end{align}
		and induced norm
		\begin{align}
			\norm{\mathsf{f}_i}_{\mathcal{F}_i}=\sqrt{\Rprod{\mathsf{f}_i}{\mathsf{f}_i}_i}.\label{eq:4.9.1}
		\end{align}
		
		We begin solving \eqref{eq:4.3}. First, we find the general solution to \eqref{eq:4.3}. By Theorem \ref{theo:ode2}, it suffices to find two linearly independent solutions $\mathsf{f}_{i,1}$ and $\mathsf{f}_{i,2}$ and any solution $\mathsf{f}_i$ to \eqref{eq:4.3} can be written as a linear combination $c_{i,1}\mathsf{f}_{i,1}+c_{i,2}\mathsf{f}_{i,2}$ for some $c_{i,1},c_{i,2}\in\R$. We start seeking solutions that can be represented by the following power-series expansion:
		\begin{align}
			\mathsf{f}_i(x_i)=\sum_{j=0}^{\infty}b_{i,j}x_i^j,\label{eq:4.5}
		\end{align}
		where $\{b_{i,j}\}_{j=0}^\infty$ are coefficients ensuring that $\mathsf{f}_i$ has an infinite radius of convergence and is therefore analytic on $\R$. By Cauchy–Hadamard theorem and Weierstrass $M$-test \cite{rudin1976principles}, this implies that the series in \eqref{eq:4.5} converges uniformly on every compact subset of $\R$. Combining with the analyticity of $\mathsf{f}_i$ on $\R$, we further write:
		\begin{align*}
			\mathsf{f}_i'(x_i)&=\sum_{j=1}^{\infty}jb_{i,j}x_i^{j-1},\quad \mathsf{f}_i''(x_i)=\sum_{j=2}^{\infty}j(j-1)b_{i,j}x_i^{j-2}.
		\end{align*}
		Substituting the above expansion into \eqref{eq:4.3} and matching coefficients of $x_i^j$ gives, for all $j\ge0$,
		$$b_{i,j+2}=\frac{2j/\sigma_i^2-\lambda_i}{(j+2)(j+1)}b_{i,j}.$$
		To find two linearly independent solutions to \eqref{eq:4.3}, it suffices to impose some initial conditions so that we can find two linearly independent solutions to \eqref{eq:4.3} in terms of \eqref{eq:4.5}. Without loss of generality, we assume that 
		$$\mathsf{f}_i(0)=1,\mathsf{f}_i'(0)=0\text{ and }\mathsf{f}_i(0)=0,\mathsf{f}_i'(0)=1.$$
		This yields two linearly independent functions—an even $\mathsf{f}_{i,1}$ and an odd $\mathsf{f}_{i,2}$:
		\begin{align}
			\mathsf{f}_{i,1}(x_i)&=1-\frac{\lambda_i}{2!}x_i^2-\frac{(4/\sigma_i^2-\lambda_i)\lambda_i}{4!}x_i^4-\cdots,\label{eq:4.6}\\
			\mathsf{f}_{i,2}(x_i)&=x_i+\frac{(2/\sigma_i^2-\lambda_i)}{3!}x_i^3+\frac{(6/\sigma_i^2-\lambda_i)(2/\sigma_i^2-\lambda_i)}{5!}x_i^5+\cdots.\label{eq:4.7}
		\end{align}
		By Cauchy–Hadamard theorem \cite{rudin1976principles}, both $\mathsf{f}_{i,1}(x_i)$ and $\mathsf{f}_{i,2}(x_i)$ have infinite radius of convergence. So $\mathsf{f}_{i,1}$ and $\mathsf{f}_{i,2}$ are two linearly independent solutions to \eqref{eq:4.3} in terms of \eqref{eq:4.5}. Combining the above discussion, we see that $\{\mathsf{f}_{i,1},\mathsf{f}_{i,2}\}$ forms a fundamental system and any solution $\mathsf{f}_i$ to \eqref{eq:4.3} can be written as a linear combination $c_{i,1}\mathsf{f}_{i,1}+c_{i,2}\mathsf{f}_{i,2}$ for some $c_{i,1},c_{i,2}\in\R$.
		
		Second, we determine which of these solutions belong to $\cF_i$. We analyze the asymptotics of $\mathsf{f}_{i,1}(x_i)$ and $\mathsf{f}_{i,2}(x_i)$ as $|x_i|\rightarrow\infty$. By Remark \ref{remark:chf1}, Lemma \ref{lemma:chfkarge} and Remark \ref{remark:chf2}, we have that
		\begin{align*}
			\mathsf{f}_{i,1}(x_i)\sim\frac{\Gamma(1/2)}{\Gamma(-\lambda_i\sigma_i^2/4)}e^{x_i^2/\sigma_i^2}|x_i/\sigma_i|^{-\lambda_i\sigma_i^2/2-1}\quad(|x_i|\rightarrow\infty),
		\end{align*}
		when $\lambda_i\sigma_i^2\notin\{0,4,8,\cdots\}$. Hence, we have
		$$\mathsf{f}_{i,1}^2(x_i)\exp\left(-x_i^2/\sigma_i^2\right)\sim e^{x_i^2/\sigma_i^2}|x_i/\sigma_i|^{-\lambda_i\sigma_i^2-2},$$
		which diverges as $|x_i|\rightarrow\infty$. Combining with \eqref{eq:4.4}, this implies that $\mathsf{f}_{i,1}\notin\cF_i$ when $\lambda_i\sigma_i^2\notin\{0,4,8,\cdots\}$. On the other hand, when $\lambda_i=4s/\sigma_i^2$ for $s\in\{0,1,2,\cdots\}$, $\mathsf{f}_{i,1}$ becomes $H_{2s}(x_i/\sigma_i)$ (up to normalization), where $H_{2s}$ is the $2s$-th physicist's Hermite polynomial \cite{szeg1939orthogonal} and has degree $2s$. This implies that $\mathsf{f}_{i,1}(x_i)$ is square-integrable with respect to the measure $e^{-x_i^2/\sigma_i^2}\dd x_i$. Consequently, $\mathsf{f}_{i,1}\in\cF_i$ if and only if $\lambda_i=0,4/\sigma_i^2,8/\sigma_i^2,\cdots$. Similarly, $\mathsf{f}_{i,2}(x_i)$ is square-integrable with respect to the measure $e^{-x_i^2/\sigma_i^2}\dd x_i$ if and only if $\lambda_i=2/\sigma_i^2,6/\sigma_i^2,10/\sigma_i^2,\cdots$. This yields $H_1(x_i/\sigma_i),H_3(x_i/\sigma_i),H_5(x_i/\sigma_i),\cdots$ (up to normalization). Combining the two cases, the only $\lambda_{i}$ admitting solutions in $\cF_i$ are $0,2/\sigma_i^2,4/\sigma_i^2,\cdots$. Moreover, the corresponding eigenfunctions of \eqref{eq:4.3} are $\{H_{m}(x_i/\sigma_i)\}_{m=0}^\infty$, up to normalization.
		
		Next, we normalize $\{H_{m}(x_i/\sigma_i)\}_{m=0}^\infty$ to form a complete orthonormal basis of $(\cF_i,\Rprod{\cdot}{\cdot}_i)$ to address the eigenvalue problem \eqref{eq:4.3}. By the orthogonality property of Hermite polynomials \cite{szeg1939orthogonal}, we have
		\begin{align*}
			\int_\R H_l(x_i/\sigma_i)H_s(x_i/\sigma_i)e^{-x_i^2/\sigma_i^2}\dd x_i=\sqrt{\pi}2^l l!\sigma_i\bm{1}_{\{l=s\}}.
		\end{align*}
		Normalizing each $H_m(x_i/\sigma_i)$ yields the orthonormal eigenfunctions: 
		\begin{align}
			\psi_{i,m}(x_i)=\left(\frac{1}{\sqrt{\pi}m!2^m\sigma_i}\right)^{1/2}H_m(x_i/\sigma_i),\text{ for }m=0,1,\cdots.\label{eq:4.8}
		\end{align}
		This shows that the solutions to \eqref{eq:4.3} are $2(k_i-1)/\sigma_i^2$ with corresponding orthonormal eigenfunctions $\psi_{i,k_i-1}(x_i)$ for $k_i\in\N^+$.
		
		Now we construct the eigenvalues and eigenfunctions of $\Dr$ using those of \eqref{eq:4.3}. By the arguments below Theorem \ref{theo:ode2}, $\left\{\prod_{i=1}^d\psi_{i,k_i-1}(x_i)\right\}_{k_1,\cdots,k_d\in\N^+}$ constitutes a complete orthonormal basis of $\cF$. Combining with the fact that each $\prod_{i=1}^d\psi_{i,k_i-1}(x_i)$ is a solution to \eqref{eq:4.2} with $\lambda=\sum_{i=1}^d 2(k_i-1)/\sigma_i^2$, we have that \eqref{eq:4.2} is solvable if and only if $$\lambda\in\left\{\sum_{i=1}^d 2(k_i-1)/\sigma_i^2\right\}_{k_1,\cdots,k_d\in\N^+},$$
		and any solution $f(\xb)$ to \eqref{eq:4.2} satisfies
		$$f(\xb)\in\operatorname{span}\left\{\prod_{i=1}^d\psi_{i,k_i-1}(x_i):\sum_{i=1}^d 2(k_i-1)/\sigma_i^2=\lambda,\quad k_1,\cdots,k_d\in\N^+\right\}.$$
		Consequently, the spectrum of $\Dr$ on $(\cF,\Rprod{\cdot}{\cdot})$ is discrete and consists of eigenvalues $\sum_{i=1}^d 2(k_i-1)/\sigma_i^2$ with corresponding orthonormal eigenfunctions $\prod_{i=1}^{d}\psi_{i,k_i-1}(x_i)$ for $k_1,\cdots,k_d\in\N^+$. This completes our proof.
	\end{proof}

	\section{Proof of main results}\label{sec:mainproof}
	In this section, we prove the main results of this paper. We first show how Theorem \ref{theo:normald=1} implies Proposition \ref{prop:d>1}.
	
	\begin{proof}[\bf{Proof of Proposition \ref{prop:d>1}}]
		By Theorem \ref{theo:eigen}, we know that the eigenvalues of $\Dr$ in \eqref{eq:2.1.1} on $(\cF,\Rprod{\cdot}{\cdot})$ are 
		$$\left\{\sum_{i=1}^d 2(k_i-1)/\sigma_i^2\right\}_{k_1,\cdots,k_d\in\N^+},$$
		counted with multiplicities. Recall \eqref{eq:2.1.5}. We can then complete the proof of Proposition \ref{prop:d>1} with Theorem \ref{theo:normald=1}.
	\end{proof}
	
	The rest of the section leaves to prove Theorem \ref{theo:normald=1}. We divide the proof into two part. In Section \ref{sec:k0}, we will bound $K_0$ in \eqref{eq:2.1}. Then in Sections \ref{sec:converge} and \ref{sec:prop}, we will prove \eqref{eq:main}.
	
	\subsection{Bounding $K_0$}\label{sec:k0}
	In light of the relation between the eigenvalues of $\Lb_{rw}$ in \eqref{eq:1.1} and $\Lb$ in \eqref{eq:2.1.2}, we will work directly with $\Lb_{rw}$. For convenience of notation, denote 
	\begin{align}
		\xi(d)\coloneq (d+2)^2/2.\label{eq:xid}
	\end{align}
	The proof relies on our observation that the majority of the eigenvalues of $\Lb_{rw}$, denoted as $\{\lambda_l^*\}_{l=1}^n$, will be around 1 in the sense that $\frac{1}{n}\sum_{l=1}^n(1-\lambda_l^*)^2$ can be bounded by $\OO(n^{-1/\xi(d)})$ with high probability. 
	Consequently, it is impossible to have more than $n^{1-2/(2\xi(d)+\eta)}$ eigenvalues of $\Lb_{rw}$ deviating significantly below 1, which yields the $n^{1-2/(2\xi(d)+\eta)}$ upper bound for $K_0$.
	\begin{proof}
		Recall $\Lb_{rw}$ in \eqref{eq:1.1}. We denote the eigenvalues of $\Lb_{rw}$ as follows,
		$$\lambda_1^*\le\lambda_2^*\le\cdots\le\lambda_n^*.$$
		We can see that for $l=1,\cdots,n$, $\lambda_l^*=\frac{m_2r_n^2}{2m_0}\lambda_l(\Lb)$, with $m_0,m_2$ as in \eqref{eq:mj} and $\Lb$ in \eqref{eq:2.1.2}. Therefore, showing $K_0\le n^{1-\frac{2}{2\xi(d)+\eta}}$ is equivalent to showing that the number of eigenvalues of $\Lb_{rw}$ smaller than or equal to $\frac{m_2\delta}{2m_0}r_n^2$ is less than or equal to $n^{1-\frac{2}{2\xi(d)+\eta}}$, where $\delta$ is in \eqref{eq:2.1}.
		We now prove $K_0\le n^{1-\tfrac{2}{2\xi(d)+\eta}}$ by contradiction. We assume to the contrary that, for all sufficiently large $n$, 
		\begin{align}
			K_0>n^{1-\tfrac{2}{2\xi(d)+\eta}}.\label{eq:5.1assump}
		\end{align}
		The assumption \eqref{eq:5.1assump} indicates that for all $l\le n^{1-\tfrac{2}{2\xi(d)+\eta}}$, 
		\begin{align}
			\lambda_l^*\le\frac{m_2\delta}{2m_0}r_n^2.\label{eq:5.1assump1}
		\end{align}
		Recall $\Kb$ in \eqref{eq_originalK} and $\Db$ in \eqref{eq:1.1}. Applying \eqref{eq:HSnorm} to $\Db^{-1}\Kb$ gives
		\begin{align}
			\frac{1}{n}\sum_{l=1}^n|\lambda_l(\Db^{-1}\Kb)|^2\le\frac{1}{n}\sum_{j=1}^n\sum_{s=1}^n |(\Db^{-1}\Kb)_{js}|^2,\label{eq:5.1.0}
		\end{align}
		where $\lambda_l(\Db^{-1}\Kb)$ is the $l$-th eigenvalue of $\Db^{-1}\Kb$ and $(\Db^{-1}\Kb)_{js}$ is the $(j,s)$-th entry of $\Db^{-1}\Kb$. Observe that the eigenvalues of $\Db^{-1}\Kb$ are $\{1-\lambda_l^*\}_{l=1}^n$ and $(\Db^{-1}\Kb)_{js}=\Kb(j,s)/\Db(j,j)$. Therefore, \eqref{eq:5.1.0} becomes
		\begin{align}
			\frac{1}{n}\sum_{l=1}^n(1-\lambda_l^*)^2\le\frac{1}{n}\sum_{j=1}^n\sum_{s=1}^n\left(\frac{\Kb(j,s)}{\Db(j,j)}\right)^2.\label{eq:5.1.1}
		\end{align}
		For sufficiently large $n$, we have that
		\begin{align}
			\frac{1}{n}\sum_{l=1}^n(1-\lambda_l^*)^2&\ge\frac{1}{n}\sum_{l=1}^{\floor{n^{1-2/(2\xi(d)+\eta)}}}(1-\lambda_l^*)^2\notag\\
			&\ge n^{-\tfrac{2}{2\xi(d)+\eta}}/2,\label{eq:5.1.2}
		\end{align}
		where in the last step we used \eqref{eq:5.1assump1} and the fact that $\frac{m_2\delta}{2m_0}r_n^2=\oo(1)$.
		Under the boundedness of $g(\cdot)$ in Theorem \ref{theo:normald=1}, there exist $C_1,C_2$ such that $0<C_1\le g(y)\le C_2$ for any $y\in[0,1]$. This indicates that for any $\xb_j,\xb_s$,
		\begin{align}
			C_1\bm{1}(\normp{\xb_j-\xb_s}\le r_n)\le \Kb(j,s)\le C_2\bm{1}(\normp{\xb_j-\xb_s}\le r_n).\label{eq:5.1.3}
		\end{align}
		Denote the number of points connected to $\xb_j$ as
		\begin{align}
			\mathsf{d}_j\coloneq\sum_{s=1}^n\bm{1}(\normp{\xb_s-\xb_j}\le r_n)\ge1.\label{eq:5.1.11}
		\end{align}
		Then we have that
		\begin{align}
			\frac{1}{n}\sum_{j=1}^n\sum_{s=1}^n\left(\frac{\Kb(j,s)}{\Db(j,j)}\right)^2&\le\frac{1}{n}\sum_{j=1}^n\sum_{s=1}^n\left(\frac{C_2\bm{1}(\normp{\xb_j-\xb_s}\le r_n)}{C_1\mathsf{d}_j}\right)^2\notag\\
			&=\frac{C_2^2}{C_1^2n}\sum_{j=1}^{n}\frac{1}{\mathsf{d}_j},\label{eq:5.1.5}
		\end{align}
		where the first step is from \eqref{eq:5.1.3}. Combining \eqref{eq:5.1.1}, \eqref{eq:5.1.2} and \eqref{eq:5.1.5} gives, 
		\begin{align}
			\frac{1}{n}\sum_{j=1}^{n}\frac{1}{\mathsf{d}_j}\ge\frac{C_1^2}{2C_2^2}n^{1-\tfrac{2}{2\xi(d)+\eta}}.\label{eq:5.1.6}
		\end{align}
		
		To contradict \eqref{eq:5.1.6}, on the other hand, we will show that with probability $1-\oo(1)$, $\frac{1}{n}\sum_{j=1}^{n}\frac{1}{\mathsf{d}_j}=\OO(n^{-1/(\xi(d))})=\oo(n^{-2/(2\xi(d)+\eta)})$. Before we proceed, we define a bulk region $B_n\subset\R^d$ by
		\begin{align}
			B_n\coloneq \prod_{i=1}^d\left[-\sigma_i\sqrt{\log n/(d^2/4+d+1)},\sigma_i\sqrt{\log n/(d^2/4+d+1)}\right],\label{eq:5.1.8}
		\end{align}
		where $\prod$ denotes Cartesian product for sets, and set the corresponding tail region $T_n=\R^d\setminus B_n$. Therefore, we split all points $\xb_j$ into the bulk points in $B_n$ and the tail points in $T_n$. To show $\frac{1}{n}\sum_{j=1}^{n}\frac{1}{\mathsf{d}_j}=\oo(n^{-2/(2\xi(d)+\eta)})$, we will show that there are only $\OO(n^{1-1/(\xi(d)}))$ tail points, and that for any $\xb_j$ in the bulk region, $\mathsf{d}_j\ge n^{1/2}$.
		
		First, we show that there are only $\OO(n^{1-1/(\xi(d))})$ points in the tail region. 
		Denote $p_n\coloneq\P(\xb\in T_n)$, where $\xb=(x_1,\cdots,x_d)^\top\sim\mathcal{N}_d(\bm{0},\Sigma)$ with $\Sigma$ in \eqref{eq:2.1.0}. Recall \eqref{eq:xid}. We have that
		\begin{align}
			p_n&=\P\left(\bigcup_{i=1}^d\left\{x_i\notin \left[-\sigma_i\sqrt{\log n/(d^2/4+d+1)},\sigma_i\sqrt{\log n/(d^2/4+d+1)}\right]\right\}\right)\notag\\
			&\le\sum_{i=1}^d \P\left(\left|\frac{x_i}{\sigma_i}\right|\ge\sqrt{\log n/(d^2/4+d+1)}\right)\le2d\exp\left(-\frac{1}{d^2/2+2d+2}\log n\right)=2dn^{-1/(\xi(d))},\label{eq:5.3}
		\end{align}
		where in the second step we used the union bound and in the third step we used the fact that $x_i/\sigma_i\sim\mathcal{N}(0,1)$ and the Gaussian tail bound \cite[Section 2.1.2]{wainwright2019high} that for $z\sim\mathcal{N}(0,1)$ and any $t\ge0$,
		$$\P(|z|\ge t)\le2\exp\left(-t^2/2\right).$$
		On the other hand, we have that 
		\begin{align}
			p_n&=\P\left(\bigcup_{i=1}^d\left\{x_i\notin \left[-\sigma_i\sqrt{\log n/(d^2/4+d+1)},\sigma_i\sqrt{\log n/(d^2/4+d+1)}\right]\right\}\right)\notag\\
			&\ge\P\left(\left|\frac{x_i}{\sigma_i}\right|\ge\sqrt{\log n/(d^2/4+d+1)}\right)\gtrsim \frac{1}{\sqrt{\log n}}\exp\left(-\frac{1}{d^2/2+2d+2}\log n\right)\ge n^{-1/(d^2/2+2d)},\label{eq:5.3a}
		\end{align}
		where in the third step we the Gaussian tail bound \cite[Section 7.1]{feller1991introduction} that $z\sim\mathcal{N}(0,1)$ and any $t\ge2$,
		$$\P(|z|\ge t)\ge \frac{2}{\sqrt{2\pi}}(t^{-1}-t^{-3})\exp\left(-t^2/2\right)\ge t^{-1}\exp(-t^2/2)/3.$$
		Moreover, since $\xb_j\stackrel{\operatorname{i.i.d.}}{\sim}\mathcal{N}_d(\bm{0},\Sigma)$, we have that
		$$\bm{1}(\xb_j\in T_n)\stackrel{\mathrm{i.i.d.}}{\sim}\operatorname{Ber}(p_n),\quad j=1,\cdots,n.$$
		Now we proceed to control the summation of $\bm{1}(\xb_j\in T_n)$, $j=1,\cdots,n$.
		We have that
		\begin{align}
			\P\left(\sum_{j=1}^n\bm{1}(\xb_j\in T_n)-np_n\ge \frac{1}{2}np_n\right)&\le\exp\left(-\frac{np_n}{12}\right)\le\exp\left(-Cn^{1-1/(d^2/2+2d)}\right)=\oo(1),\notag
		\end{align}where in the first step we used Lemma \ref{lemma:chernoff}, and in the second step we used \eqref{eq:5.3a} and $C$ is some absolute positive constant. Consequently, with probability $1-\oo(1)$, there are at most $1.5np_n\le3dn^{1-1/(\xi(d))}=\OO(n^{1-1/(\xi(d))})$ points in the tail region $T_n$.
		
		Second, we show that with probability $1-\oo(1)$, for all $\xb_j\in B_n$, $\mathsf{d}_j\ge n^{1/2}$. Fix one $\xb_{j_0}\in B_n$, we have that for all $j\ne j_{0}$, $\bm{1}(\normp{\xb_j-\xb_{j_0}}\le r_n)|\xb_{j_0}\stackrel{\mathrm{i.i.d.}}{\sim}\operatorname{Ber}(q_{n,j_0})$,
		where
		\begin{align}
			q_{n,j_0}\coloneq\int_{\normp{\xb_{j_0}-\xb}\le r_n}\prod_{i=1}^d\left(\sigma_i^{-1}\phi\left(\frac{x_{j_0,i}}{\sigma_i}\right)\right)\dd \xb,\label{eq:5.1.7}
		\end{align}
		and $\phi(\cdot)$ is the density of standard Gaussian distribution.
		For sufficiently large $n$ and any $\xb=(x_1,\cdots,x_d)^\top$ such that $\normp{\xb-\xb_{j_0}}\le r_n$, 
		$$\left|\frac{x_i}{\sigma_i}\right|\le\left|\frac{x_{j_0,i}+r_n}{\sigma_i}\right|\le\sqrt{\log n/(d^2/4+d+1)}+r_n/\sigma_i\le\sqrt{\log n/(d^2/4+d)},\quad i=1,\cdots, d,$$
		where in the second step we used \eqref{eq:5.1.8} and in the last step we used \eqref{eq:2.2}. Combining with \eqref{eq:5.1.7}, we have that 
		\begin{align}
			q_{n,j_0}&\ge\int_{\normp{\xb_{j_0}-\xb}\le r_n}\prod_{i=1}^d\left(\sigma_i^{-1}\frac{1}{\sqrt{2\pi}}\exp\left(-\frac{1}{d^2/2+2d}\log n\right)\right)\dd \xb\asymp r_n^d n^{-1/(d/2+2)},\label{eq:5.1.12}
		\end{align}
		where in the last step we used the fact that $\int_{\normp{\xb_{j_0}-\xb}\le r_n}\dd\xb\asymp r_n^d$. Hence, conditioning on $\xb_{j_0}$, we have that
		\begin{align}
			&\P\left(\left|\sum_{j\ne j_0}\bm{1}(\normp{\xb_j-\xb_{j_0}}\le r_n)-(n-1)q_{n,j_0}\right|\ge\frac{1}{2}(n-1)q_{n,j_0}\right)\notag\\
			\le&2\exp\left(-\frac{(n-1)q_{n,j_0}}{12}\right)\le2\exp(-C_3 n^{1-\tfrac{d}{d+4}-\tfrac{2}{d+4}})=2\exp(-C_3n^{2/(d+4)})=\oo(n^{-2}),\notag
		\end{align}
		where in the first step we used Lemma \ref{lemma:chernoff} and the fact that $0\le\bm{1}(\normp{\xb_j-\xb_{j_0}}\le r_n)|\xb_{j_0}\le1$, in the second step we used \eqref{eq:5.1.12} and \eqref{eq:2.2} and $C_3$ is some absolute positive constant. Since the conditional probability bound in \eqref{eq:5.1.8} is free of $\xb_{j_0}$, it holds unconditionally for any $\xb_{j_0}\in B_n$. Consequently, by taking the union bound over at most $n$ points $\xb_j$ such that $\xb_j\in B_n$, we have that for sufficiently large $n$, with probability $1-\oo(n^{-1})$, for all $\xb_j$ such that $\xb_j\in B_n$,
		\begin{align}
			\mathsf{d}_j&=1+\sum_{s\ne j}\bm{1}(\normp{\xb_j-\xb_s}\le r_n)\ge\frac{1}{2}(n-1)q_{n,j}\gtrsim n^{2/(d+4)},\label{eq:5.1.9}
		\end{align}
		where the second step is from \eqref{eq:5.1.8} and the last step is from \eqref{eq:5.1.12} and \eqref{eq:2.2}.
		
		Combining the above discussion, we have that with probability $1-\oo(1)$,
		\begin{align}
			\frac{1}{n}\sum_{j=1}^n\frac{1}{\mathsf{d}_j}&=\frac{1}{n}\sum_{\{j:\xb_j\in T_n\}}\frac{1}{\mathsf{d}_j}+\frac{1}{n}\sum_{\{j:\xb_j\in B_n\}}\frac{1}{\mathsf{d}_j}\notag\\
			&\lesssim\frac{1}{n}\sum_{\{j:\xb_j\in T_n\}}1+\frac{1}{n}\sum_{\{j:\xb_j\in B_n\}}\frac{1}{n^{2/(d+4)}}\notag\\
			&=\OO(n^{-1/(\xi(d)))})+\OO(n^{-2/(d+4)})=\OO(n^{-1/(\xi(d)))}),\label{eq:5.1.10}
		\end{align}
		where in the second step we used \eqref{eq:5.1.11} and \eqref{eq:5.1.9} and the last step is from the fact that for $d\ge1$, $d^2/2+2d+2\ge (d+4)/2$. Note that for sufficiently large $n$, \eqref{eq:5.1.10} contradicts \eqref{eq:5.1.6} that
		$$\frac{1}{n}\sum_{j=1}^n\frac{1}{\mathsf{d}_j}\gtrsim n^{-2/(2\xi(d)+\eta)}.$$ 
		Therefore, our assumption \eqref{eq:5.1assump} cannot hold. 
		Consequently, we have that with probability $1-\oo(1)$, $K_0\le n^{1-2/(2\xi(d)+\eta)}$. Recall \eqref{eq:xid} that $\xi(d)=(d+2)^2/2$. This completes our proof.
	\end{proof}
	
	\subsection{Proof of \eqref{eq:main}}\label{sec:converge}
	Recall $\Lb$ in \eqref{eq:2.1.2}, $\Dr$ in \eqref{eq:2.1.1}, $\Lc$ in \eqref{eq:2.2.2} and $\sLn$ in \eqref{eq:2.2.3}. We will follow the strategy described in Section \ref{sec:stg} to complete the proof. We first show the closeness between the eigenvalues of $\Lb$ and $\sLn$ by the closeness between the indicator function in $\Lb$ and its sigmoid-type approximation in $\Lc$ and the construction of $\sLn$ from $\Lc$. Then, we use Proposition \ref{prop:pointwise} below and a counting formula (Theorem \ref{theo:counting}) to employ an eigenvalue-counting argument, which establishes a one-to-one correspondence between the first $M$ non-trivial eigenvalues of $\sLn$ and those of $\Dr$. Combining the above, we show the convergence of the first $M$ non-trivial eigenvalues of $\Lb$ to those of $\Dr$.
	\begin{proof}
		We begin by comparing the eigenvalues of $\Lb$ and $\sLn$. Recall $\Lc$ in \eqref{eq:2.2.2}, which is a smoothed version of $\Lb$. We list the eigenvalues of $\Lc$ as $\lambda_1(\Lc)\le\lambda_2(\Lc)\le\cdots\le\lambda_n(\Lc)$.
		Applying Lemma \ref{lemma:S2L}, we have that whenever $\alpha_n$ in \eqref{eq:2.2.1} satisfies that $\alpha_n\gg n^{7/2}$, for sufficiently large $n$, with probability $1-\oo(1)$,
		\begin{align}
			\lambda_j(\Lb)=\lambda_{j}(\Lc)+\oo(1),\quad\text{for all } j=1,\cdots,n.\label{eq:5.2.0.0}
		\end{align}
		On the other hand, from Lemma \ref{lemma:S2Sn}, we have that $\lambda_j(\Lc)=\lambda_j(\sLn)$ for all $1\le j\le n$ whenever $\lambda_j(\sLn)<\frac{2m_0}{m_2r_n^2}$. Combining with \eqref{eq:5.2.0.0}, we have that with probability $1-\oo(1)$, for all $j$ such that $\lambda_j(\sLn)<\frac{2m_0}{m_2r_n^2}$,
		\begin{align}
			\lambda_j(\Lb)=\lambda_{j}(\sLn)+\oo(1).\label{eq:lbsln}
		\end{align}
		
		Next, we compare the eigenvalues of $\sLn$ and $\Dr$. Before we proceed, we choose some testing functions from $\Dr$. Let $0<\mu_1^*<\mu_2^*<\cdots$ be the distinct non-trivial eigenvalues of $\Dr$ with corresponding multiplicities $n_1,n_2,\cdots$ and eigenfunctions
		$$\varphi_{1,1},\cdots\varphi_{1,n_1},\quad\varphi_{2,1},\cdots\varphi_{2,n_2},\quad\cdots.$$
		We choose an integer $M_0$ such that $\sum_{j=1}^{M_0}n_j\ge M$ and use the associated eigenfunctions $\{\varphi_{j,l}\}_{1\le j\le M_0,1\le l\le n_j}$ as our testing functions. 
		Our proof relies on the following proposition.
		\begin{proposition}
			\label{prop:pointwise}
			Let $\varphi$ be an eigenfunction of $\Dr$ on $(\cF,\Rprod{\cdot}{\cdot})$. If the assumption \eqref{eq:2.2} holds, then for sufficiently large $n$, with probability $1-\oo(1)$,
			\begin{align*}
				\normf{(\Dr-\sLn)\varphi}=\oo(1),
			\end{align*}
			with $\normf{\cdot}$ as in \eqref{eq:fnorm}.
		\end{proposition}
		\noindent We defer the proof of Proposition \ref{prop:pointwise} to Section \ref{sec:prop}. By Proposition \ref{prop:pointwise}, with probability $1-\oo(1)$, for all $j=1,\cdots,M_0$ and $l=1,\cdots,n_j$,
		\begin{align}
			\normf{(\Dr-\sLn)\varphi_{j,l}}=\oo(1).\label{eq:5.2.0.1}
		\end{align}
		
		Now we proceed to apply the eigenvalue-counting argument. Combining \eqref{eq:2.1} with \eqref{eq:lbsln}, we have that for sufficiently large $n$, the trivial eigenvalues of $\sLn$ will all locate in the interval $[0,\delta]$. Next, we choose some sufficiently small positive constant $\vartheta\equiv\vartheta(n)=\oo(1)$, whose choice will be specified later. Therefore, to study the convergence of first $M$ non-trivial eigenvalues of $\sLn$ to those of $\Dr$, we can restrict our attention to $(\delta,\mu_{M_0}^*+\vartheta)$. Moreover, we partition $(\delta,\mu_{M_0}^*+\vartheta)$ into the following three possible types of intervals: 
		\begin{align*}
			(\mu_j^*-\vartheta,\mu_j^*+\vartheta)\text{ for } j=1,\cdots,M_0;\quad [\mu_{j}^*+\vartheta,\mu_{j+1}^*-\vartheta]\text{ for }j=1,\cdots,M_0-1;\quad \text{and }(\delta,\mu_1^*-\vartheta].
		\end{align*}
		We illustrate this partition in Figure \ref{fig:partition}.
		\begin{figure}[htbp]
			\centering
			\begin{tikzpicture}[x=1.3cm,y=1cm]
				
				\coordinate (zero) at (0,0);    
				\coordinate (d)    at (1.5,0);  
				\coordinate (m1m)  at (3.0,0);  
				\coordinate (m1)   at (3.6,0);  
				\coordinate (m1p)  at (4.2,0);  
				\coordinate (m2m)  at (5.3,0);  
				\coordinate (m2)   at (5.9,0);  
				\coordinate (m2p)  at (6.5,0);  
				\coordinate (mMm)  at (8.4,0);  
				\coordinate (mM)   at (9.0,0);  
				\coordinate (mMp)  at (9.6,0);  
				
				\draw[dashed,gray,shorten <=2pt,shorten >=2pt] (zero) -- (d);
				
				\draw (d) -- (m2p);
				
				\draw (mMm) -- (mMp);
				
				\draw[dashed,gray,shorten <=2pt,->] (mMp) -- (11,0) node[below right] {$\mathbb{R}$};
				
				\foreach \x/\lab in {
					0/{$0$},
					1.5/{$\delta$},
					3.0/{$\mu_1^*-\vartheta$},
					4.2/{$\mu_1^*+\vartheta$},
					5.3/{$\mu_2^*-\vartheta$},
					6.5/{$\mu_2^*+\vartheta$},
					8.4/{$\mu_{M_0}^*-\vartheta$},
					9.6/{$\mu_{M_0}^*+\vartheta$}
				}{
					\draw (\x,0.14) -- (\x,-0.14);
					\node[below,yshift=-2pt] at (\x,0) {\scriptsize \lab};
				}
				\foreach \P/\txt in {m1/{$\mu_1^*$}, m2/{$\mu_2^*$}, mM/{$\mu_{M_0}^*$}}{
					\draw (\P)+(0,0.14) -- (\P)+(0,-0.14);
					\node[above,yshift=2pt] at (\P) {\scriptsize \txt};
				}
				
				\draw[very thick, teal] (d) -- (m1m);              
				\draw[very thick, red]  (m1m) -- (m1p);            
				\draw[very thick, blue] (m1p) -- (m2m);            
				\draw[very thick, red]  (m2m) -- (m2p);            
				\draw[very thick, blue, dashed,shorten <=2pt,shorten >=2pt]
				(m2p) -- (mMm);                                
				\draw[very thick, red]  (mMm) -- (mMp);            
				
				\begin{scope}[shift={(0,-0.9)}]
					\draw[very thick, teal] (0,0) -- (0.9,0) node[right]{\scriptsize $(\delta,\mu_1^*-\vartheta]$};
					\draw[very thick, red]  (3,0) -- (3.9,0) node[right]{\scriptsize $(\mu_j^*-\vartheta,\mu_j^*+\vartheta)$};
					\draw[very thick, blue] (6,0) -- (6.9,0) node[right]{\scriptsize $[\mu_j^*+\vartheta,\mu_{j+1}^*-\vartheta]$};
				\end{scope}
				
			\end{tikzpicture}
			\caption{Partition of $(\delta,\mu_{M_0}^*+\vartheta)$. $[0,\delta]$ is the trivial region and will not be studied by our eigenvalue-counting argument, $(\mu_j^*-\vartheta,\mu_j^*+\vartheta)$'s are the counting regions, and $[\mu_{j}^*+\vartheta,\mu_{j+1}^*-\vartheta]$'s and $(\delta,\mu_1^*-\vartheta]$ are the empty regions.}
			\label{fig:partition}
		\end{figure}
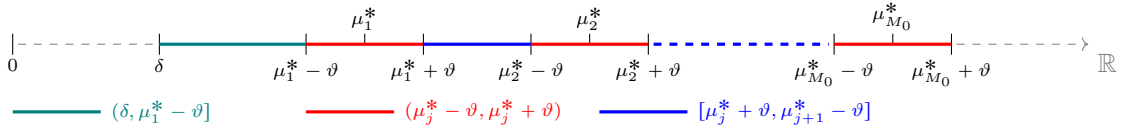
		The rest of the proof is carried out in three steps. In Step one, we show that each $(\mu_j^*-\vartheta,\mu_j^*+\vartheta)$ contains at least one eigenvalue of $\sLn$. In Step two, we show that exactly $n_j$ eigenvalues of $\sLn$ (counting multiplicities) lie in $(\mu_j^*-\vartheta,\mu_j^*+\vartheta)$. In Step three, we show that none of the eigenvalues of $\sLn$ lie in the other two types of empty regions $[\mu_{j}^*+\vartheta,\mu_{j+1}^*-\vartheta]$ and $(\delta,\mu_1^*-\vartheta]$.
		\par\vspace{.5\baselineskip}
		\noindent\textbf{Step one.}\;
		Denote 
		$$\operatorname{dist}(\mu_j^*,\sigma(\sLn))\coloneq\min_{l}|\mu_j^*-\lambda_{l}(\sLn)|.$$
		Now consider the following two cases. 
		When $\operatorname{dist}(\mu_j^*,\sigma(\sLn))=0$, then $\mu_j^*$ itself is an eigenvalue of $\sLn$, so the interval $(\mu_j^*-\vartheta,\mu_j^*+\vartheta)$ contains at least one eigenvalue of $\sLn$. On the other hand, when $\operatorname{dist}(\mu_j^*,\sigma(\sLn))>0$, we have that
		\begin{align}
			\norm{R(\mu_j^*,\sLn)}=1/\operatorname{dist}(\mu_j^*,\sigma(\sLn)),\label{eq:5.2.0.2}
		\end{align}
		where $R(z,\sLn)=(\sLn-z\mathcal{I})^{-1}$ is the resolvent of $\sLn$. Since $\varphi_{j,1}=R(\mu_j^*,\sLn)(\sLn-\mu_j^*)\varphi_{j,1}$, we have
		\begin{align}
			1=\normf{\varphi_{j,1}}\le\norm{R(\mu_j^*,\sLn)}\normf{(\sLn-\mu_j^*)\varphi_{j,1}},\label{eq:5.2.0.3}
		\end{align}
		where the last step follows from Cauchy-Schwarz inequality. Since $\normb{R(\mu_j^*,\sLn)}$ is strictly positive by \eqref{eq:5.2.0.2}, we have that
		\begin{align*}
			\operatorname{dist}(\mu_j^*,\sigma(\sLn))=\frac{1}{\normb{R(\mu_j^*,\sLn)}}\le\normf{(\sLn-\mu_j^*)\varphi_{j,1}}=\normf{(\sLn-\Dr)\varphi_{j,1}},
		\end{align*}
		where in the second step we used \eqref{eq:5.2.0.3} and the last step is from the fact that $\Dr\varphi_{j,1}=\mu_j^*\varphi_{j,1}$. By \eqref{eq:5.2.0.1}, $\normf{(\sLn-\Dr)\varphi_{j,1}}=\oo(1)$ with probability $1-\oo(1)$. Hence, for sufficiently large $n$, $\vartheta$ can be chosen such that $\operatorname{dist}(\mu_j^*,\sigma(\sLn))<\vartheta$.
		Combining the two cases, we conclude that with probability $1-\oo(1)$, for $j=1,\cdots,M_0$, there always exists at least one eigenvalue of $\sLn$ inside $(\mu_j^*-\vartheta,\mu_j^*+\vartheta)$.
		\par\vspace{.5\baselineskip}
		\noindent\textbf{Step two.}\;
		Without loss of generality, we assume there are $n_j+1$ eigenvalues of $\sLn$ inside of $(\mu_j^*-\vartheta,\mu_j^*+\vartheta)$. Define a positively oriented contour around $\mu_j^*$ by
		\begin{align}
			\Gamma_j\coloneq\{z\in\C:|z-\mu_j^*|=\vartheta\}.\label{eq:diskj}
		\end{align}
		Observe that $\Dr$ and $\sLn$ are defined on $\cF$ for which $\{\varphi_{j,l}\}_{j\ge0,1\le l\le n_j}$ is an orthonormal basis where $n_0=1$ and $\varphi_{0,1}\equiv\prod_{i=1}^d\left(\frac{1}{\sqrt{\pi}\sigma_i}\right)^{1/2}$. Then by Theorem \ref{theo:counting}, we have that
		\begin{align}
			n_j&=-\frac{1}{2\pi\operatorname{i}}\oint_{\Gamma_j}\operatorname{Tr}(R(z,\Dr))\dd z\equiv\sum_{s=0}^\infty\sum_{l=1}^{n_s}-\frac{1}{2\pi\operatorname{i}}\oint_{\Gamma_j}\Rprod{\varphi_{s,l}}{R(z,\Dr)\varphi_{s,l}}\dd z,\label{eq:5.2.0.4}\\
			n_j+1&=-\frac{1}{2\pi\operatorname{i}}\oint_{\Gamma_j}\operatorname{Tr}(R(z,\sLn))\dd z\equiv\sum_{s=0}^\infty\sum_{l=1}^{n_s}-\frac{1}{2\pi\operatorname{i}}\oint_{\Gamma_j}\Rprod{\varphi_{s,l}}{R(z,\sLn)\varphi_{s,l}}\dd z,\label{eq:5.2.0.5}
		\end{align}
		where the second equations in \eqref{eq:5.2.0.4} and \eqref{eq:5.2.0.5} are from \eqref{eq:trace}. Since $\varphi_{s,l}$'s are eigenfunctions of $\Dr$, for $s\ne j$ and $1\le l\le n_s$,
		\begin{align}
			-\frac{1}{2\pi\operatorname{i}}\oint_{\Gamma_j}\Rprod{\varphi_{s,l}}{R(z,\Dr)\varphi_{s,l}}\dd z=-\frac{1}{2\pi\operatorname{i}}\oint_{\Gamma_j} \frac{1}{\mu_s^*-z}\dd z=0,\label{eq:5.2.0.7}
		\end{align}
		where in the last step we used Theorem \ref{them:residue} and the fact that $\mu_j^*\notin \Gamma_j\cup\operatorname{Int}(\Gamma_j)$. Combining with \eqref{eq:5.2.0.4}, we have that 
		\begin{align}
			n_j=\sum_{l=1}^{n_j}-\frac{1}{2\pi\operatorname{i}}\oint_{\Gamma_j}\Rprod{\varphi_{j,l}}{R(z,\Dr)\varphi_{j,l}}\dd z.\label{eq:5.2.0.6}
		\end{align}
		On the other hand, for $1\le l\le n_j$,
		\begin{align}
			\left|\frac{1}{2\pi\operatorname{i}}\oint_{\Gamma_j}\Rprod{\varphi_{j,l}}{R(z,\Dr)\varphi_{j,l}}-\Rprod{\varphi_{j,l}}{R(z,\sLn)\varphi_{j,l}}\dd z\right|&\le\frac{1}{2\pi}\oint_{\Gamma_j}\normf{\varphi_{j,l}}\normf{(R(z,\Dr)-R(z,\sLn)\varphi_{j,l})}\dd z\notag\\
			&=\frac{1}{2\pi}\oint_{\Gamma_j}\normf{R(z,\sLn)(\sLn-\Dr)R(z,\Dr)\varphi_{j,l}}\dd z\notag\\
			&\le\frac{1}{2\pi}\oint_{\Gamma_j}\norm{R(z,\sLn)}\normbf{(\Dr-\sLn)\frac{1}{|z-\mu_j^*|}\varphi_{j,l}}\dd z\notag\\
			&=\frac{\normf{(\Dr-\sLn)\varphi_{j,l}}}{2\pi\vartheta}\oint_{\Gamma_j}\norm{R(z,\sLn)}\dd z\notag\\
			&\le\normf{(\Dr-\sLn)\varphi_{j,l}}\max_{z\in\Gamma_j}\norm{R(z,\sLn)},\notag
		\end{align}
		where in the first step we used Cauchy-Schwarz inequality, in the second step we used \eqref{eq:resolvent}, in the third step we used the fact that $R(z,\Dr)\varphi_{j,l}=\varphi_{j,l}/(\mu_j^*-z)$ and in the fourth step we used \eqref{eq:diskj}. We choose some $\vartheta=\oo(1)$ so that $\normf{(\Dr-\sLn)\varphi_{j,l}}\max_{z\in\Gamma_j}\norm{R(z,\sLn)}=\oo(1)$. This is allowed since by \eqref{eq:5.2.0.1}, with probability $1-\oo(1)$, $\normf{(\Dr-\sLn)\varphi_{j,l}}=\oo(1)$ and it is independent of $\vartheta$, and the second term $\max_{z\in\Gamma_j}\norm{R(z,\sLn)}$ is $\vartheta$-dependent and independent of the first one. 
		Therefore, we have 
		$$\sum_{l=1}^{n_j}-\frac{1}{2\pi\operatorname{i}}\oint_{\Gamma_j}\Rprod{\varphi_{j,l}}{R(z,\sLn)\varphi_{j,l}}\dd z=\sum_{l=1}^{n_j}-\frac{1}{2\pi\operatorname{i}}\oint_{\Gamma_j}\Rprod{\varphi_{j,l}}{R(z,\Dr)\varphi_{j,l}}\dd z+\oo(1)=n_j+\oo(1),$$
		where in the last step we used \eqref{eq:5.2.0.6}.
		Combining with \eqref{eq:5.2.0.5}, we have that there exists some $j^*\ne j$ and $1\le l^*\le n_{j^*}$ such that 
		\begin{align}
			-\frac{1}{2\pi\operatorname{i}}\oint_{\Gamma_j}\Rprod{\varphi_{j^*,l^*}}{R(z,\sLn)\varphi_{j^*,l^*}}\dd z>\kappa,\label{eq:5.2.0.8}
		\end{align}
		for some $\kappa>0$. 
		Likewise, we have 
		\begin{align}
			&\left|\frac{1}{2\pi\operatorname{i}}\oint_{\Gamma_j}\Rprod{\varphi_{j^*,l^*}}{R(z,\Dr)\varphi_{j^*,l^*}}-\Rprod{\varphi_{j^*,l^*}}{R(z,\sLn)\varphi_{j^*,l^*}}\dd z\right|\notag\\
			\le&\frac{1}{2\pi}\oint_{\Gamma_j}\normf{\varphi_{j^*,l^*}}\normf{(R(z,\sLn)-R(z,\Dr))\varphi_{j^*,l^*}}\dd z\notag\\
			=&\frac{1}{2\pi}\oint_{\Gamma_j}\normf{R(z,\sLn)(\Dr-\sLn)R(z,\Dr)\varphi_{j^*,l^*}}\dd z\le\frac{1}{2\pi}\oint_{\Gamma_j}\norm{R(z,\sLn)}\normbf{(\Dr-\sLn)\frac{1}{|z-\mu_{j^*}^*|}\varphi_{j^*,l^*}}\dd z\notag\\
			\le&\frac{\normf{(\Dr-\sLn)\varphi_{j^*,l^*}}}{\pi|\mu_{j^*}^*-\mu_{j}^*|}\oint_{\Gamma_j}\norm{R(z,\sLn)}\dd z\lesssim\normf{(\Dr-\sLn)\varphi_{j^*,l^*}}\vartheta\max_{z\in\Gamma_j}\norm{R(z,\sLn)},\label{eq:5.2.0.11}
		\end{align}
		where in the first step we used Cauchy-Schwarz inequality, in the second step we used \eqref{eq:resolvent} and the last but one step is from the fact that $|z-\mu_{j^*}^*|\ge|\mu_{j^*}^*-\mu_{j}^*|/2$ for all $z\in\Gamma_j$. Similar to the above argument, combining with \eqref{eq:5.2.0.7}, we have with probability $1-\oo(1)$,
		\begin{align*}
			-\frac{1}{2\pi\operatorname{i}}\oint_{\Gamma_j}\Rprod{\varphi_{j^*,l^*}}{R(z,\sLn)\varphi_{j^*,l^*}}=\oo(1),
		\end{align*}
		which contradicts \eqref{eq:5.2.0.8}. Since $ -\frac{1}{2\pi\operatorname{i}}\oint_{\Gamma_j}\operatorname{Tr}(R(z,\sLn))\dd z$ is an integer by Theorem \ref{theo:counting}, it can only be $n_j$ by according to the contradiction argument. Consequently, with probability $1-\oo(1)$, there are exactly $n_j$ eigenvalues of $\sLn$ inside $(\mu_j^*-\vartheta,\mu_j^*+\vartheta)$, for all $j=1,\cdots,M_0$.
		\par\vspace{.5\baselineskip}
		\noindent\textbf{Step three.}\;
		For each $[\mu_j^*+\vartheta,\mu_{j+1}^*-\vartheta]$, we pick a positively oriented contour 
		\begin{align}
			\Gamma_{j,j+1}=\{z\in\C:|z-(\mu_j^*+\mu_{j+1}^*)/2|=(\mu_{j+1}^*-\mu_{j}^*)/2-\vartheta\}.\label{eq:gammajj+1}
		\end{align}
		Since $[\mu_j^*+\vartheta,\mu_{j+1}^*-\vartheta]$ contains no eigenvalues of $\Dr$, we have that from the same procedure in \eqref{eq:5.2.0.7}, for all $j\ge0$ and $1\le l\le n_j$,
		\begin{align}
			-\frac{1}{2\pi\operatorname{i}}\oint_{\Gamma_{j,j+1}}\Rprod{\varphi_{s,l}}{R(z,\Dr)\varphi_{s,l}}=0.\label{eq:5.2.0.9}
		\end{align}
		Here we also prove by contradiction. We assume that there are at least one eigenvalue of $\sLn$ inside $[\mu_j^*+\vartheta,\mu_{j+1}^*-\vartheta]$. By Theorem \ref{theo:counting}, this indicates $-\frac{1}{2\pi\operatorname{i}}\oint_{\Gamma_{j,j+1}}\operatorname{Tr}(R(z,\sLn))\dd z\ge1$. Therefore, there exists some eigenfunction $\varphi^*$ of $\Dr$ associated with $\mu^*$ such that 
		\begin{align}
			-\frac{1}{2\pi\operatorname{i}}\oint_{\Gamma_{j,j+1}}\Rprod{\varphi^*}{R(z,\sLn)\varphi^*}\dd z>\kappa_1,\label{eq:5.2.0.10}
		\end{align}
		for some $\kappa_1>0$. Using similar procedure as in \eqref{eq:5.2.0.11}, we have that 
		\begin{align}
			\left|\frac{1}{2\pi\operatorname{i}}\oint_{\Gamma_{j,j+1}}\Rprod{\varphi^*}{R(z,\Dr)\varphi^*}-\Rprod{\varphi^*}{R(z,\sLn)\varphi^*}\dd z\right|&\le\frac{1}{2\pi}\oint_{\Gamma_{j,j+1}}\norm{R(z,\sLn)}\normbf{(\Dr-\sLn)\frac{1}{|z-\mu^*|}\varphi^*}\dd z\notag\\
			&\le\frac{\normf{(\Dr-\sLn)\varphi^*}}{2\pi\vartheta}\oint_{\Gamma_{j,j+1}}\norm{R(z,\sLn)}\dd z\notag\\
			&\lesssim\normf{(\Dr-\sLn)\varphi^*}\frac{\max_{z\in\Gamma_{j,j+1}}\norm{R(z,\sLn)}}{\vartheta},\notag
		\end{align}
		where in the second step we used \eqref{eq:gammajj+1}.
		Similarly, combining with Proposition \ref{prop:pointwise} and \eqref{eq:5.2.0.9}, we have that with probability $1-\oo(1)$, 
		$$-\frac{1}{2\pi\operatorname{i}}\oint_{\Gamma_{j,j+1}}\Rprod{\varphi^*}{R(z,\sLn)\varphi^*}\dd z=\oo(1),$$
		which contradicts \eqref{eq:5.2.0.10}. Consequently, with probability $1-\oo(1)$, 
		$$-\frac{1}{2\pi\operatorname{i}}\oint_{\Gamma_j}\operatorname{Tr}(R(z,\sLn))\dd z=0,$$
		i.e., there are no eigenvalues of $\sLn$ that lie in $[\mu_j^*+\vartheta,\mu_{j+1}^*-\vartheta]$. Using similar argument to show that no eigenvalues lie in $[\mu_j^*+\vartheta,\mu_{j+1}^*-\vartheta]$ above, we can also have that with probability $1-\oo(1)$, there are no eigenvalues of $\sLn$ that lie in $(\delta,\mu_1^*-\vartheta]$.
		
		Combining all the above arguments, we have that for $l=1,\cdots,K_0$, $\lambda_l(\sLn)\le\delta$, and for $j=1,\cdots,M_0$, and when $K_0+\sum_{s=1}^{j-1} n_s+1\le l\le K_0+\sum_{s=1}^{j} n_s$, with probability $1-\oo(1)$,
		$$|\lambda_l(\sLn)-\mu_j^*|<\vartheta,$$
		where we recall $\sum_{j=1}^{M_0}n_j\ge M$. Since $\vartheta$ is chosen such that $\vartheta=\oo(1)$, we have that $\lambda_l(\sLn)=\mu_j^*+\oop(1)$. Note by Remark \ref{rmk:eigen}, $\mu_1(\Dr)=0$ and $\mu_2(\Dr)>\delta$. Consequently, with probability $1-\oo(1)$, for $k=1,\cdots,M$,
		$$\lambda_{K_0+k}(\sLn)=\mu_{k+1}(\Dr)+\oo(1).$$
		Combining with \eqref{eq:lbsln}, we conclude \eqref{eq:main} and complete our proof.
	\end{proof}
	
	\subsection{Proof of Proposition \ref{prop:pointwise}}\label{sec:prop}
	In this section, we prove Proposition \ref{prop:pointwise}. We first introduce notation for multi-index products. For $\bm{\sigma}=(\sigma_1,\cdots,\sigma_d)^\top\in\R_+^d,\bm{\alpha}=(\alpha_1,\cdots,\alpha_d)^\top\in\N^d,\bm{z}=(z_1,\cdots,z_d)^\top\in\R^d$, denote $|\bm{\alpha}|\coloneq\sum_{i=1}^d\alpha_i$, $\bm{\alpha}!\coloneq\prod_{i=1}^d\alpha_i!$, $\bm{z}^{\bm{\alpha}}\coloneq\prod_{i=1}^dz_i^{\alpha_i}$ and $\mathsf{H}_{\bm{\alpha}}(\xb/\bm{\sigma})\coloneq\prod_{i=1}^d H_{\alpha_i}(x_i/\sigma_i)$, where $H_m(\cdot)$ is the $m$-th physicist Hermite polynomial. For $\bm{s}=(s_1,\ldots,s_d)^\top,\bm{k}=(k_1,\ldots,k_d)^\top\in\N^d$, we write $\bm{s}\le\bm{k}$ if $s_i\le k_i$ for all $i=1,\ldots,d$.
	
	Recall $\mathcal{T}_n$ in \eqref{eq:2.2.5}. The proof of Proposition \ref{prop:pointwise} splits into two parts, the bias part and the variance part, based on the following decomposition
	$$\normf{(\Dr-\sLn)\varphi}\le\normf{(\Dr-\mathcal{T}_n)\varphi}+\normf{(\mathcal{T}_n-\sLn)\varphi},$$
	where the first part can be regarded as a deterministic bias part and the second part can be regarded as a random variance part. For the bias part, we show that $\normf{(\Dr-\mathcal{T}_n)\varphi}=\OO(r_n^2)$, whose proof relies on the key representation (see Lemma \ref{lemma:expansion}) that 
	\begin{align}
		\mathcal{T}_n \varphi(\xb)=\frac{2C_\varphi}{m_2}\sum_{\substack{\bm{0}\le\bm{s}\le\bm{k}\\ \bm{s}\ne\bm{k}}}\sum_{\bm{\alpha}\in\N^d}\frac{(-1)^{|\bm{\alpha}|}}{\bm{\alpha}!2^{|\bm{\alpha}|}}\mathsf{C}_{\bm{k},\bm{s},\bm{\alpha}}\bm{\sigma}^{-\bm{\gamma}}r_n^{|\bm{\gamma}|-2}\Lambda_{\bm{\gamma}}(r_n;\bm{\sigma},g),\label{eq:5.3.0.1}
	\end{align}
	where $\bm{\gamma}=\bm{k}-\bm{s}+\bm{\alpha}$, $C_\varphi$ is some constant and 
	\begin{align}
		\Lambda_{\bm{\gamma}}(r_n;\bm{\sigma},g)&=\int_{\normp{\zb}\le1}g(\normp{\zb})\exp\left(-\sum_{i=1}^d\frac{r_n^2z_i^2}{4\sigma_i^2}\right)\bm{z}^{\bm{\gamma}}\dd\zb,\label{eq:5.3.0.2}\\
		\mathsf{C}_{\bm{k},\bm{s},\bm{\alpha}}&=-\mathsf{H}_{\bm{s}}(\xb/\bm{\sigma})\mathsf{H}_{\bm{\alpha}}(\xb/\bm{\sigma})\prod_{i=1}^d\binom{k_i}{s_i}2^{k_i-s_i}.\label{eq:5.3.0.3}
	\end{align}
	Such a representation allows us to control $\normf{(\mathcal{T}_n-\Dr)\varphi}$ by analyzing the summation in \eqref{eq:5.3.0.1} term by term.
	
	For the variance part, we will show that $\normf{(\mathcal{T}_n-\sLn)\varphi}=\oop(1)$. In the actual proof, we will introduce an intermediate quantity $\widetilde{\mathcal{T}}_n$ as follows,
	\begin{align}
		\widetilde{\mathcal{T}}_n \varphi(\xb)\coloneq\frac{2m_0}{m_2r_n^2}\frac{\int_{\R^d}\bm{1}(\normp{\xb-\yb}\le r_n)g(\normp{\xb-\yb}/r_n)(\varphi(\xb)-\varphi(\yb))\varrho(\yb)\dd \yb}{\int_{\R^d}\bm{1}(\normp{\xb-\yb}\le r_n)g(\normp{\xb-\yb}/r_n)\varrho(\yb)\dd \yb},\label{eq:defttn}
	\end{align}
	which is a finer approximation of $\mathcal{T}_n$. It suffices to control $\normf{(\widetilde{\mathcal{T}}_n-\mathcal{T}_n)\varphi}$ and $\normf{(\sL_n-\widetilde{\mathcal{T}}_n)\varphi}$ using a truncation argument.
	\begin{proof}
		Recall from Theorem \ref{theo:eigen},
		\begin{align}
			\varphi(\xb)=C_\varphi\prod_{i=1}^d H_{k_i}(x_i/\sigma_i),\quad k_1,\cdots,k_d\in\N,\label{eq:varphi}
		\end{align}
		where $C_{\varphi}$ is some constant and $\sigma_1,\cdots,\sigma_d$ are as in \eqref{eq:2.1}.
		\par\vspace{.5\baselineskip}
		\noindent{\bf{Control of the bias part.}}\; 
		Recall \eqref{eq:5.3.0.1}, \eqref{eq:5.3.0.2} and \eqref{eq:5.3.0.3}. By Lemma \ref{lemma:expansion}, $\mathcal{T}_n\varphi(\xb)$ can be represented as 
		\begin{align}
			\mathcal{T}_n \varphi(\xb)=\frac{2C_\varphi}{m_2}\sum_{\substack{\bm{0}\le\bm{s}\le\bm{k}\\ \bm{s}\ne\bm{k}}}\sum_{\bm{\alpha}\in\N^d}\frac{(-1)^{|\bm{\alpha}|}}{\bm{\alpha}!2^{|\bm{\alpha}|}}\mathsf{C}_{\bm{k},\bm{s},\bm{\alpha}}\bm{\sigma}^{-\bm{\gamma}}r_n^{|\bm{\gamma}|-2}\Lambda_{\bm{\gamma}}(r_n;\bm{\sigma},g),\notag
		\end{align}
		where $\bm{\gamma}=\bm{k}-\bm{s}+\bm{\alpha}$. We point out that by symmetry, for \eqref{eq:5.3.0.2}, $\Lambda_{\bm{\gamma}}(r_n;\bm{\sigma},g)=0$ whenever there exists some odd $\gamma_i$, $i=1,\cdots,d$, and that for all $\bm{\gamma}\in\N^d$,
		\begin{align}
			|\Lambda_{\bm{\gamma}}(r_n;\bm{\sigma},g)|\le\int_{\normp{\zb}\le1}g(\normp{\zb})\dd\zb\lesssim m_0,\label{eq:5.3.0.4}
		\end{align}
		where in the first step we used the fact that $|z_i|\le\normp{\zb}\le 1$ for all $i=1,\cdots,d$ and in the second step we used \eqref{eq:mj} and our assumption on $g(\cdot)$ in Theorem \ref{theo:normald=1}. 
		Moreover, observe that when $|\bm{\gamma}|\ge4$, $r_n^{|\bm{\gamma}|-2}\Lambda_{\bm{\gamma}}(r_n;\bm{\sigma},g)=\OO(r_n^2)$. Since $|\bm{\gamma}|>0$ and all terms with odd $|\bm{\gamma}|$ yield zero, we can decompose $\mathcal{T}_n=\mathsf{M}+\mathsf{R}$, where the main part $\mathsf{M}$ collects all terms with $|\bm{\gamma}|=2$ and the remainder part $\mathsf{R}$ collects all terms with $|\bm{\gamma}|\ge4$, i.e.,
		\begin{align}
			\mathsf{M}&\coloneq \frac{2C_\varphi}{m_2}\sum_{\substack{\bm{0}\le\bm{s}\le\bm{k},\bm{\alpha}\in\N^d\\ \bm{s}\ne\bm{k},|\bm{\gamma}|=2}}\frac{(-1)^{|\bm{\alpha}|}}{\bm{\alpha}!2^{|\bm{\alpha}|}}\mathsf{C}_{\bm{k},\bm{s},\bm{\alpha}}\bm{\sigma}^{-\bm{\gamma}}r_n^{|\bm{\gamma}|-2}\Lambda_{\bm{\gamma}}(r_n;\bm{\sigma},g),\label{eq:M}\\
			\mathsf{R}&\coloneq \frac{2C_\varphi}{m_2}\sum_{\substack{\bm{0}\le\bm{s}\le\bm{k},\bm{\alpha}\in\N^d\\ \bm{s}\ne\bm{k},|\bm{\gamma}|\ge4}}\frac{(-1)^{|\bm{\alpha}|}}{\bm{\alpha}!2^{|\bm{\alpha}|}}\mathsf{C}_{\bm{k},\bm{s},\bm{\alpha}}\bm{\sigma}^{-\bm{\gamma}}r_n^{|\bm{\gamma}|-2}\Lambda_{\bm{\gamma}}(r_n;\bm{\sigma},g).\label{eq:R}
		\end{align}
		For the main part $\mathsf{M}$, $\Lambda_{\bm{\gamma}}(r_n;\bm{\sigma},g)\ne0$ only when there is a unique index $i$ such that $\gamma_i=2$ and $\gamma_j=0$ for all $j\ne i$. In that case,
		\begin{align}
			\Lambda_{\bm{\gamma}}(r_n;\bm{\sigma},g)&=\int_{\normp{\zb}\le1}g(\normp{\zb})\exp\left(-\sum_{i=1}^d\frac{r_n^2z_i^2}{4\sigma_i^2}\right)z_i^2\dd\zb\notag\\
			&=\int_{\normp{\zb}\le1}g(\normp{\zb})(1+\OO(r_n^2))z_i^2\dd\zb=m_2(1+\OO(r_n^2)),\label{eq:5.3.0.5}
		\end{align}
		where in the second step we used the fact that $\exp(-t)=1+\OO(t)$ uniformly for $0\le t\le 1$, and in the last step we used \eqref{eq:mj}.
		The condition $\gamma_i=2$ arises either from $s_i=k_i-1,\alpha_i=1$ or from $s_i=k_i-2,\alpha_i=0$. Then we proceed to discuss these two cases. When $s_i=k_i-1,\alpha_i=1$,
		\begin{align}
			\mathsf{C}_{\bm{k},\bm{s},\bm{\alpha}}=-2x_i\left(\prod_{j\ne i}H_{k_j}(x_j/\sigma_j)\right)2k_iH_{k_i-1}(x_i/\sigma_i)/\sigma_i=-2x_i\frac{\partial \mathsf{H}_{\bm{k}}(\xb/\bm{\sigma})}{\partial x_i},\label{eq:5.3.0.6}
		\end{align}
		where in the first step we used the fact that $H_0\equiv1,H_1(x)=2x$ and in the second step we used the identity \eqref{eq:C3.2.1}. When $s_i=k_i-2,\alpha_i=0$,
		\begin{align}
			\mathsf{C}_{\bm{k},\bm{s},\bm{\alpha}}=-H_{k_i-2}(x_i/\sigma_i)\left(\prod_{j\ne i}H_{k_j}(x_j/\sigma_j)\right)2k_i(k_i-1)=-\frac{\sigma_i^2}{2}\frac{\partial^2 \mathsf{H}_{\bm{k}}(\xb/\bm{\sigma})}{\partial x_i^2},\label{eq:5.3.0.7}
		\end{align}
		where in the last step we applied \eqref{eq:C3.2.1} twice. Recall $\mathsf{M}$ in \eqref{eq:M}. Combining \eqref{eq:5.3.0.5}, \eqref{eq:5.3.0.6} and \eqref{eq:5.3.0.7}, we have that 
		\begin{align}
			|\mathsf{M}-\Dr\varphi(\xb)|&=\left|\frac{2C_\varphi}{m_2}\sum_{i=1}^d \left(\frac{1}{2}2x_i\frac{\partial \mathsf{H}_{\bm{k}}(\xb/\bm{\sigma})}{\partial x_i}\sigma_i^{-2}m_2(1+\OO(r_n^2))-\frac{\sigma_i^2}{2}\frac{\partial^2 \mathsf{H}_{\bm{k}}(\xb/\bm{\sigma})}{\partial x_i^2}\sigma_i^{-2}m_2(1+\OO(r_n^2))\right)-\Dr\varphi(\xb)\right|\notag\\
			&=\OO(r_n^2)\widetilde{\varphi}(\xb),\label{eq:5.3.0.8}
		\end{align}
		where in the last step we used \eqref{eq:4.2} and 
		\begin{align}
			\widetilde{\varphi}(\xb)=C_\varphi\sum_{i=1}^d\left(\left|\frac{\partial^2\mathsf{H}_{\bm{k}}}{\partial x_i^2}\right|+\left|\frac{2x_i}{\sigma_i^2}\frac{\partial \mathsf{H}_{\bm{k}}}{\partial x_i}\right|\right).\notag
		\end{align}
		We point out that $\widetilde{\varphi}$ is the summation of the absolute values of $2d$ polynomials. Therefore, by the triangle inequality, $\normf{\widetilde{\varphi}}<\infty$.
		
		We now control the remainder $\mathsf{R}$ with $|\bm{\gamma}|\ge4$. Recall \eqref{eq:R}. We have that
		\begin{align}
			|\mathsf{R}|&\lesssim r_n^2\sum_{\bm{0}\le\bm{s}\le\bm{k}}\sum_{\bm{\alpha}\in\N^d}\frac{1}{\bm{\alpha}!2^{|\bm{\alpha}|}}|\mathsf{C}_{\bm{k},\bm{s},\bm{\alpha}}|\bm{\sigma}^{-\bm{\gamma}}\notag\\
			&\le r_n^2\prod_{i=1}^{d}\left\{\sum_{s_i=0}^{k_i}\sum_{\alpha_i=0}^\infty\frac{1}{\alpha_i!2^{\alpha_i}}\left|H_{s_i}(x_i/\sigma_i)H_{\alpha_i}(x_i/\sigma_i)\right|\binom{k_i}{s_i}2^{k_i-s_i}\sigma_i^{-(k_i-s_i+\alpha_i)}\right\},\notag
		\end{align}
		where in the first step we used \eqref{eq:5.3.0.4} and in the second step we used the factorization structure of \eqref{eq:5.3.0.2} and \eqref{eq:5.3.0.3}. Recall $\normf{\cdot}$ in \eqref{eq:fnorm} and $\norm{\cdot}_{\cF_i}$ in \eqref{eq:4.9.1}.
		Then we have 
		\begin{align}
			\normf{\mathsf{R}}&\lesssim r_n^2\prod_{i=1}^d\normbb{\sum_{s_i=0}^{k_i}\sum_{\alpha_i=0}^\infty\frac{1}{\alpha_i!2^{\alpha_i}}\left|H_{s_i}(x_i/\sigma_i)H_{\alpha_i}(x_i/\sigma_i)\right|\binom{k_i}{s_i}2^{k_i-s_i}\sigma_i^{-(k_i-s_i+\alpha_i)}}_{\mathcal{F}_i}\notag\\
			&\le r_n^2\prod_{i=1}^d\left\{\sum_{s_i=0}^{k_i}\sum_{\alpha_i=0}^\infty\frac{1}{\alpha_i!2^{\alpha_i}}\binom{k_i}{s_i}2^{k_i-s_i}\sigma_i^{-(k_i-s_i+\alpha_i)}\normb{H_{s_i}(x_i/\sigma_i)H_{\alpha_i}(x_i/\sigma_i)}_{\mathcal{F}_i}\right\}\notag\\
			&\lesssim r_n^2\prod_{i=1}^d\left\{\sum_{s_i=0}^{k_i}\sum_{\alpha_i=0}^\infty\frac{1}{\alpha_i!2^{\alpha_i}}\sqrt{(s_i+1)!\alpha_i!}(2\sqrt{2})^{s_i+\alpha_i}\binom{k_i}{s_i}2^{k_i-s_i}\sigma_i^{-\alpha_i}\right\}\notag\\
			&\le r_n^2\prod_{i=1}^d\left\{\sum_{s_i=0}^{k_i}\binom{k_i}{s_i}2^{k_i-s_i}(2\sqrt{2})^{s_i}\sqrt{(s_i+1)!}\sum_{\alpha_i=0}^\infty\frac{(\sqrt{2}/\sigma_i)^{\alpha_i}}{\sqrt{\alpha_i!}}\right\}=\OO(r_n^2),\label{eq:5.3.0.9}
		\end{align}
		where in the second step we used the triangle inequality, in the third step we used Lemma \ref{lemma:hermitenielsen} and the fact that $\sigma_i^{-(k_i-s_i)}$ is bounded and in the last step we used the fact that by ratio test,
		$\sum_{\alpha_i=0}^\infty(\sqrt{2}/\sigma_i)^{\alpha_i}/\sqrt{\alpha_i!}<\infty$,
		and that $k_1,\cdots,k_d$ are fixed integers.
		
		Combining \eqref{eq:5.3.0.8} and \eqref{eq:5.3.0.9}, we have that 
		\begin{align}
			\normf{(\mathcal{T}_n-\Dr)\varphi}&=\normf{\mathsf{M}-\Dr\varphi+\mathsf{R}}\le\normf{\mathsf{M}-\Dr\varphi}+\normf{\mathsf{R}}\le\OO(r_n^2)\normf{\widetilde{\varphi}}+\OO(r_n^2)=\OO(r_n^2),\label{eq:5.3.0.10}
		\end{align}
		which completes our proof of the bias part.
		\par\vspace{.5\baselineskip}
		\noindent{\bf{Control of the variance part.}}\; Set the bulk region
		\begin{align}
			B_1\coloneq\prod_{i=1}^d\left[-\sigma_i\sqrt{2(1+\beta)\log n},\sigma_i\sqrt{2(1+\beta)\log n}\right],\label{eq:bulk}
		\end{align}
		and denote $\bm{1}_{B_1}(\xb)=\bm{1}(\xb\in B_1)$, where $\beta>0$ is any fixed small constant. Recall \eqref{eq:defttn}.
		By the triangle inequality, it suffices to show $\normf{(\sLn-\widetilde{\mathcal{T}}_n)\varphi}=\oop(1)$ and $\normf{(\mathcal{T}_n-\widetilde{\mathcal{T}}_n)\varphi}=\oo(1)$. We start with the proof of the first term. Recall $\mathsf{s}(\cdot)$ and $g^*(\cdot)$ as in \eqref{eq:2.2.1}.
		For $j=1,\cdots,n$, set
		\begin{align}
			w_j&\coloneq\mathsf{s}\left(\alpha_n(r_n^2-\normp{\xb-\xb_j}^2)\right)g^*(\normp{\xb-\xb_j}/r_n)(\varphi(\xb)-\varphi(\xb_j)),\label{eq:wj}\\
			v_j&\coloneq\mathsf{s}\left(\alpha_n(r_n^2-\normp{\xb-\xb_j}^2)\right)g^*(\normp{\xb-\xb_j}/r_n).\label{eq:vj}
		\end{align}
		For convenience of notation, denote 
		\begin{align}
			N_1\coloneq\frac{1}{n}\sum_{j=1}^n w_j&,\quad N_2\coloneq\int_{\normp{\xb-\yb}\le r_n}g(\normp{\xb-\yb}/r_n)(\varphi(\xb)-\varphi(\yb))\varrho(\yb)\dd \yb,\label{eq:N1N2}\\
			D_1\coloneq\frac{1}{n}\sum_{j=1}^n v_j&,\quad D_2\coloneq\int_{\normp{\xb-\yb}\le r_n}g(\normp{\xb-\yb}/r_n)\varrho(\yb)\dd \yb,\label{eq:D1D2}
		\end{align}
		and
		\begin{align}
			\norm{\xb}_{\Sigma}\coloneq&\left(\sum_{i=1}^d x_i^2/(2\sigma_i^2)\right)^{1/2}.\label{eq:sigmanorm}
		\end{align}
		Our argument relies on the following decomposition
		\begin{align}
			\frac{m_2r_n^2}{2m_0}\norm{(\sLn-\widetilde{\mathcal{T}}_n)\varphi}_{\cF}\le&\normbb{\left(\frac{N_1}{D_1}-\frac{N_1}{D_2}\right)\bm{1}_{B_1}}_{\cF}+\normbb{\left(\frac{N_1}{D_1}-\frac{N_1}{D_2}\right)\bm{1}_{B_1^\complement}}_{\cF}+\notag\\
			&\normbb{\left(\frac{N_1}{D_2}-\frac{N_2}{D_2}\right)\bm{1}_{B_1}}_{\cF}+\normbb{\left(\frac{N_1}{D_2}-\frac{N_2}{D_2}\right)\bm{1}_{B_1^\complement}}_{\cF}\notag\\
			\coloneq&\mathsf{E}_1+\mathsf{E}_2+\mathsf{E}_3+\mathsf{E}_4.\label{eq:decomp}
		\end{align}
		In what follows, we will show that $\mathsf{E}_i=\oop(r_n^2)$, $1\le i\le 4$.
		\begin{enumerate}
			\item {\bf{Control of $\mathsf{E}_4$.}}\; By the triangle inequality,
			$$\normb{\mathsf{E}_4}_{\cF}\le\normb{\left(N_2/D_2\right)\bm{1}_{B_1^\complement}}_{\cF}+\normb{\left(N_1/D_2\right)\bm{1}_{B_1^\complement}}_{\cF}.$$
			Recall $\widetilde{\mathcal{T}}_n$ in \eqref{eq:defttn}, $N_2$ in \eqref{eq:N1N2} and $D_2$ in \eqref{eq:D1D2}. For the first term,
			\begin{align}
				\normb{(N_2/D_2)\bm{1}_{B_1^\complement}}_{\cF}\asymp r_n^2\normf{\widetilde{\mathcal{T}}_n \varphi\bm{1}_{B_1^\complement}}\le r_n^2(\normf{\mathcal{T}_n \varphi\bm{1}_{B_1^\complement}}+\normf{(\mathcal{T}_n-\widetilde{\mathcal{T}}_n)\varphi})=\oo(r_n^2),\label{eq:5.3.2.1}
			\end{align}
			where in the last step we used Lemma \ref{lemma:ttn} and \eqref{eq:bulk}.
			Next, we control the second term. Denote 
			\begin{align}
				i_0\equiv i_0(\xb,\{\sigma_i\}_{i=1}^d)\coloneq\arg\max_i\{|x_i/\sigma_i|\}.\label{eq:i0}
			\end{align}
			We can see that whenever $\xb\notin B_1$, $|x_{i_0}|\ge\sigma_{i_0}\sqrt{2(1+\beta)\log n}$. Then with probability $1-\oo(1)$, for all $j=1,\cdots,n$ and $\xb\notin B_1$,
			\begin{align}
				\normp{\xb-\xb_j}\ge|x_{i_0}|(1-1/\sqrt{1+\beta})\ge\frac{(1-1/\sqrt{1+\beta})}{\sqrt{d}}\left(\sum_{i=1}^d \frac{\sigma_{i_0}^2}{\sigma_i^2}x_{i}^2\right)^{1/2}\ge \zeta\norm{\xb}_{\Sigma},\label{eq:5.3.2.2}
			\end{align}
			where in the first step we used Lemma \ref{lemma:xj}, in the second step we used \eqref{eq:i0} and $\zeta>0$ is some fixed constant independent of $\xb$. Therefore, for sufficiently large $n$, with probability $1-\oo(1)$, 
			\begin{align}
				\mathsf{s}\left(\alpha_n(r_n^2-\normp{\xb-\xb_j}^2)\right)&=\frac{1}{1+\exp(-\alpha_n(r_n^2-\normp{\xb-\xb_j}^2))}\le\frac{1}{1+\exp(\alpha_n(\zeta\norm{\xb}_{\Sigma}^2-r_n^2))}\notag\\
				&\le\exp(-\alpha_n(\zeta\norm{\xb}_{\Sigma}^2-r_n^2))\le\exp(-\alpha_n\zeta\norm{\xb}_{\Sigma}^2/2),\label{eq:5.3.2.3}
			\end{align}
			where in the second step we used \eqref{eq:5.3.2.2} and the last step is from the fact that for sufficiently large $n$, $r_n^2\le\zeta\norm{\xb}_{\Sigma}^2/2$. 
			Moreover, recalling \eqref{eq:varphi} and \eqref{eq:i0}, for $\xb\notin B_1$, we have that with probability $1-\oo(1)$,
			\begin{align}
				|\varphi(\xb)|,|\varphi(\xb_j)|\lesssim\prod_{i=1}^d|x_{i_0}/\sigma_{i_0}|^{k_i}\lesssim\prod_{i=1}^d\norm{\xb}_{\Sigma}^{k_i}=\OO(\norm{\xb}_{\Sigma}^{|\bm{k}|}).\label{eq:5.3.2.9}
			\end{align}
			Recall $N_1$ in \eqref{eq:N1N2} and that $g^*(\cdot)$ is bounded. Then we have that
			\begin{align}
				N_1&=\frac{1}{n}\sum_{j=1}^n\mathsf{s}(\alpha_n(r_n^2-\normp{\xb-\xb_j}^2))g^*(\normp{\xb-\xb_j}/r_n)(\varphi(\xb)-\varphi(\xb_j))\notag\\
				&\lesssim\frac{1}{n}\sum_{j=1}^n(|\varphi(\xb)|+|\varphi(\xb_j)|)\exp\left(-\frac{\alpha_n\zeta}{2}\norm{\xb}_{\Sigma}^2\right)\notag\\
				&\lesssim\norm{\xb}_{\Sigma}^{|\bm{k}|}\exp\left(-2\norm{\xb}_{\Sigma}^2\right)\exp\left(-\frac{\alpha_n\zeta}{4}\norm{\xb}_{\Sigma}^2\right)\le\norm{\xb}_{\Sigma}^{|\bm{k}|}\varrho^2(\xb)\exp\left(-\frac{\alpha_n\zeta}{8}|x_{i_0}/\sigma_{i_0}|^2\right)\notag\\
				&\le\norm{\xb}_{\Sigma}^{|\bm{k}|}\varrho^2(\xb)\exp(-\alpha_n\zeta \log n/4)\le\norm{\xb}_{\Sigma}^{|\bm{k}|}\varrho^2(\xb)n^{-5},\label{eq:5.3.2.5}
			\end{align}
			where in the second step we used \eqref{eq:5.3.2.3}, in the third step we used \eqref{eq:5.3.2.9} and in the last but one step we used that $|x_{i_0}/\sigma_{i_0}|>\sqrt{2(1+\beta)\log n}$.
			On the other hand, for $\xb\notin B_1$ and $\normp{\xb-\yb}\le r_n$, we have that 
			\begin{align}
				\norm{\yb}_{\Sigma}\le\norm{\xb}_{\Sigma}+\norm{\xb-\yb}_{\Sigma}=\norm{\xb}_{\Sigma}+\OO(r_n)\le\sqrt{2}\norm{\xb}_{\Sigma}.\label{eq:5.3.2.6}
			\end{align}
			Recall $D_2$ in \eqref{eq:D1D2}. \eqref{eq:5.3.2.6} indicates that 
			\begin{align}
				D_2=\int_{\normp{\xb-\yb}\le r_n}g(\normp{\xb-\yb}/r_n)\exp(-\norm{\yb}_\Sigma^2)\dd\yb\gtrsim r_n^d\exp(-2\norm{\xb}_{\Sigma}^2).\label{eq:5.3.2.7}
			\end{align}
			Therefore, with probability $1-\oo(1)$,
			\begin{align}
				\normbb{\left(\frac{N_1}{D_2}\right)\bm{1}_{B_1^\complement}}_{\cF}^2\lesssim&\int_{B_1^\complement}\left(\frac{\norm{\xb}_\Sigma^{|\bm{k}|}\varrho^2(\xb)n^{-5}}{r_n^d \varrho^2(\xb)}\right)^2\varrho^2(\xb)\dd\xb= n^{-10}r_n^{-2d}\int_{B_1^\complement}\norm{\xb}_\Sigma^{2|\bm{k}|}\varrho^2(\xb)\dd\xb\notag\\
				\le&n^{-10}r_n^{-2d}\int_{\R^d}\norm{\xb}_\Sigma^{2|\bm{k}|}\varrho^2(\xb)\dd\xb=\oo(r_n^4),\label{eq:5.3.2.8}
			\end{align}
			where in the first step we used \eqref{eq:5.3.2.5} and \eqref{eq:5.3.2.7}, and in the last step we used \eqref{eq:2.2} and the fact that $\int_{\R^d}\norm{\xb}_\Sigma^{2|\bm{k}|}\varrho^2(\xb)\dd\xb<\infty$.
			Combining \eqref{eq:5.3.2.1} and \eqref{eq:5.3.2.8}, we conclude that $\mathsf{E}_4=\oop(r_n^2)$.
			
			\item {\bf{Control of $\mathsf{E}_2$.}}\; Recall $N_1$ in \eqref{eq:N1N2} and $D_1$ in \eqref{eq:D1D2}. By \eqref{eq:5.3.2.9}, for $\xb\notin B_1$,
			\begin{align}
				\left|N_1/D_1\right|\le|\varphi(\xb)|+\max_{1\le j\le n}|\varphi(\xb_j)|=\OOp(\norm{\xb}_{\Sigma}^{|\bm{k}|}).\label{eq:5.3.2.10}
			\end{align}
			Therefore, with probability $1-\oo(1)$,
			\begin{align}
				\normbb{\left(\frac{N_1}{D_1}\right)\bm{1}_{B_1^\complement}}_{\cF}^2&\lesssim\int_{B_1^\complement}\norm{\xb}_\Sigma^{2|\bm{k}|}\varrho^2(\xb)\dd\xb\lesssim\exp(-x_{i_0}^2/(2\sigma_{i_0}^2))\int_{B_1^\complement}\norm{\xb}_\Sigma^{2|\bm{k}|}\varrho(\xb)\dd\xb\notag\\
				&\le\OO(n^{-1-\beta})=\oo(r_n^4),\label{eq:5.3.2.11}
			\end{align}
			where in the first step we used \eqref{eq:5.3.2.10}, in the second step we used the construction of $B_1$ in \eqref{eq:bulk} and the fact that $\varrho(\xb)\lesssim\exp(-x_{i_0}^2/(2\sigma_{i_0}^2))$ with $i_0$ in \eqref{eq:i0}, in the third step we used the fact that $\int_{\R^d}\norm{\xb}_\Sigma^{2|\bm{k}|}\varrho(\xb)\dd\xb<\infty$ and in the last step we used \eqref{eq:2.2}. Combining \eqref{eq:5.3.2.8} and \eqref{eq:5.3.2.11}, we conclude that $\mathsf{E}_2=\oop(r_n^2)$.
			
			\item {\bf{Control of $\mathsf{E}_3$.}}\; Recall $w_j$ in \eqref{eq:wj}.
			We have that
			\begin{align}
				\E \norm{\mathsf{E}_3}_{\cF}^2=&\E\int_{\xb\in B_1}\left(\frac{1}{n}\sum_{j=1}^{n}w_j-N_2\right)^2\frac{1}{D_2^2}\varrho^2(\xb)\dd\xb\notag\\
				=&\OO(r_n^{-2d})\E\int_{\xb\in B_1}\left(\frac{1}{n}\sum_{j=1}^{n}w_j-N_2\right)^2\dd\xb=\OO(r_n^{-2d})\int_{\xb\in B_1}\E\left(\frac{1}{n}\sum_{j=1}^{n}w_j-N_2\right)^2\dd\xb\notag\\
				=&\OO(r_n^{-2d})\int_{\xb\in B_1}\frac{1}{n}(\E w_j^2-N_2^2)+\frac{n-1}{n}\E w_j(\E w_j-N_2)+\frac{n+1}{n}N_2(N_2-\E w_j)\dd\xb\notag\\
				\le&\OO(r_n^{-2d})\int_{\xb\in B_1}\frac{1}{n}(\E w_j^2-N_2^2)+2|\E w_j-N_2+N_2||\E w_j-N_2|+2|N_2||N_2-\E w_j|\dd\xb\notag\\
				\le&\OO(r_n^{-2d})\int_{\xb\in B_1}\frac{1}{n}(\E w_j^2+N_2^2)+2|\E w_j-N_2|^2+4|N_2||N_2-\E w_j|\dd\xb\notag\\
				=&\OO(r_n^{-2d})\int_{\xb\in B_1}\frac{1}{n}\left(\varrho^2(\xb)\OO(r_n^{2d+2}\log^{|\bm{k}|-1}n)+\varrho(\xb)\OO(r_n^{d+2})\log^{(|\bm{k}|-1)/2}n\right)\label{eq:5.3.2.12}\\
				&+\varrho^2(\xb)\OO(r_n^{2d-2}\log^{|\bm{k}|-1}n)/\alpha_n^2+\varrho^2(\xb)\OO(r_n^{2d}\log^{|\bm{k}|-1}n)/\alpha_n\dd\xb\notag\\
				\lesssim&\frac{\log^{|\bm{k}|-1}n}{nr_n^{d-2}}\int_{\R^d}\varrho(\xb)\dd\xb=\oo(r_n^4),\notag
			\end{align}
			where in the second step we used Lemma \ref{lemma:interror}, in the third step we used Fubini's theorem, in \eqref{eq:5.3.2.12} we used \eqref{eq:B.4.0.1}, \eqref{eq:B.4.0.2} and \eqref{eq:B.4.0.3}, and in the last two steps we used \eqref{eq:2.2} and the assumption that $\alpha_n\gg n^{7/2}$. By Markov's inequality, we have that 
			\begin{align*}
				&\P\left(\normf{\mathsf{E}_3}\ge \left(\E \norm{\mathsf{E}_3}_{\cF}^2\right)^{1/4}r_n\right)\le \frac{\E \norm{\mathsf{E}_3}_{\cF}^2}{\left(\E \norm{\mathsf{E}_3}_{\cF}^2\right)^{1/2}r_n^2}=\frac{\left(\E \norm{\mathsf{E}_3}_{\cF}^2\right)^{1/2}}{r_n^2}=\oo(1).
			\end{align*}
			Since $\left(\E \norm{\mathsf{E}_3}_{\cF}^2\right)^{1/4}r_n=\oo(r_n^2)$, this concludes that $\mathsf{E}_3=\oop(r_n^2)$.
			
			\item {\bf{Control of $\mathsf{E}_1$.}}\;
			Recall $N_1$ in \eqref{eq:N1N2}, $D_1,D_2$ in \eqref{eq:D1D2}, $B_1$ in \eqref{eq:bulk} and $v_j$ in \eqref{eq:vj}. For all $\xb\in B_1$,
			$$\varphi(\xb)\lesssim\prod_{i=1}^{d}(x_i/\sigma_i)^{k_i}\lesssim \log^{|\bm{k}|/2} n.$$
			Combining with Lemma \ref{lemma:xj}, we have that with probability $1-\oo(1)$, for all $\xb\in B_1$
			\begin{align*}
				\left|N_1/D_1\right|\le|\varphi(\xb)|+\max_{j}|\varphi(\xb_j)|=\OO(\log^{|\bm{k}|/2} n).
			\end{align*}
			Therefore, with probability $1-\oo(1)$,
			\begin{align}
				\normb{\mathsf{E}_1}_{\cF}=\normbb{\frac{N_1}{D_1}\left(\frac{D_1-D_2}{D_2}\right)\bm{1}_{B_1}}_{\cF}=\OO(\log^{|\bm{k}|/2} n)\normbb{\left(\frac{D_1-D_2}{D_2}\right)\bm{1}_{B_1}}_{\cF}.\label{eq:5.3.2.13}
			\end{align}
			Then, we have that
			\begin{align}
				&\E\normbf{\left(\frac{D_1-D_2}{D_2}\right)\bm{1}_{B_1}}^2=\E\int_{\xb\in B_1}\left(\frac{1}{n}\sum_{i=1}^n v_j-D_2\right)^2\frac{\varrho^2(\xb)}{D_2^2}\dd\xb\notag\\
				\lesssim&r_n^{-2d}\E\int_{\xb\in B_1}\left(\frac{1}{n}\sum_{i=1}^n v_j-D_2\right)^2\dd\xb=r_n^{-2d}\int_{\xb\in B_1}\E\left(\frac{1}{n}\sum_{i=1}^n v_j-D_2\right)^2\dd\xb\notag\\
				\le&r_n^{-2d}\int_{\xb\in B_1}\frac{1}{n}(\E v_j^2+D_2^2)+2|\E v_j-D_2|^2+4|D_2||D_2-\E v_j|\dd\xb\notag\\
				=&r_n^{-2d}\int_{\xb\in B_1}\frac{1}{n}\left(\varrho(\xb)\OO(r_n^d)+\varrho(\xb)\OO(r_n^d)\right)+\varrho^2(\xb)\OO(r_n^{2d-4}/\alpha_n^2)+\varrho^2(\xb)\OO(r_n^{2d-2}/\alpha_n^2)\dd\xb\notag\\
				\asymp&\frac{1}{nr_n^d}=\oo\left(\frac{r_n^4}{\log^{|\bm{k}|}n}\right),\notag
			\end{align}
			where in the second step we used Lemma \ref{lemma:interror}, in the third step we used Fubini's theorem, in the fifth step we used \eqref{eq:B.4.0.4}, \eqref{eq:B.4.0.5} and Lemma \ref{lemma:interror} and in the last two steps we used \eqref{eq:2.2} and the assumption that $\alpha_n\gg n^{7/2}$. Similarly, by Markov's inequality, we have that,
			\begin{align*}
				\normbb{\left(\frac{D_1-D_2}{D_2}\right)\bm{1}_{B_1}}_{\cF}=\oop\left(\frac{r_n^2}{\log^{|\bm{k}|/2}n}\right).
			\end{align*}
			Combining with \eqref{eq:5.3.2.13}, we conclude that $\mathsf{E}_1=\oop(r_n^2)$. 
		\end{enumerate}
		
		Combining all above discussion, we have shown that $\normf{(\sLn-\widetilde{\mathcal{T}}_n)\varphi}=\oop(1)$. The control of $\normf{(\mathcal{T}_n-\widetilde{\mathcal{T}}_n)\varphi}$ enjoys a similar argument and we put its proof to Lemma \ref{lemma:ttn}, which shows that $\normf{(\mathcal{T}_n-\sLn)\varphi}=\oop(1)$. This completes our proof of the variance part.
	\end{proof}

	\appendix 
	\section{Some technical background}\label{appendix:bgs}
	\subsection{Second order linear differential equation}\label{appendix:pde}
	In this section, we summarize some useful results regarding the second order linear differential equation which allows us to study the weighted Laplace-Beltrami operator in \eqref{eq:2.1.1}. The results are standard in the literature, for example see \cite{NagyODE,strauss2007partial,walter2013ordinary}.
	
	Consider the following eigenvalue problem for second order linear partial differential equations:
	\begin{align}
		\sum_{i=1}^d\frac{\partial^2f}{\partial x_i^2}-\sum_{j=1}^d \frac{2x_i}{\sigma_i^2}\frac{\partial f}{\partial x_i}+\lambda f=0,\label{eq:A2.1.1}
	\end{align}
	where $f:\R^d\rightarrow\R$ belongs to the Hilbert space $\cF$ defined in \eqref{eq:2.1.2} and the parameters $\sigma_i>0$ are fixed constants. Recall the Hilbert space $\cF_i$ as in \eqref{eq:4.4} with corresponding inner product $\Rprod{\cdot}{\cdot}_i$ in \eqref{eq:4.9} for $i=1,\cdots,d$. We start looking for separable solutions of the form: 
	\begin{align}
		f(\xb)=\prod_{i=1}^d \mathsf{f}_i(x_i),\quad \xb=(x_1,\cdots,x_d)^\top,\label{eq:A2.1.2}
	\end{align}
	where $\mathsf{f}_i$ satisfies
	\begin{align}
		\mathsf{f}_i''(x_i)-\frac{2x_i}{\sigma_i^2}\mathsf{f}_i'(x_i)+\lambda_i\mathsf{f}_i(x_i)=0,\quad \mathsf{f}_i\in\cF_i,\quad \mathsf{f}_i(x_i)\not\equiv0,\quad i=1,\cdots,d,\label{eq:A2.1.3}
	\end{align}
	for some $\lambda_1,\cdots,\lambda_d$ such that $\sum_{i=1}^d \lambda_i=\lambda$. 
	It is straightforward to verify that products of solutions to \eqref{eq:A2.1.3} yield a set of solutions to \eqref{eq:A2.1.1} of the form of \eqref{eq:A2.1.2}. 
	
	To analyze \eqref{eq:A2.1.3}, consider homogeneous ordinary differential equations of the form:
	\begin{align}
		\mathsf{u}''(t)+b_1t\mathsf{u}'(t)+b_0\mathsf{u}(t)=0,\quad b_1,b_0\in\R.\label{eq:A2.2}
	\end{align}
	The following theorem shows that solutions to \eqref{eq:A2.2} can be written into general form. Recall that two functions $\mathsf{u}_1(t)$ and $\mathsf{u}_2(t)$ are linearly independent on $\R$ if the following holds: $c_1 \mathsf{u}_1(t)+c_2 \mathsf{u}_2(t)=0$ for all $t\in\R$ if and only if $c_1=c_2=0$.
	
	\begin{theorem}
		\label{theo:ode2}
		If $\mathsf{u}_1,\mathsf{u}_2\in C^2(\R)$ are both solutions to \eqref{eq:A2.2} and are linearly independent on $\R$, then for any solution $\mathsf{u}(t)$ to \eqref{eq:A2.2}, there exist unique constants $c_1,c_2\in\R$ so that
		\begin{align*}
			\mathsf{u}(t)=c_1\mathsf{u}_1(t)+c_2\mathsf{u}_2(t).
		\end{align*}
	\end{theorem}
	\begin{proof}
		See Chapter 2 in \cite{NagyODE}.
	\end{proof}
	
	\begin{remark}
		Theorem \ref{theo:ode2} allows us to construct all possible solutions from a fundamental set of two linearly independent solutions. This implies that to solve \eqref{eq:A2.2}, it suffices to impose initial conditions that generate two linearly independent solutions.
	\end{remark}
	
	Classic Sturm–Liouville theory (e.g. \cite[Chapter 27]{walter2013ordinary}) implies that solutions to \eqref{eq:A2.1.3} form a complete orthonormal basis of $(\cF_i,\Rprod{\cdot}{\cdot}_i)$. 
	Denote the tensor product
	\begin{align*}
		\widetilde{\cF}\coloneq\cF_1\otimes\cdots\otimes\cF_d=\operatorname{span}\left\{\prod_{i=1}^d \mathsf{f}_i(x_i)\Biggm|\mathsf{f}_i\in\cF_i,i=1,\cdots d\right\}.
	\end{align*}
	Now we define an inner product on $\widetilde{\cF}$:
	$$\Rprod{f_1}{f_2}\coloneq\prod_{i=1}^d\Rprod{\mathsf{f_i}}{\mathsf{g}_i}_i,$$
	where $f_1,f_2\in\widetilde{\cF}$, $f_1(\xb)=\prod_{i=1}^d\mathsf{f}_i(x_i)$, $f_2(\xb)=\prod_{i=1}^d\mathsf{g}_i(x_i)$ and $\mathsf{f}_i,\mathsf{g}_i\in\cF_i$ for $i=1,\cdots,d$. It is straightforward to verify that $\cF$ is the completion of $\widetilde{\cF}$ with respect to $\Rprod{\cdot}{\cdot}$. Therefore, products of solutions to \eqref{eq:A2.1.3} constitute a complete orthonormal basis of $(\cF,\Rprod{\cdot}{\cdot})$. Consequently, solving \eqref{eq:A2.1.1} reduces to solving the family of problems \eqref{eq:A2.1.3} respectively.

	\subsection{Confluent hypergeometric function}\label{appendix:chf}
	\begin{definition}(\cite[Chapter 13]{olver2010nist})
		\label{definiton:chf}
		For $a,b\in\R$ and $b\notin\{0,-1,-2,\cdots\}$, the confluent hypergeometric function of the first kind is defined by the power series in $z\in\C$:
		\begin{align}
			_1F_1(a,b,z)=\sum_{j=0}^{\infty}\frac{a^{(j)}z^j}{b^{(j)}j!},\label{eq:A2.3}
		\end{align}
		where $a^{(0)}=b^{(0)}=1$ and for $j\ge1$,
		$$a^{(j)}=a(a+1)\cdots(a+j-1),\quad b^{(j)}=b(b+1)\cdots(b+j-1).$$
	\end{definition}
	\begin{remark}
		\label{remark:chf1}
		Note that for $\mathsf{f}_{i,1}(x_i)$ and $\mathsf{f}_{i,2}(x_i)$ in \eqref{eq:4.6} and \eqref{eq:4.7}, we can rewrite them into confluent hypergeometric functions of the first kind in terms of \eqref{eq:A2.3} as follows:
		$$\mathsf{f}_{i,1}(x_i)={}_1F_1\left(-\frac{\lambda_i \sigma_i^2}{4},\frac{1}{2},\frac{x_i^2}{\sigma_i^2}\right),\quad f_2(x)={}_1F_1\left(\frac{2-\lambda_i \sigma_i^2}{4},\frac{3}{2},\frac{x_i^2}{\sigma_i^2}\right)x_i.$$
	\end{remark}
	
	In the following proposition, we introduce an important special case of Definition \ref{definiton:chf} when the first parameter $a$ is a non-positive integer. In this case, \eqref{eq:A2.3} becomes a polynomial with finite degree.
	
	\begin{proposition}
		\label{prop:hermite}
		When $a=-k$ for $k\in\{0,1,2,\cdots\}$, $_1F_1(a,b,z)$ becomes a polynomial in $z$ of degree $k$ in the sense that
		$$_1F_1(a,b,z)=\sum_{j=0}^{k}\frac{a^{(j)}z^j}{b^{(j)}j!}.$$
	\end{proposition}
	\begin{proof}
		The proof follows from the simple fact that when $a=-k$, we have $a^{(j)}=0$ for $j>k$. 
	\end{proof}
	
	With Proposition \ref{prop:hermite}, we can obtain the asymptotic properties for confluent hypergeometric functions of the first kind when $a=-k$ for $k\in\{0,1,2,\cdots\}$. Recall that for $z\in\C$, $\operatorname{ph}(z)$ is denoted as the phase of $z$, i.e.,
	$$\operatorname{ph}(z)=\arctan\left(\frac{\operatorname{Im}(z)}{\operatorname{Re}(z)}\right).$$
	
	\begin{lemma}
		\label{lemma:chfkarge}
		If $a\notin\{0,-1,-2,\cdots\}$, for some arbitrary small positive constant $\delta$, when $|\operatorname{ph}(z)|\le\frac{1}{2}\pi-\delta$,
		\begin{align*}
			{}_1F_1(a,b,z)\sim \frac{1}{\Gamma(a)}e^z z^{a-b},\ |z|\rightarrow\infty.
		\end{align*}
	\end{lemma}
	\begin{proof}
		See Chapter 13 in \cite{olver2010nist}.
	\end{proof}
	
	\begin{remark}
		\label{remark:chf2}
		By Remark \ref{remark:chf1} and the fact that $\operatorname{ph}(x_i^2)=0$ for $x_i\in\R$, the asymptotic properties for $\mathsf{f}_{i,1}(x_i)$ in \eqref{eq:4.6} when $\lambda_{i}\sigma_i^2\notin\{0,4,8,\cdots\}$ and $\mathsf{f}_{i,2}(x_i)$ in \eqref{eq:4.7} when $\lambda_{i}\sigma_i^2\notin\{2,6,10,\cdots\}$ can be studied using Lemma \ref{lemma:chfkarge}.
	\end{remark}
	
	\subsection{Preliminary results from complex analysis}
	In this section, we summarize some useful results from complex analysis that will be used in the proof of Theorem \ref{theo:normald=1}. Most of the results can be found in \cite{ahlfors1979complex,kato2013perturbation}. We begin by recalling some important definitions \cite{ahlfors1979complex}. A function $f$ is holomorphic on a domain $D$ if it is complex‐differentiable at every point of $D$. A domain $D$ is simply connected if its complement in the extended complex plane is connected. 
	\begin{theorem}
		\label{them:residue}
		Let $f(z)$ be holomorphic on a simply connected domain $D$ except for possibly finitely many isolated poles, and let $\Gamma\subset D$ be a positively oriented, piecewise‐smooth, simple closed contour avoiding those poles. If $\{a_1,\cdots,a_m\}\subset\operatorname{Int}(\Gamma)$ are the poles of $f$ enclosed by $\Gamma$, then
		$$\frac{1}{2\pi\mathrm{i}}\int_\Gamma f(z)\dd z=\sum_{k=1}^m\operatorname{Res}(f,a_k),$$
		where $\operatorname{Res}(f,a_k)$ is the residue of $f$ at the singularity $a_k$.
		In particular, if $f(z)$ is holomorphic on $D$ (i.e., $m=0$), then
		$\int_{\Gamma} f(z)\,\mathrm{d}z = 0.$
	\end{theorem}
	\begin{proof}
		See Chapter 4 in \cite{ahlfors1979complex}. 
	\end{proof}
	
	\begin{remark}
		According to Theorem \ref{them:residue}, a contour integral can reduce to a sum of residues at associated poles. In our application, those poles will be the eigenvalues of the concerned operators. Consequently, this theorem enables us to study the spectral properties via complex integration.
	\end{remark}
	
	The following theorem provides a method for counting the eigenvalues of a closed linear operator using a contour integral. Recall that for a trace-class closed linear operator $\mathcal{L}$ on a Hilbert space with an orthonormal basis $\{e_j\}_{j=1}^{\infty}$, its trace is defined as:
	\begin{align}
		\operatorname{Tr}(\mathcal{L})=\sum_{j=1}^\infty\Rprod{e_j}{\mathcal{L}e_j},\label{eq:trace}
	\end{align}
	where $\Rprod{\cdot}{\cdot}$ is the equipped inner product of the Hilbert space.
	\begin{theorem}
		\label{theo:counting}
		Let $\mathcal{K}$ be a closed linear operator on a Hilbert space and $\Gamma$ be a positively oriented, piecewise‐smooth, simple closed contour lying entirely in the resolvent set of $\mathcal{K}$. Suppose the resolvent $R(z,\mathcal{K})=(\mathcal{K}-z\mathcal{I})^{-1}$ is of trace class for all $z\in\Gamma$. Further, suppose $\operatorname{spec}(\mathcal{K})\cap\operatorname{Int}(\Gamma)=\{\lambda_j\}_{j=1}^k$, where $\operatorname{spec}(\mathcal{K})$ is the spectrum of $\mathcal{K}$ and $\lambda_j$'s are distinct and each with multiplicity $n_j$. Then
		\begin{align}
			-\frac{1}{2\pi\mathrm{i}}\oint_\Gamma \operatorname{Tr}(R(z,\mathcal{K}))\dd z=\sum_{j=1}^{k}n_j.\notag
		\end{align}
	\end{theorem}
	\begin{proof}
		See Chapter 3 in \cite{kato2013perturbation}.
	\end{proof}
	\begin{remark}
		Theorem \ref{theo:counting} allows us to compare the numbers of eigenvalues of $\Delta$ and $\sLn$ in \eqref{eq:2.2.3} within the same given region. By establishing that they have the same numbers of eigenvalues in some given regions, we can construct a one-to-one correspondence between their edge eigenvalues, which is a crucial step in showing the convergence of the edge eigenvalues of $\sLn$ to those of $\Delta$.
	\end{remark}
	
	\section{Proof of additional lemmas}
	Recall $\Kb$ in \eqref{eq_originalK}, $\Ka$ in \eqref{eq:2.2.1}, $\Lb$ in \eqref{eq:2.1.2} and $\Lc$ in \eqref{eq:2.2.2}. The following lemma shows the closeness between $\Lb$ and $\Lc$ and was used in the proof of \eqref{eq:main} in Section \ref{sec:mainproof}. Recall $\alpha_n$ in \eqref{eq:2.2.1}.
	\begin{lemma}
		\label{lemma:S2L}
		Under the assumptions of Theorem \ref{theo:normald=1}, if $\alpha_n\gg n^{7/2}$, then, $\norm{\Lb-\Lc}=\oop(1)$.
	\end{lemma}
	The proof relies on the fact that from Lemma \ref{lemma:sig2ind}, for sufficiently large $\alpha_n$, with high probability, the difference between the 0-1 kernel and the sigmoid‐type kernel is negligible. Consequently, $\Kb$ and $\Ka$ are entry‐wise close. Since $\Db$ and $\Da$ are their associated degree matrices, they are also close. This will yield the closeness of $\Lb$ and $\Lc$. 
	\begin{proof}
		From the definition of $\Lb$ and $\Lc$, we have that
		\begin{align}
			\norm{\Lc-\Lb}&=\frac{2m_0}{m_2r_n^2}\norm{\Da^{-1}\Ka-\Db^{-1}\Kb}\lesssim\frac{\sqrt{n}}{r_n^2}\norm{\Da^{-1}\Ka-\Db^{-1}\Kb}_{\infty} \notag\\
			&=\frac{\sqrt{n}}{r_n^2}\max_{1\le i\le n}\sum_{j=1}^{n}\left|\frac{\Ka(i,j)}{\tilde{d_i}}-\frac{\Kb(i,j)}{d_i}\right|\label{eq:5.2.2},
		\end{align}
		where in the second step we used \eqref{eq:matrixnorm} and the last step follows from the definition of $\norm{\cdot}_{\infty}$. According to the form of \eqref{eq:5.2.2}, it remains to bound the following quantities uniformly for all $1\le i,j\le n$:
		$$
		\bigl|\Ka(i,j) - \Kb(i,j)\bigr|
		\text{ , }
		\bigl|\tilde{d_i} - d_i\bigr|.
		$$
		
		First, observe that
		\begin{align}
			|\Ka(i,j)-\Kb(i,j)|&=\left|g(\normp{\xb_i-\xb_j}/r_n)\bm{1}(\normp{\xb_i-\xb_j}\le r_n)-\frac{g^*(\norm{\xb_i-\xb_j}/r_n)}{1+\exp(-\alpha_n(r_n^2-\normp{\xb_i-\xb_j}^2))}\right|\notag\\
			&\le\frac{\max_{t\ge 0}g^*(t)}{1+\exp(\alpha_n|r_n^2-\normp{\xb_i-\xb_j}^2|)}\lesssim\exp(-\alpha_n|r_n^2-\normp{\xb_i-\xb_j}^2|),\label{eq:5.2.3}
		\end{align}
		where the second step is from Lemma \ref{lemma:sig2ind} and in the last step we used the boundedness assumption on $g^*(\cdot)$ as in \eqref{eq:2.2.1}. To control \eqref{eq:5.2.3}, we first pause to provide controls for $\norm{\xb_i-\xb_j}^2$. 
		When $i=j$, 
		$$\left|\normp{\xb_i-\xb_j}^2-r_n^2\right|=r_n^2>\frac{r_n^{2-d}}{n^{3+\xi}},$$
		where we used \eqref{eq:2.2}. This yields for $i=j$,
		\begin{align}
			\mathbb{P}\left(\left|\normp{\xb_i-\xb_j}^2-r_n^2\right|\le\frac{r_n^{2-d}}{n^{3+\xi}}\right)=0.\label{eq:B.1.7}
		\end{align}
		Recall $\xb_i\stackrel{\operatorname{i.i.d.}}{\sim}\mathcal{N}_d(\bm{0},\Sigma)$, $i=1,\cdots,n$, with $\Sigma$ as in \eqref{eq:2.1}. Then, for $i \neq j$, $\xb_i-\xb_j\sim\mathcal{N}_d(\bm{0},2\Sigma)$. Denote $S\coloneq\frac12\normp{\xb_i-\xb_j}^2$. Now we discuss two cases, $1\le p<\infty$ and $p=\infty$.
		\begin{enumerate}
			\item $1\le p<\infty$.\;\; Let 
			$$T\coloneq\sum_{i=1}^d \sigma_i^p|W_i|^p\coloneq\sum_{i=1}^dY_i,\quad W_i\stackrel{\operatorname{i.i.d.}}{\sim}\mathcal{N}(0,1),\quad i=1,\cdots,d.$$
			Then we have $S\stackrel{d}{=}T^{2/p}$, where $\stackrel{d}{=}$ denotes equal in distribution. Observe that each $Y_i$ has the density 
			\begin{align}
				\rho_i(y)=\frac{2}{p\sigma_i\sqrt{2\pi}}y^{1/p-1}\exp\left(-\frac{y^{2/p}}{2\sigma_i^2}\right)\le\frac{y^{1/p-1}}{\sigma_i}\coloneq\hat\rho_i(y),\quad y>0.\label{eq:B.1.1}
			\end{align}
			Denote the density of $T$ as $\rho_{T}(\cdot)$. We can see that $\rho_{T}(\cdot)$ is the $d$-fold convolution of $\rho_1,\cdots,\rho_d$, i.e.,
			\begin{align}
				\rho_{T}=\rho_1*\cdots*\rho_d,\notag
			\end{align}
			where $*$ denotes convolution. By \eqref{eq:B.1.1}, for $y>0$
			$$\rho_{T}(y)=(\rho_1*\cdots*\rho_d)(y)\le(\hat\rho_1*\cdots*\hat\rho_d)(y).$$
			Using the identity (5.14.1) in \cite{olver2010nist} with the change of variables $u_i=yt_i$ and setting $z_i=1/p$, $i=1,\cdots,d$, we have that
			\begin{align}
				(\hat\rho_1*\cdots*\hat\rho_d)(y)=\frac{\prod_{i=1}^d\Gamma(1/p)/\sigma_i}{\Gamma(d/p)}y^{d/p-1},\notag
			\end{align}
			which indicates that $\rho_{T}(y)\lesssim y^{d/p-1}$. Denote the density of $S$ as $\rho_S(\cdot)$. Using a change of variable $S=T^{2/p}$, we have that for $y>0$,
			\begin{align}
				\rho_S(y)=\rho_T(y^{p/2})\frac{p}{2}y^{p/2-1}\lesssim y^{d/2-1}.\label{eq:B.1.3}
			\end{align}
			Fix some small fixed constant $0<\xi < 0.3$ and denote 
			\begin{align}
				I\equiv I_n\coloneq \left[\frac{r_n^2}{2} - \frac{r_n^{2-d}}{2n^{3+\xi}},\frac{r_n^2}{2} + \frac{r_n^{2-d}}{2n^{3+\xi}}\right].\label{eq:I}
			\end{align}
			For any $s \in I$, we have from \eqref{eq:2.2} that $r_n^2 \gg \tfrac{r_n^{2-d}}{n^{3+\xi}}$. This implies that $|y - \tfrac{r_n^2}{2}|\le\tfrac14\,r_n^2$ for $y \in I$. Therefore, uniformly for $y \in I$,
			\begin{align}
				\rho_S(y)
				&\lesssim \left(\tfrac{r_n^2}{2} + y-\tfrac{r_n^2}{2}\right)^{d/2-1}\lesssim\bigl(r_n^2/4\bigr)^{d/2-1}=\OO(r_n^{d-2}). \label{eq:5.2.4}
			\end{align}
			Consequently,
			\begin{align}
				\mathbb{P}\left(\left|\normp{\xb_i-\xb_j}^2-r_n^2\right|\le\frac{r_n^{2-d}}{n^{3+\xi}}\right)&=\mathbb{P}\left(\left|\frac{1}{2}\normp{\xb_i-\xb_j}^2-\frac{1}{2}r_n^2\right|\le\frac{r_n^{2-d}}{2n^{3+\xi}}\right)\notag\\
				&=\int_{I}\rho_S(y)\dd y\lesssim\frac{r_n^{2-d}}{n^{3+\xi}}r_n^{d-2}=\OO\left(\frac{1}{n^{3+\xi}}\right),\label{eq:5.2.5}
			\end{align}
			where in the last but one step we used \eqref{eq:5.2.4}.
			
			\item $p=\infty$.\;\; We have that 
			$$S\stackrel{d}{=}\max_{1\le i\le d}V_i,\quad V_i=\sigma_i^2W_i^2,\quad W_i\stackrel{\operatorname{i.i.d.}}{\sim}\mathcal{N}(0,1),\quad i=1,\cdots,d.$$
			Observe that each $V_i$ has density
			\begin{align}
				\rho'_i(y)=\frac{1}{\sigma_i\sqrt{2\pi}}y^{-1/2}\exp\left(-\frac{y}{2\sigma_i^2}\right)\le\frac{1}{\sigma_i}y^{-1/2}.\notag
			\end{align}
			Similarly, we have for $y\in I$ with $I$ as in \eqref{eq:I}, for all $i=1,\cdots,d$,
			\begin{align}
				\rho'_i(y)\lesssim r_n^{-1}.\label{eq:B.1.5}
			\end{align}
			Then we have that 
			\begin{align}
				\mathbb{P}\left(\left|\norm{\xb_i-\xb_j}_\infty^2-r_n^2\right|\le\frac{r_n^{2-d}}{n^{3+\xi}}\right)&=
				\mathbb{P}\left(\left|\max_{1\le i\le d}V_i-\frac{1}{2}r_n^2\right|\le\frac{r_n^{2-d}}{2n^{3+\xi}}\right)\notag\\
				&\le\sum_{i=1}^d \mathbb{P}\left(\left|V_i-\frac{1}{2}r_n^2\right|\le\frac{r_n^{2-d}}{2n^{3+\xi}}\right)\notag\\
				&\le\sum_{i=1}^d\int_I\rho'_i(y)\dd y\lesssim d \frac{r_n^{2-d}}{n^{3+\xi}}r_n^{-1}\notag\\
				&=\OO\left(\frac{r_n^{1-d}}{n^{3+\xi}}\right)=\OO\left(\frac{1}{n^{2+\xi}}\right),\label{eq:B.1.4}
			\end{align}
			where in the last but two step we used \eqref{eq:B.1.5} and in the last step we used \eqref{eq:2.2}.
		\end{enumerate}
		
		Combining the two cases and \eqref{eq:B.1.7}, we have that for $1\le p\le\infty$ and any $1\le i,j\le n$,
		\begin{align}
			\mathbb{P}\left(\left|\normp{\xb_i-\xb_j}^2-r_n^2\right|\le\frac{r_n^{2-d}}{n^{3+\xi}}\right)=\OO\left(\frac{1}{n^{2+\xi}}\right).\label{eq:B.1.6}
		\end{align}
		To provide a control uniformly in $i$ and $j$, we make use of the following probability event
		\begin{align}
			E\equiv E_n\coloneq\left\{\min_{i,j}\left|\normp{\xb_i-\xb_j}^2-r_n^2\right|>\frac{r_n^{2-d}}{n^{3+\xi}}\right\}.\label{eq:5.2.6}
		\end{align}
		In what follows, we will show that $E$ holds with high probability. By Boole’s inequality,
		\begin{align}
			\P(E)&=1-\P\left(\min_{i,j}\left|\normp{\xb_i-\xb_j}^2-r_n^2\right|\le\frac{r_n^{2-d}}{n^{3+\xi}}\right)\notag\\
			&\ge1-\sum_{i=1}^n\sum_{j=1}^n\mathbb{P}\left(\left|\normp{\xb_i-\xb_j}^2-r_n^2\right|\le\frac{r_n^{2-d}}{n^{3+\xi}}\right)\notag\\
			&\ge1-n^2\OO\left(\frac{1}{n^{2+\xi}}\right)=1-\OO\left(\frac{1}{n^{\xi}}\right)\label{eq:5.2.7},
		\end{align}
		where the last but one step follows from \eqref{eq:B.1.6}. From now on, we restrict the discussion on the event $E$ so that all the arguments will be deterministic.
		
		Now we proceed to control the right-hand side of \eqref{eq:5.2.2}. Note on $E$, we have 
		$$
		\min_{i,j} \bigl|\normp{\xb_i-\xb_j}^2 - r_n^2\bigr|
		> \frac{r_n^{2-d}}{n^{3+\xi}}.
		$$
		Consequently, by \eqref{eq:5.2.3}, we have
		\begin{align}
			\max_{i,j}|\Ka(i,j)-\Kb(i,j)|&\le\max_{i,j}\exp(-\alpha_n|r_n^2-\normp{\xb_i-\xb_j}^2|)\notag\\
			&=\exp\left(-\alpha_n\min_{i,j}\left|\normp{\xb_i-\xb_j}^2-r_n^2\right|\right)\le\exp\left(-\frac{\alpha_nr_n^{2-d}}{n^{3+\xi}}\right).\label{eq:5.2.8}
		\end{align}
		To control the right-hand side of \eqref{eq:5.2.8}, by \eqref{eq:2.2} and the assumption that $\alpha_n\gg n^{7/2}$, for sufficiently large $n$,
		$$-\frac{\alpha_n r_n^{2-d}}{n^{3+\xi}}\le-\frac{\alpha_nr_n}{n^{3+\xi}}\le-\frac{n^{7/2-1/5}}{n^{3+\xi}}=-n^{0.3-\xi}.$$ Combining with \eqref{eq:5.2.8}, we have 
		\begin{align}
			\max_{i,j}|\Ka(i,j)-\Kb(i,j)|=\oo(n^{-3})\label{eq:5.2.9}.
		\end{align}
		
		Second, on $E$, we have
		\begin{align}
			\max_{i}|\tilde{d_i}-d_i|&=\max_{i}\left|\sum_{j=1}^n(\Ka(i,j)-\Kb(i,j))\right|\le\max_{i}\sum_{j=1}^n\left|\Ka(i,j)-\Kb(i,j)\right|\notag\\
			&\le n\max_{i,j}\left|\Ka(i,j)-\Kb(i,j)\right|=\oo(n^{-2}),\label{eq:5.2.10}
		\end{align}
		where in the last step we used \eqref{eq:5.2.9}.
		
		Together with the above bounds, we proceed to control \eqref{eq:5.2.2}. We decompose that
		\begin{align}
			\max_{i,j}\left|\frac{\Ka(i,j)}{\tilde{d_i}}-\frac{\Kb(i,j)}{d_i}\right|&=\max_{i,j}\left|\frac{\Ka(i,j)}{\tilde{d_i}}-\frac{\Kb(i,j)}{\tilde{d_i}}+\frac{\Kb(i,j)}{\tilde{d_i}}-\frac{\Kb(i,j)}{d_i}\right|\notag\\
			&\le\max_{i,j}\left|\frac{\Ka(i,j)}{\tilde{d_i}}-\frac{\Kb(i,j)}{\tilde{d_i}}\right|+\max_{i,j}\left|\frac{\Kb(i,j)}{\tilde{d_i}}-\frac{\Kb(i,j)}{d_i}\right|\notag\\
			&\le\max_{i,j}\frac{|\Ka(i,j)-\Kb(i,j)|}{\tilde{d_i}}+\max_{i,j}\Kb(i,j)\frac{|\tilde{d_i}-d_i|}{\tilde{d_i}d_i}\label{eq:5.2.11}.
		\end{align}
		To control the first term of \eqref{eq:5.2.11}, we start providing a lower bound of $\tilde{d_i}$. Observe that
		\begin{align}
			d_i = \sum_{j=1}^{n} \Kb(i,j)\ge\Kb(i,i) = 1.\label{eq:5.2.12}
		\end{align}
		According to \eqref{eq:5.2.10}, we readily obtain that the followings holds on $E$:
		\begin{align}
			\tilde{d_i}&=d_i+\tilde{d_i}-d_i\ge d_i-|\tilde{d_i}-d_i|\ge1-\oo(n^{-1/2})\ge\frac{1}{2}.\label{eq:5.2.13}
		\end{align}
		Together with \eqref{eq:5.2.9}, we find that the first term of the right-hand side of \eqref{eq:5.2.11} is bounded by $\oo(n^{-3})$.
		
		Then we proceed to control the second term of the right-hand side of \eqref{eq:5.2.11}. Using the fact that $\Kb(i,j)=\bm{1}(\normp{\xb_i-\xb_j}\le r_n)\le1$, \eqref{eq:5.2.11} , \eqref{eq:5.2.12} and \eqref{eq:5.2.13}, we see that the second term of the right-hand side of \eqref{eq:5.2.11} can be bounded by $\oo(n^{-2})$. In summary, we have proved that on $E$,
		\begin{align}
			\max_{i,j}\left|\frac{\Ka(i,j)}{\tilde{d_i}}-\frac{\Kb(i,j)}{d_i}\right|=\oo(n^{-2}).\notag
		\end{align}
		
		Consequently, we have from \eqref{eq:5.2.2} that on $E$,
		\begin{align*}
			\norm{\Lc-\Lb}&\lesssim\frac{\sqrt{n}}{r_n^2}\max_{1\le i\le n}\sum_{j=1}^{n}\left|\frac{\Ka(i,j)}{\tilde{d_i}}-\frac{\Kb(i,j)}{d_i}\right|\lesssim r_n^{-2}\sqrt{n}n\oo(n^{-2})=\oo\left(\frac{1}{\sqrt{n}r_n^2}\right)=\oo(1),\notag
		\end{align*}
		where in the last step we again used \eqref{eq:2.2}. This shows $\|\Lc-\Lb\| = \oop(1)$ and completes our proof.
	\end{proof}
	
	Recall $m_2$ as in \eqref{eq:mj}, $\varphi$ as in \eqref{eq:varphi}, $\mathcal{T}_n$ in \eqref{eq:2.2.5} and the multi-index notation in Section \ref{sec:prop}. The following lemma provides a key representation for $\mathcal{T}_n\varphi(\xb)$. This result is an essential component of the proof of Proposition \ref{prop:pointwise} in Section \ref{sec:prop}.
	\begin{lemma}
		\label{lemma:expansion}
		For $\varphi(\xb)$ as in \eqref{eq:varphi}, we have that
		\begin{align}
			\mathcal{T}_n \varphi(\xb)=\frac{2C_\varphi}{m_2}\sum_{\substack{0\le\bm{s}\le\bm{k}\\ \bm{s}\ne\bm{k}}}\sum_{\bm{\alpha}\in\N^d}\frac{(-1)^{|\bm{\alpha}|}}{\bm{\alpha}!2^{|\bm{\alpha}|}}\mathsf{C}_{\bm{k},\bm{s},\bm{\alpha}}\bm{\sigma}^{-\bm{\gamma}}r_n^{|\bm{\gamma}|-2}\Lambda_{\bm{\gamma}}(r_n;\bm{\sigma},g),\label{eq:B.2}
		\end{align}
		with $\Lambda_{\bm{\gamma}}(r_n;\bm{\sigma},g)$ in \eqref{eq:5.3.0.2} and $\mathsf{C}_{\bm{k},\bm{s},\bm{\alpha}}$ in \eqref{eq:5.3.0.3}.
	\end{lemma}
	\begin{proof}
		Recall $\mathcal{T}_n$ in \eqref{eq:2.2.5}. By a change of variable that $\zb=(\yb-\xb)/r_n$, \eqref{eq:2.2.5} becomes
		\begin{align}
			\mathcal{T}_n \varphi(\xb)=\frac{2C_\varphi}{m_2r_n^2}\int_{\normp{\zb}\le1}g(\normp{\zb})(\varphi(\xb)-\varphi(\xb+r_n\zb))\frac{\varrho(\xb+r_n\zb)}{\varrho(\xb)}\dd\zb.\label{eq:B.3}
		\end{align}
		Applying Lemma \ref{lemma:hermitediff} to $\varphi(\xb)-\varphi(\xb+r_n\zb)$ and Lemma \ref{lemma:hermitegenerating} to $\varrho(\xb+r_n\zb)/\varrho(\xb)$, we can rewrite
		\eqref{eq:B.3} as
		\begin{align}
			\mathcal{T}_n \varphi(\xb)=&\frac{2C_\varphi}{m_2r_n^2}\int_{\normp{\zb}\le1}-g(\normp{\zb})\sum_{\substack{0\le\bm{s}\le\bm{k}\\ \bm{s}\ne\bm{k}}}\left(\prod_{i=1}^d\binom{k_i}{s_i}\right)\mathsf{H}_{\bm{s}}\left(\frac{\xb}{\bm{\sigma}}\right)\left(\frac{2\zb r_n}{\bm{\sigma}}\right)^{\bm{k}-\bm{s}}\times\notag\\
			&\exp\left(-\sum_{i=1}^d\frac{z_i^2r_n^2}{4\sigma_i^2}\right)\sum_{\bm{\alpha}\in\N^d}\frac{1}{\bm{\alpha}!}\mathsf{H}_{\bm{\alpha}}\left(\frac{\xb}{\bm{\sigma}}\right)\left(\frac{-\zb r_n}{2\bm{\sigma}}\right)^{\bm{\alpha}}\dd\zb.\label{eq:B.1}
		\end{align}
		
		On the other hand, to justify swapping the integral and summation, we show that the series
		\begin{align}
			\sum_{\bm{\alpha}\in\N^d}\frac{1}{\bm{\alpha}!}\mathsf{H}_{\bm{\alpha}}\left(\frac{\xb}{\bm{\sigma}}\right)\left(\frac{-\zb r_n}{2\bm{\sigma}}\right)^{\bm{\alpha}},\notag
		\end{align}
		converges uniformly for all $\normp{\zb}\le1$. Since $\normp{\zb}\le1$ and $r_n=\oo(1)$, there exists $\rho>0$ such that for sufficiently large $n$, $|z_i r_n/(2\sigma_i)|<\rho<1/2$ for $i=1,\cdots,d$. Then, we have that for all $\normp{\zb}\le1$,
		\begin{align*}
			\sum_{\bm{\alpha}\in\N^d}\frac{1}{\bm{\alpha}!}|\mathsf{H}_{\bm{\alpha}}(\xb/\bm{\sigma})|(|\zb r_n|/(2\bm{\sigma}))^{\bm{\alpha}}&\le\prod_{i=1}^d\sum_{\alpha_i=0}^\infty\frac{|H_{\alpha_i}(x_i/\sigma_i)|}{\alpha_i!}\rho^{\alpha_i}\\
			&\le\prod_{i=1}^d\sum_{\alpha_i=0}^\infty\frac{\rho^{\alpha_i}}{\alpha_i!}2^{\alpha_i/2}\sqrt{\alpha_i!}\exp(x_i^2/(2\sigma_i^2))\\
			&\le\prod_{i=1}^d\exp(x_i^2/(2\sigma_i^2))\sum_{\alpha_i=0}^\infty\frac{1}{\sqrt{\alpha_i!}}<\infty,
		\end{align*}
		where in the second step we used (18.14.9) in \cite{olver2010nist} that $|H_m(x)|\le (2^m m!)^{1/2}\exp(x^2/2)$. This indicates that 
		$\sum_{\bm{\alpha}\in\N^d}\frac{1}{\bm{\alpha}!}\mathsf{H}_{\bm{\alpha}}\left(\frac{\xb}{\bm{\sigma}}\right)\left(\frac{-\zb r_n}{2\bm{\sigma}}\right)^{\bm{\alpha}}$
		converges uniformly for all $\normp{\zb}\le1$. Consequently, we can apply Fubini's theorem to switch the order of integral and summation in \eqref{eq:B.1} and conclude \eqref{eq:B.2}.
	\end{proof}
	
	Recall $\widetilde{\mathcal{T}}_n$ in \eqref{eq:defttn}, $\mathcal{T}_n$ in \eqref{eq:2.2.3} and $\varphi$ as in \eqref{eq:varphi}. In the following lemma, we finish the second step of the variance part in the proof of Proposition \ref{prop:pointwise} in Section \ref{sec:prop}.
	\begin{lemma}
		\label{lemma:ttn}
		Under the assumptions of Theorem \ref{theo:normald=1}, for $\varphi$ as in \eqref{eq:varphi}, $\normf{(\mathcal{T}_n-\widetilde{\mathcal{T}}_n)\varphi}=\oo(1)$.
	\end{lemma}
	\begin{proof}
		Recall $B_1$ in \eqref{eq:bulk}. By the triangle inequality ,
		$$\norm{(\mathcal{T}_n-\widetilde{\mathcal{T}}_n)\varphi}_{\cF}\le\norm{(\mathcal{T}_n-\widetilde{\mathcal{T}}_n)\varphi\bm{1}_{B_1}}_{\cF}+\norm{(\mathcal{T}_n-\widetilde{\mathcal{T}}_n)\varphi\bm{1}_{B_1^\complement}}_{\cF}.$$
		Therefore, the rest of the proof is to control $\norm{(\mathcal{T}_n-\widetilde{\mathcal{T}}_n)\varphi\bm{1}_{B_1}}_{\cF}$ and $\norm{(\mathcal{T}_n-\widetilde{\mathcal{T}}_n)\varphi\bm{1}_{B_1^\complement}}_{\cF}$ respectively.
		
		To control $\norm{(\mathcal{T}_n-\widetilde{\mathcal{T}}_n)\varphi\bm{1}_{B_1}}_{\cF}$, we have that
		\begin{align}
			&\norm{(\mathcal{T}_n-\widetilde{\mathcal{T}}_n)\varphi\bm{1}_{B_1}}_{\cF}\notag\\
			\asymp&r_n^{-2}\normbb{\int_{\normp{\xb-\yb}\le r_n}(\varphi(\xb)-\varphi(\yb))g\left(\tfrac{\normp{\xb-\yb}}{r_n}\right)\varrho(\yb)\dd\yb\left(\frac{1}{r_n^dm_0\varrho(\xb)}-\frac{1}{\int_{\normp{\xb-\yb}\le r_n}g\left(\tfrac{\normp{\xb-\yb}}{r_n}\right)\varrho(\yb)\dd\yb}\right)\bm{1}_{B_1}}_{\cF}\notag\\
			\lesssim&r_n^{-2}\normbb{\oo(1)\frac{\int_{\normp{\xb-\yb}\le r_n}(\varphi(\xb)-\varphi(\yb))g(\normp{\xb-\yb}/r_n)\varrho(\yb)\dd\yb}{r_n^dm_0\varrho(\xb)}\bm{1}_{B_1}}_{\cF}\notag\\
			\lesssim&\oo(1)\normf{\mathcal{T}_n\varphi}=\oo(1)(\normf{\Dr\varphi}+\oo(1))=\oo(1),\label{eq:B.3.1}
		\end{align}
		where in the second step we used Lemma \ref{lemma:interror} and in the last but one step we used \eqref{eq:5.3.0.10}.
		
		Next we proceed to control $\norm{(\mathcal{T}_n-\widetilde{\mathcal{T}}_n)\varphi\bm{1}_{B_1^\complement}}_{\cF}$. Since $\normf{\mathcal{T}_n\varphi}<\infty$, recalling \eqref{eq:bulk}, we have that 
		\begin{align}
			\normf{\mathcal{T}_n\varphi\bm{1}_{B_1^\complement}}=\oo(1).\label{B.3.2}
		\end{align}
		Recall $i_0$ in \eqref{eq:i0} and $\norm{\cdot}_{\Sigma}$ in \eqref{eq:sigmanorm}. On the other hand, uniformly for $\xb\notin B_1$ and $\yb$ such that $\normp{\xb-\yb}\le r_n$, we have that 
		\begin{align}
			|\varphi(\xb)|,|\varphi(\yb)|\lesssim\prod_{i=1}^{d}|x_{i_0}/\sigma_{i_0}|^{k_i}=\OO(\norm{\xb}_{\Sigma}^{|\bm{k}|}),\notag
		\end{align}
		where in the first step we used that $|y_i/\sigma_i|\le|x_i/\sigma_i|+\OO(r_n)\lesssim|x_{i_0}/\sigma_{i_0}|$.
		Therefore, we have for $\xb\notin B_1$,
		\begin{align}
			|\widetilde{\mathcal{T}}_n\varphi(\xb)|\le\frac{2m_0}{m_2r_n^2}\left(|\varphi(\xb)|+\max_{\yb:\normp{\xb-\yb}\le r_n}|\varphi(\yb)|\right)\lesssim r_n^{-2}\norm{\xb}_{\Sigma}^{|\bm{k}|}.\label{eq:B.3.3}
		\end{align}
		Consequently,
		\begin{align*}
			\normb{\widetilde{\mathcal{T}}_n\varphi\bm{1}_{B_1^\complement}}_\cF^2\lesssim& r_n^{-4}\int_{\xb\notin B_1}\norm{\xb}_{\Sigma}^{2|\bm{k}|}\varrho^2(\xb)\dd\xb\\
			\lesssim&r_n^{-4}\int_{\xb\notin B_1}\norm{\xb}_{\Sigma}^{2|\bm{k}|}\varrho(\xb)\exp\left(-\frac{x_{i_0}^2}{2\sigma_{i_0}^2}\right)\dd\xb\\
			\le&\frac{1}{n^{1+\beta}r_n^4}\int_{\R^d} \norm{\xb}_{\Sigma}^{2|\bm{k}|}\varrho(\xb)\dd\xb=\oo(1),
		\end{align*}
		where in the first step we used \eqref{eq:B.3.3} and in the last step we used \eqref{eq:2.2}. Combining with \eqref{eq:B.3.1}, we complete the proof.
	\end{proof}
	
	Recall $N_2$ in \eqref{eq:N1N2}, $D_2$ in \eqref{eq:D1D2}, $w_j$ in \eqref{eq:wj}, $v_j$ in \eqref{eq:vj} and $B_1$ as in \eqref{eq:bulk}. In the following lemma, we provide some useful properties about the above quantities, which will be used in the proof of the variance part of Proposition \ref{prop:pointwise} in Section \ref{sec:prop}.
	\begin{lemma}
		\label{lemma:case3}
		Under the assumptions of Theorem \ref{theo:normald=1}, if $\alpha_n\gg n^{7/2}$, then uniformly for $\xb\in B_1$, we have the following:
		\begin{flalign}
			\text{(i)}&\quad N_2 = \varrho(\xb)\,\OO\!\left(r_n^{d+1}\log^{(|\bm{k}|-1)/2} n\right), && \label{eq:B.4.0.1}\\
			\text{(ii)}&\quad \E w_j - N_2 = \varrho(\xb)\,\OO\!\left(\tfrac{r_n^{d-1}\log^{(|\bm{k}|-1)/2} n}{\alpha_n}\right), && \label{eq:B.4.0.2}\\
			\text{(iii)}&\quad \E w_j^2 = \varrho(\xb)\,\OO\!\left(r_n^{d+2}\log^{|\bm{k}|-1} n\right), && \label{eq:B.4.0.3}\\
			\text{(iv)}&\quad \E v_j - D_2 = \varrho(\xb)\,\OO\!\left(\tfrac{r_n^{d-2}}{\alpha_n}\right), && \label{eq:B.4.0.4}\\
			\text{(v)}&\quad \E v_j^2 = \varrho(\xb)\,\OO\!\left(r_n^d\right). && \label{eq:B.4.0.5}
		\end{flalign}
	\end{lemma}
	\begin{proof}
		We prove the five statements in order.
		\begin{enumerate}
			\item Recall $\varphi(\xb)=C_\varphi\prod_{i=1}^d H_{k_i}(x_i/\sigma_i)$. For $\xb\in B_1$,  $\normp{\xb-\yb}\le r_n$ and sufficiently large $n$, we have that $|x_i/\sigma_i|,|y_i/\sigma_i|\le\sqrt{2(1+2\beta)\log n}\lesssim\sqrt{\log n}$ for $i=1,\cdots,d$. Therefore,
			\begin{align}
				|\varphi(\xb)-\varphi(\yb)|&=|C_\varphi|\left|\sum_{i=1}^d\left(H_{k_i}\left(\frac{x_i}{\sigma_i}\right)-H_{k_i}\left(\frac{y_i}{\sigma_i}\right)\right)\prod_{j<i}H_{k_j}\left(\frac{x_j}{\sigma_j}\right)\prod_{j>i}H_{k_j}\left(\frac{y_j}{\sigma_j}\right)\right|\notag\\
				&\le|C_\varphi|\sum_{i=1}^d \frac{|x_i-y_i|}{\sigma_i}2k_i\max_{|t|\le\sqrt{2(1+2\beta)\log n}}|H_{k_i-1}(t)|\left|\prod_{j<i}H_{k_j}\left(\frac{x_j}{\sigma_j}\right)\prod_{j>i}H_{k_j}\left(\frac{y_j}{\sigma_j}\right)\right|\notag\\
				&\lesssim r_n \log^{(|\bm{k}|-1)/2} n,\label{eq:B.4.1}
			\end{align}
			where in the first step we used \eqref{eq:prod} and in the second step we used mean value theorem and \eqref{eq:C3.2.1}. Then we have that
			\begin{align*}
				|N_2|&=\left|\int_{\normp{\xb-\yb}\le r_n}(\varphi(\xb)-\varphi(\yb))g(\normp{\xb-\yb}/r_n)\varrho(\yb)\dd\yb\right|\\
				&\lesssim r_n \log^{(|\bm{k}|-1)/2} n\int_{\normp{\xb-\yb}\le r_n}\varrho(\yb)\dd\yb\\
				&\lesssim r_n^{d+1}\log^{(|\bm{k}|-1)/2}\varrho(\xb).
			\end{align*}
			where in the second step we used \eqref{eq:B.4.1} and the boundedness of $g(\cdot)$ and in the last step we used Lemma \ref{lemma:interror}. This concludes \eqref{eq:B.4.0.1}
			
			\item We split $\R^d$ into the following regions based on $\xb$,
			\begin{align}
				R_1\coloneq\{\yb:\normp{\xb-\yb}\le r_n\},\ R_2\coloneq\{\yb:r_n<\normp{\xb-\yb}\le 2r_n\},\text{ and }R_3\coloneq\{\yb:\normp{\xb-\yb}> 2r_n\}.\label{eq:xysplit}
			\end{align}
			Now we discuss the difference on these three regions. For $R_1$,
			\begin{align}
				&\left|N_2-\int_{R_1}\mathsf{s}(\alpha_n(r_n^2-\normp{\xb-\yb}^2))g^*(\normp{\xb-\yb}/r_n)(\varphi(\xb)-\varphi(\yb))\varrho(\yb)\dd\yb\right|\notag\\
				=&\left|\int_{R_1}(1-\mathsf{s}(\alpha_n(r_n^2-\normp{\xb-\yb}^2)))g(\normp{\xb-\yb}/r_n)(\varphi(\xb)-\varphi(\yb))\varrho(\yb)\dd\yb\right|\notag\\
				\lesssim&r_n\log^{(|\bm{k}|-1)/2}n\int_{R_1}\frac{1}{1+\exp(\alpha_n|r_n^2-\normp{\xb-\yb}^2|)}\varrho(\yb)\dd\yb\notag\\
				\lesssim&r_n\log^{(|\bm{k}|-1)/2}n\varrho(\xb)\int_{\normp{\xb-\yb}\le r_n}\exp(\alpha_n(\normp{\xb-\yb}^2-r_n^2))\dd\yb\notag\\
				\asymp&r_n\log^{(|\bm{k}|-1)/2}n\varrho(\xb)\int_{0}^{r_n}t^{d-1}\exp(\alpha_n(t^2-r_n^2))\dd t\notag\\
				\le&r_n\log^{(|\bm{k}|-1)/2}n\varrho(\xb)r_n^{d-1}\int_{0}^{r_n}\exp(-\alpha_nr_n(r_n-t))\dd t\notag\\
				=&r_n\log^{(|\bm{k}|-1)/2}n\varrho(\xb)r_n^{d-1}(1-e^{-\alpha_nr_n^2})/(r_n\alpha_n)\notag\\
				\le&\varrho(\xb)r_n^{d-1}\log^{(|\bm{k}|-1)/2}n/\alpha_n,\label{eq:B.4.2}
			\end{align}
			where in the first step we used our assumption that $g(t)=g^*(t)$ for $t\in[0,1]$, in the second step we used Lemma \ref{lemma:sig2ind} and \eqref{eq:B.4.1}, the third step is from \eqref{eq:denstiyratio}, in the fourth step we applied a change of variable with $t=\normp{\xb-\yb}$. Next, for $R_2$, similar to \eqref{eq:B.4.2}, we have that
			\begin{align}
				&\left|\int_{R_2}\mathsf{s}(\alpha_n(r_n^2-\normp{\xb-\yb}^2))g^*(\normp{\xb-\yb}/r_n)(\varphi(\xb)-\varphi(\yb))\varrho(\yb)\dd\yb\right|\notag\\
				\lesssim&r_n\log^{(|\bm{k}|-1)/2}n\int_{R_2}\exp(-\alpha_n(\normp{\xb-\yb}^2-r_n^2))\varrho(\yb)\dd\yb\notag\\
				\lesssim&r_n\log^{(|\bm{k}|-1)/2}n\varrho(\xb)\int_{r_n<\normp{\xb-\yb}\le 2r_n}\exp(-\alpha_n(\normp{\xb-\yb}^2-r_n^2))\dd\yb\notag\\
				\asymp&r_n\log^{(|\bm{k}|-1)/2}n\varrho(\xb)\int_{r_n}^{2r_n}t^{d-1}\exp(-\alpha_n(t^2-r_n^2))\dd t\notag\\
				\lesssim&\alpha_n^{-1}r_n^{d-1}\log^{(|\bm{k}|-1)/2}n\varrho(\xb).\label{eq:B.4.3}
			\end{align}
			Moreover, for $R_3$, we have that 
			\begin{align}
				&\left|\int_{R_3}\mathsf{s}(\alpha_n(r_n^2-\normp{\xb-\yb}^2))g^*(\normp{\xb-\yb}/r_n)(\varphi(\xb)-\varphi(\yb))\varrho(\yb)\dd\yb\right|\notag\\
				\lesssim&r_n\log^{(|\bm{k}|-1)/2}n\int_{\normp{\xb-\yb}>2r_n}\exp(-\alpha_n(\normp{\xb-\yb}^2-r_n^2))\varrho(\yb)\dd\yb\notag\\
				\lesssim&r_n\log^{(|\bm{k}|-1)/2}n\int_{\normp{\xb-\yb}>2r_n}\exp(-3\alpha_nr_n^2)\varrho(\yb)\dd\yb\notag\\
				\lesssim&r_n\log^{(|\bm{k}|-1)/2}n\exp(-3\alpha_nr_n^2),\label{eq:B.4.4}
			\end{align}
			where in the first step we used \eqref{eq:B.4.1}. Combining \eqref{eq:B.4.2}, \eqref{eq:B.4.3} and \eqref{eq:B.4.4}, we have that 
			\begin{align*}
				|\E w_j-N_2|&=\varrho(\xb)\OO(r_n^{d-1}\log^{(|\bm{k}|-1)/2}n/\alpha_n)+\OO(r_n\log^{(|\bm{k}|-1)/2}n\exp(-3\alpha_nr_n^2))\\
				&=\varrho(\xb)\OO(r_n^{d-1}\log^{(|\bm{k}|-1)/2}n/\alpha_n),
			\end{align*}
			where in the last step we used the assumption that $\alpha\gg n^{7/2}$ and \eqref{eq:2.2}. This concludes \eqref{eq:B.4.0.2}.
			
			\item Recall our split of $\R^d$ in \eqref{eq:xysplit}. For $R_1$, similar to \eqref{eq:B.4.2}, we have that
			\begin{align}
				&\left|\int_{R_1}(1-\mathsf{s}^2(\alpha_n(r_n^2-\normp{\xb-\yb}^2)))(g^*(\normp{\xb-\yb}/r_n))^2(\varphi(\xb)-\varphi(\yb))^2\varrho(\yb)\dd\yb\right|\notag\\
				\le&\left|\int_{R_1}2(1-\mathsf{s}(\alpha_n(r_n^2-\normp{\xb-\yb}^2)))(g^*(\normp{\xb-\yb}/r_n))^2(\varphi(\xb)-\varphi(\yb))^2\varrho(\yb)\dd\yb\right|\notag\\
				\lesssim&r_n^2\log^{|\bm{k}|-1}n\int_{R_1}\frac{1}{1+\exp(\alpha_n|r_n^2-\normp{\xb-\yb}^2|)}\varrho(\yb)\dd\yb\notag\\
				\lesssim&r_n^2\log^{|\bm{k}|-1}n\varrho(\xb)r_n^{d-2}/\alpha_n=\varrho(\xb)\OO(r_n^d\log^{|\bm{k}|-1}/\alpha_n).\label{eq:B.4.5}
			\end{align}
			Next, similar to \eqref{eq:B.4.3} and \eqref{eq:B.4.4}, we have that
			\begin{align*}
				&\left|\int_{R_2}\mathsf{s}^2(\alpha_n(r_n^2-\normp{\xb-\yb}^2)))(g^*(\normp{\xb-\yb}/r_n))^2(\varphi(\xb)-\varphi(\yb))^2\varrho(\yb)\dd\yb\right|\\
				\le&\left|\int_{R_2}\mathsf{s}(\alpha_n(r_n^2-\normp{\xb-\yb}^2)))(g^*(\normp{\xb-\yb}/r_n))^2(\varphi(\xb)-\varphi(\yb))^2\varrho(\yb)\dd\yb\right|=\varrho(\xb)\OO(r_n^d\log^{|\bm{k}|-1}/\alpha_n),\\
			\end{align*}
			and
			\begin{align*}
				&\left|\int_{R_3}\mathsf{s}^2(\alpha_n(r_n^2-\normp{\xb-\yb}^2)))(g^*(\normp{\xb-\yb}/r_n))^2(\varphi(\xb)-\varphi(\yb))^2\varrho(\yb)\dd\yb\right|\\
				\le&\left|\int_{R_3}\mathsf{s}(\alpha_n(r_n^2-\normp{\xb-\yb}^2)))(g^*(\normp{\xb-\yb}/r_n))^2(\varphi(\xb)-\varphi(\yb))^2\varrho(\yb)\dd\yb\right|=\OO(r_n^2\log^{|\bm{k}|-1}\exp(-6\alpha_nr_n^2)).
			\end{align*}
			Combining with \eqref{eq:B.4.5}, we have that
			\begin{align*}
				\E w_j^2&\le\int_{\normp{\xb-\yb}\le r_n}(g^*(\normp{\xb-\yb}/r_n))^2(\varphi(\xb)-\varphi(\yb))^2\varrho(\yb)\dd\yb+\varrho(\xb)\OO(r_n^d\log^{|\bm{k}|-1}n/\alpha_n)\\
				&\lesssim \varrho(\xb)r_n^2\log^{|\bm{k}|-1}n\int_{\normp{\xb-\yb}\le r_n}
				\dd\yb+\varrho(\xb)\OO(r_n^d\log^{|\bm{k}|-1}n/\alpha_n)\\
				&=\varrho(\xb)\OO(r_n^{2+d}\log^{|\bm{k}|-1}n),
			\end{align*}
			where in the second step we used \eqref{eq:B.4.1}.
			This completes \eqref{eq:B.4.0.3}.
			
			\item Similar to the proof of \eqref{eq:B.4.0.2}, substituting $\varphi(\xb)-\varphi(\yb)$ with 1, we can conclude \eqref{eq:B.4.0.4}.
			
			\item Similar to the proof of \eqref{eq:B.4.0.3}, substituting $\varphi(\xb)-\varphi(\yb)$ with 1, we can conclude \eqref{eq:B.4.0.5}.
		\end{enumerate}
	\end{proof}
	
	Recall $\Lc$ in \eqref{eq:2.2.2} and $\sLn$ in \eqref{eq:2.2.3}. In the following lemma, we link the eigenvalues of $\Lc$ with $\sLn$, which will be used in the proof of \eqref{eq:main} in Section \ref{sec:mainproof}.
	\begin{lemma}
		\label{lemma:S2Sn}
		$\Lc$ and $\sLn$ have the same eigenvalues counting multiplicities except at $\frac{2m_0}{m_2r_n^2}$.
	\end{lemma}
	The proof is modified from the proof of Proposition 9 in \cite{von2008consistency}.
	\begin{proof}
		Suppose $\lambda\ne\frac{2m_0}{m_2r_n^2}$ is an eigenvalue of $\sLn$ with corresponding eigenfunction $\varphi\not\equiv0$. We have that 
		\begin{align}
			\sLn\varphi(\xb)=\lambda\varphi(\xb),\text{ for }x\in\R^d.\label{eq:5.6.1}
		\end{align}
		Evaluating \eqref{eq:5.6.1} at $\xb_1,\xb_2,\cdots,\xb_n$ and recalling $\Ka$ in \eqref{eq:2.2.1}, we have
		\begin{align}
			\frac{2m_0}{m_2r_n^2}\frac{\frac{1}{n}\sum_{j=1}^n\Ka(\xb_i,\xb_j)(\varphi(\xb_i)-\varphi(\xb_j))}{\frac{1}{n}\sum_{j=1}^n\Ka(\xb_i,\xb_j)}=\lambda\varphi(\xb_i),\text{ for }i=1,\cdots,n.\label{eq:5.6.2}
		\end{align}
		Let $\rho\circ\varphi\coloneq(\varphi(\xb_1),\cdots,\varphi(\xb_n))^\top$. Then \eqref{eq:5.6.2} becomes
		\begin{align}
			\frac{2m_0}{m_2r_n^2}(\Ib-\Da^{-1}\Ka)(\rho\circ\varphi)=\lambda(\rho\circ\varphi).\label{eq:5.6.7}
		\end{align}
		Moreover, \eqref{eq:5.6.1} can be rewritten as
		\begin{align}
			\varphi(\xb)=\frac{1}{1-\frac{m_2r_n^2}{2m_0}\lambda}\frac{\frac{1}{n}\sum_{j=1}^n\Ka(\xb,\xb_j)\varphi(\xb_j)}{\frac{1}{n}\sum_{j=1}^n\Ka(\xb,\xb_j)}.\label{eq:5.6.6}
		\end{align}
		This indicates that $\varphi$ is uniquely determined by $\rho\circ\varphi$. Since $\varphi\not\equiv0$, \eqref{eq:5.6.6} also implies that $\rho\circ\varphi\ne\bm{0}$. Combining with \eqref{eq:5.6.7}, we have that $\lambda$ is an eigenvalue of $\Lc$ with eigenvector $\rho\circ\varphi$.
		Furthermore, let $\varphi_1,\cdots,\varphi_s$ be linearly independent eigenvectors of $\sLn$ corresponding to $\lambda$. We show the linear independence of $\rho\circ\varphi_1,\cdots,\rho\circ\varphi_s$ by contradiction. Assume there are some constants $c_1,\dots,c_l$ such that $\sum_{t=1}^l c_t(\rho\circ\varphi_t)=0$ and at least one $c_t\ne0$. Combining with \eqref{eq:5.6.6}, we have that $\sum_{t=1}^l c_t\varphi_t(\xb)=0$ for all $\xb\in\R^d$, which contradicts the linear independence of $\varphi_1,\cdots,\varphi_s$. Therefore, $\rho\circ\varphi_1,\cdots,\rho\circ\varphi_s$ are linearly independent eigenfunctions of $\Lc$ corresponding to $\lambda$. Consequently, the multiplicity of $\lambda$ of $\Lc$ is larger than or equal to that of $\sLn$.
		
		On the other hand, suppose $\lambda^*\ne\frac{2m_0}{m_2r_n^2}$ is an eigenvalue of $\Lc$ with eigenvector $\bm{v}=(v_1,\cdots,v_n)^\top$. We have that 
		\begin{align}
			\frac{2m_0}{m_2r_n^2}\left(v_i-\frac{\frac{1}{n}\sum_{j=1}^n\Ka(\xb_i,\xb_j)v_j}{\frac{1}{n}\sum_{j=1}^n\Ka(\xb_i,\xb_j)}\right)=\lambda^*v_i,\quad i=1,\cdots,n.\label{eq:5.6.5}
		\end{align}
		Denote 
		\begin{align}
			\varphi^*(\xb)\coloneq\frac{1}{1-\frac{m_2r_n^2}{2m_0}\lambda^*}\frac{\frac{1}{n}\sum_{j=1}^n\Ka(\xb,\xb_j)v_j}{\frac{1}{n}\sum_{j=1}^n\Ka(\xb,\xb_j)}.\label{eq:5.6.4}
		\end{align}
		Combining with \eqref{eq:5.6.5}, we have that for $i=1,\cdots,n$, $\varphi^*(\xb_i)=v_i$. Substituting $v_i$ with $\varphi^*(\xb_i)$ in \eqref{eq:5.6.4}, we have that
		\begin{align*}
			\frac{2m_0}{m_2r_n^2}\left(\varphi^*(\xb)-\frac{\frac{1}{n}\sum_{j=1}^n\Ka(\xb,\xb_j)\varphi^*(\xb_j)}{\frac{1}{n}\sum_{j=1}^n\Ka(\xb,\xb_j)}\right)=\lambda^*\varphi^*(\xb),\quad \xb\in \R^d,
		\end{align*}
		which indicates that $\lambda^*$ is an eigenvalue of $\sLn$ with eigenfunction $\varphi^*$ in \eqref{eq:5.6.4}. Moreover, if $\bm{v}_1,\bm{v}_2,\cdots,\bm{v}_l$ are linearly independent eigenvectors of $\Lc$ corresponding to $\lambda^*$, the eigenfunctions $\varphi^*_1,\cdots,\varphi^*_l$ constructed in the form of \eqref{eq:5.6.4} are also linearly independent. Consequently, the multiplicity of $\lambda^*$ of $\sLn$ is larger than or equal to that of $\Lc$.

		Combining all above discussion, we conclude that $\Lc$ and $\sLn$ have the same multiplicities for all eigenvalues not equal to $\frac{2m_0}{m_2r_n^2}$. This completes our proof.
	\end{proof}

	\section{Auxiliary lemmas}
	\paragraph*{Matrix preliminaries.}
	Let $\mathbf{A}=(a_{ij})$ be an $n\times n$ complex matrix and $\norm{\cdot}$ be the operator norm. Then,
	\begin{align}
		\norm{\mathbf{A}}\le\sqrt{n}\max_{1\le i\le n}\sum_{j=1}^n|a_{ij}|.\label{eq:matrixnorm}
	\end{align}
	Denote the eigenvalues of $\mathbf{A}$ by $\{\lambda_i\}_{i=1}^n$. Then we have from Section 2.2 in \cite{gray2006toeplitz} that
	\begin{align}
		\frac{1}{n}\sum_{i=1}^n|\lambda_i|^2\le\frac{1}{n}\sum_{i=1}^n\sum_{j=1}^n|a_{ij}|^2.\label{eq:HSnorm}
	\end{align}
	
	\paragraph*{Resolvent identity.}
	Recall that for a closed linear operator $\mathcal{K}$, $R(z,\mathcal{K})=(\mathcal{K}-z\mathcal{I})^{-1}$ for $z\notin\operatorname{spec}(\mathcal{K})$ is denoted as its resolvent. Let $\mathcal{K}_1,\mathcal{K}_2$ be closed linear operators, $z\in\C$ and $z\notin\operatorname{spec}(\mathcal{K}_1)\cup\operatorname{spec}(\mathcal{K}_2)$. Then
	\begin{align}
		R(z,\mathcal{K}_1)-R(z,\mathcal{K}_2)=R(z,\mathcal{K}_1)(\mathcal{K}_2-\mathcal{K}_1)R(z,\mathcal{K}_2).\label{eq:resolvent}
	\end{align}
	See Section 1.5 in \cite{de2009intermediate} for details.
	
	\paragraph*{Concentration inequality.}
	In the following lemma, we introduce a concentration inequality that controls how far the sum of independent Bernoulli random variables can deviate from its mean.
	\begin{lemma}
		\label{lemma:chernoff}
		Let $X_i\stackrel{\operatorname{i.i.d.}}{\sim}\operatorname{Ber}(p)$ for $i=1,\cdots,n$. Then for any $t\in(0,1)$,
		\begin{align*}
			\P\left(\left|\sum_{i=1}^n X_i-np\right|\ge tnp\right)\le2\exp\left(-\frac{t^2np}{3}\right).
		\end{align*}
	\end{lemma}
	\begin{proof}
		See Corollary 4.6 in \cite{mitzenmacher2017probability}.
	\end{proof}
	
	\paragraph*{Properties about physicist's Hermite polynomials.}
	In the following, we introduce some properties about physicist's Hermite polynomials that will be used in the proof of Proposition \ref{prop:pointwise}. See \cite{szeg1939orthogonal,olver2010nist} for more details. From Section 5.5 in \cite{szeg1939orthogonal}, for physicist's Hermite polynomial $H_m(x)$,
	\begin{align}
		H_m'(x)=2mH_{m-1}(x).\label{eq:C3.2.1}
	\end{align}
	Recall the multi-index notation in Section \ref{sec:prop}. In the following two lemmas, we show that both 
	$$\mathsf{H}_{\bm{k}}((\xb+\zb)/\bm{\sigma})-\mathsf{H}_{\bm{k}}(\xb/\bm{\sigma}),\quad \text{and}\quad\frac{\varrho(\xb+\zb)}{\varrho(\xb)},$$
	can be expanded into series in terms of physicist's Hermite polynomials. These two lemmas are crucial for Lemma \ref{lemma:expansion}.
	\begin{lemma}
		\label{lemma:hermitediff}
		Let $\mathsf{H}_{\bm{k}}(\xb/\bm{\sigma})=\prod_{i=1}^d H_{k_i}(x_i/\sigma_i)$. For any $\xb,\zb\in\R^d$,
		$$\mathsf{H}_{\bm{k}}((\xb+\zb)/\bm{\sigma})-\mathsf{H}_{\bm{k}}(\xb/\bm{\sigma})=\sum_{\substack{\bm{s}\le\bm{k},\\ \bm{s}\ne\bm{k}}}\prod_{i=1}^{d}\binom{k_i}{s_i}H_{s_i}(x_i/\sigma_i)(2z_i/\sigma_i)^{k_i-s_i}.$$
	\end{lemma}
	\begin{proof} 
		Using Taylor's expansion and \eqref{eq:C3.2.1}, we have that for $x,z\in\R$,
		$$H_m(x+z)-H_m(x)=\sum_{j=1}^{m}\frac{z^j}{j!}H^{(j)}_m(x).$$
		Then, we have that 
		\begin{align}
			H_m(x+z)=\sum_{j=0}^m \frac{z^j}{j!}H^{(j)}_m(x)=\sum_{j=0}^m\binom{m}{j}H_j(x)(2z)^{m-j},\label{eq:C3.3}
		\end{align}
		where in the last step we used \eqref{eq:C3.2.1}. Recall that $\mathsf{H}_{\bm{k}}(\xb/\bm{\sigma})=\prod_{i=1}^d H_{k_i}(x_i/\sigma_i)$. Therefore,
		\begin{align}
			\mathsf{H}_{\bm{k}}((\xb+\zb)/\bm{\sigma})=\sum_{\substack{\bm{s}\le \bm{k}}}\prod_{i=1}^{d}\binom{k_i}{s_i}H_{s_i}(x_i/\sigma_i)(2z_i/\sigma_i)^{k_i-s_i},\notag
		\end{align}
		which completes the proof.
	\end{proof}
	Recall $\varrho(\xb)$ in \eqref{eq:2.1.4}. The following lemma shows that $\varrho$ could be expanded into an infinite series in terms of physicist's Hermite polynomials.
	\begin{lemma}
		\label{lemma:hermitegenerating}
		For $\xb,\zb\in\R^d$, we have that
		\begin{align}
			\frac{\varrho(\xb+\zb)}{\varrho(\xb)}=\exp\left(-\sum_{i=1}^d\frac{z_i^2}{4\sigma_i^2}\right)\sum_{\alpha_1,\cdots,\alpha_d\in\N}\prod_{i=1}^{d}\frac{H_{\alpha_i}(x_i/\sigma_i)}{\alpha_i!}\left(-\frac{z_i}{2\sigma_i}\right)^{\alpha_i}.\label{eq:C3.1}
		\end{align}
	\end{lemma}
	\begin{proof}
		From the generating function of Hermite polynomials \cite[Section 18.12]{olver2010nist}, we have that
		\begin{align}
			\frac{\exp(-(x_i+z_i)^2/(2\sigma_i^2))}{\exp(-x_i^2/(2\sigma_i^2))}=\exp\left(-\frac{z_i^2}{4\sigma_i^2}\right)\sum_{\alpha_i=0}^{\infty}H_{\alpha_i}(x_i/\sigma_i)\frac{(-z_i/(2\sigma_i))^{\alpha_i}}{\alpha_i!}.\label{eq:C3.2}
		\end{align}
		Recall that $\varrho(\xb)=\prod_{i=1}^d\exp(-x_i^2/(2\sigma_i^2))$. Multiplying \eqref{eq:C3.2} over $i=1,\cdots,d$, we conclude \eqref{eq:C3.1}.
	\end{proof}
	Recall $\norm{\cdot}_{\cF_i}$ in \eqref{eq:4.9.1}. In the following lemma, without loss of generality, we assume $\sigma_i=1$ with $\sigma_i$ as in \eqref{eq:4.9}.
	\begin{lemma}
		\label{lemma:hermitenielsen}
		For physicist's Hermite polynomials $H_k(x)$ and $H_j(x)$, we have 
		\begin{align*}
			\norm{H_kH_j}_{\cF_i}\le(\pi)^{1/4}\sqrt{(k+1)!j!}\left(2\sqrt{2}\right)^{k+j}.
		\end{align*}
	\end{lemma}
	\begin{proof}
		We have that
		\begin{align*}
			\norm{H_kH_j}^2_{\cF_i}&=\normbb{\sum_{r=0}^{\min(k,j)}2^rr!\binom{k}{r}\binom{j}{r}H_{k+j-2r}(x)}^2_{\cF_i}\notag\\
			&=\sum_{r=0}^{\min(k,j)}\left(2^rr!\binom{k}{r}\binom{j}{r}\right)^2(k+j-2r)!2^{k+j-2r}\sqrt{\pi}\notag\\
			&=\sqrt{\pi}k!j!\sum_{r=0}^{\min(k,j)}\left(\binom{k}{r}\binom{j}{r}\right)^2\frac{r!r!(k+j-2r)!}{(k+j)!}\frac{(k+j)!}{k!j!}\notag\\
			&\le\sqrt{\pi}k!j!\sum_{r=0}^{\min(k,j)}(2^{k+j})^2 2^{k+j}\le\sqrt{\pi}(k+1)!j!8^{k+j}\notag,
		\end{align*}
		where in the first step we used (18.18.23) in \cite{olver2010nist}, in the second step we used the orthogonality of $H_m$'s and in the fourth step we used the fact that $\binom{a}{b}\le 2^a$ for $a,b\in\N$ and $a\ge b$.
	\end{proof}
	
	\paragraph*{Some facts and lemmas used in Section \ref{sec:prop}.}
	Now we introduce some useful facts and lemmas which were used in the proof of Proposition \ref{prop:pointwise} in Section \ref{sec:prop}. Here we first introduce the telescoping product identity,
	\begin{align}
		\prod_{i=1}^da_i-\prod_{i=1}^d b_i=\sum_{i=1}^{d}(a_i-b_i)\prod_{j<i}a_j\prod_{j>i}b_j.\label{eq:prod}
	\end{align}
	In the following lemma, we provide a control over the coordinates of $\xb_1,\cdots,\xb_n$, which enables us to employ our truncation argument in Section \ref{sec:prop}. 
	\begin{lemma}
		\label{lemma:xj}
		Let $\xb_j=(x_{j1},\cdots,x_{jd})^\top\stackrel{\operatorname{i.i.d.}}{\sim}\mathcal{N}_d(\bm{0},\operatorname{diag}(\sigma_1^2,\cdots,\sigma_d^2))$, $j=1,\cdots,n$, with $\sigma_i>0$ for $i=1,\cdots,d$. Then, with probability $1-\oo(1)$, for all $j=1,\cdots,n$ and $i=1,\cdots,d$,
		\begin{align*}
			|x_{ji}/\sigma_i|\le\sqrt{2\log n}.
		\end{align*}
	\end{lemma}
	\begin{proof}
		We have that 
		\begin{align*}
			\P\left(|x_{ji}/\sigma_i|\le\sqrt{2\log n},\forall i,j\right)&=1-\P\left(\exists i,j\text{ such that }|x_{ji}/\sigma_i|>\sqrt{2\log n}\right)\\
			&\ge 1-\sum_{i=1}^d\sum_{j=1}^n\P\left(|x_{ji}/\sigma_i|>\sqrt{2\log n}\right)\\
			&=1-nd\P\left(|Z|>\sqrt{2\log n}\right)\\
			&\ge1-nd\frac{2\exp(-\log n)}{\sqrt{2\log n}\sqrt{2\pi}}=1-\OO\left(\frac{1}{\sqrt{\log n}}\right),
		\end{align*}
		where $Z\sim\mathcal{N}(0,1)$, in the third step we used the fact that $x_{ji}/\sigma_i\sim\mathcal{N}(0,1)$ and in the last but one step we used the Gaussian tail bound \cite[Section 7.1]{feller1991introduction} that for $t>0$,
		$$\P(|Z|>t)\le\frac{2\exp(-t^2/2)}{t\sqrt{2\pi}}.$$
		This completes our proof.
	\end{proof}
	In the following lemma, we offer a control between the 0-1 kernel and sigmoid-type kernel, which will be used in the proof of Lemma \ref{lemma:S2L}. Recall $\mathsf{s}(\cdot)$ in \eqref{eq:2.2.1}.
	\begin{lemma}
		\label{lemma:sig2ind}
		For any $\alpha_n,r_n>0$, $g(\cdot)$ as in the assumptions of Theorem \ref{theo:normald=1} and $g^*(\cdot)$ as in \eqref{eq:2.2.1},
		\begin{align}
			\left|g(\normp{\xb-\yb}/r_n)\bm{1}(\normp{\xb-\yb}\le r_n)-g^*(\normp{\xb-\yb}/r_n)\mathsf{s}\left(\alpha_n(r_n^2-\normp{\xb-\yb}^2)\right)\right|\le\frac{\max_{t\ge 0}g^*(t)}{1+\exp(\alpha_n|r_n^2-\normp{\xb-\yb}^2|)}.\label{eq:A.6.1}
		\end{align}
	\end{lemma}
	\begin{proof}
		We split our proof into two cases $\normp{\xb-\yb}\le r_n$ and $\normp{\xb-\yb}>r_n$.
		\par\vspace{.5\baselineskip}
		\noindent{\bf{Case 1.}}\; When $\normp{\xb-\yb}\le r_n$, since $\bm{1}(\normp{\xb-\yb}\le r_n)=1$ and $g(\normp{\xb-\yb}/r_n)=g^*(\normp{\xb-\yb}/r_n)$, we have that
		\begin{align*}
			&\left|g(\normp{\xb-\yb}/r_n)\bm{1}(\normp{\xb-\yb}\le r_n)-g^*(\normp{\xb-\yb}/r_n)\mathsf{s}\left(\alpha_n(r_n^2-\normp{\xb-\yb}^2)\right)\right|\\
			=&\left|g^*(\normp{\xb-\yb}/r_n)-\frac{g^*(\normp{\xb-\yb}/r_n)}{1+\exp\left(-\alpha_n(r_n^2-\normp{\xb-\yb}^2)\right)}\right|\\
			=&\frac{g^*(\normp{\xb-\yb}/r_n)\exp\left(-\alpha_n(r_n^2-\normp{\xb-\yb}^2)\right)}{1+\exp\left(-\alpha_n(r_n^2-\normp{\xb-\yb}^2)\right)}\le\frac{\max_{t\ge 0}g^*(t)}{1+\exp(\alpha_n|r_n^2-\normp{\xb-\yb}^2|)}.
		\end{align*}
		\par\vspace{.5\baselineskip}
		\noindent{\bf{Case 2.}}\; When $\normp{\xb-\yb}> r_n$, since $\bm{1}(\normp{\xb-\yb}\le r_n)=0$, we have that
		\begin{align*}
			&\left|g(\normp{\xb-\yb}/r_n)\bm{1}(\normp{\xb-\yb}\le r_n)-g^*(\normp{\xb-\yb}/r_n)\mathsf{s}\left(\alpha_n(r_n^2-\normp{\xb-\yb}^2)\right)\right|\\
			=&\frac{g^*(\normp{\xb-\yb}/r_n)}{1+\exp\left(-\alpha_n(r_n^2-\normp{\xb-\yb}^2)\right)}
			\le\frac{\max_{t\ge 0}g^*(t)}{1+\exp(\alpha_n|r_n^2-\normp{\xb-\yb}^2|)}.
		\end{align*}
		This completes our proof.
	\end{proof}
	Recall $B_1$ as in \eqref{eq:bulk}. In the following lemma, we provide a uniform approximation of
	$$\int_{\normp{\xb-\yb}\le r_n}g(\normp{\xb-\yb}/r_n)\varrho(\yb)\dd\yb.$$
	\begin{lemma}
		\label{lemma:interror}
		Uniformly for any $\xb\in B_1$, 
		$$\int_{\normp{\xb-\yb}\le r_n}g(\normp{\xb-\yb}/r_n)\varrho(\yb)\dd\yb=m_0r_n^d\varrho(\xb)(1+\oo(1))$$
	\end{lemma}
	\begin{proof}
		Applying a change of variable that $\yb=\xb+r_n\zb$ gives
		\begin{align}
			\int_{\normp{\xb-\yb}\le r_n}g(\normp{\xb-\yb}/r_n)\varrho(\yb)\dd\yb=r_n^d\int_{\normp{\zb}\le1}g(\normp{\zb})\varrho(\xb+r_n\zb)\dd\zb.\label{eq:C.3.7.1}
		\end{align}
		Moreover, uniformly for all $\xb\in B_1$,
		\begin{align*}
			\varrho(\xb+r_n\zb)/\varrho(\xb)&=\exp\left(-\sum_{i=1}^d\frac{1}{2\sigma_i^2}((x_i+r_nz_i)^2-x_i^2)\right)\\
			&=\exp(-\sum_{i=1}^d |x_i|\OO(r_n))=\exp(\oo(1))=1+\oo(1),
		\end{align*}
		where in the third step we used \eqref{eq:bulk}. Note this also indicates that uniformly for $\xb$ such that $|x_i|\le C\log n$ for $i=1,\cdots,d$ and any large fixed constant $C>0$,
		\begin{align}
			\varrho(\xb+r_n\zb)=\varrho(\xb)(1+\oo(1)).\label{eq:denstiyratio}
		\end{align}
		Combining with \eqref{eq:C.3.7.1}, we have that
		\begin{align*}
			\int_{\normp{\xb-\yb}\le r_n}g(\normp{\xb-\yb}/r_n)\varrho(\yb)\dd\yb=r_n^dm_0\varrho(\xb)(1+\oo(1)).
		\end{align*}
		This completes the proof.
	\end{proof}

	\bibliographystyle{abbrv}
	\bibliographystyle{plain}
	\bibliography{references,aniso_bib,sensornonnull,ref_YH}
	
\end{document}